\documentclass{amsart}
\usepackage{amsmath,amscd,amssymb,amsthm,amsbsy,amscd,stmaryrd}
\usepackage{mathrsfs}
\input{xy}

\usepackage{pstricks}
\usepackage{color}
\usepackage{graphicx}
\usepackage[all]{xy}
\usepackage {hyperref}

\newtheorem{theorem}{Theorem}[section]

\newtheorem{proposition}[theorem]{Proposition}
\newtheorem{corollary}[theorem]{Corollary}
\newtheorem{remark}[theorem]{Remark}
\newtheorem{question}[theorem]{Question}

\theoremstyle{definition}

\numberwithin{equation}{section}

\begin{document}
\title{Mabuchi and Aubin-Yau functionals over complex manifolds}
\author{Yi Li}
\address{Department of Mathematics, Harvard University, Cambridge, MA 02138}

\email{yili@math.harvard.edu}

\begin{abstract} In the previous papers \cite{L1, L2} the author
constructed Mabuchi and Aubin-Yau functionals over any complex
surfaces and three-folds, respectively. Using the method in
\cite{L2}, we construct those functionals over any complex manifolds
of the complex dimension bigger than or equal to $2$.
\end{abstract}
\maketitle

\tableofcontents
\section{Introduction}
Mabuchi and Aubin-Yau functionals play a crucial role in studying
K\"ahler-Einstein metrics and constant scalar curvatures(see
\cite{PS}). How to generalize these functionals from K\"ahler
geometry to complex geometry is an interesting problem:

\begin{question} \label{q1.1} Can we define Mabuchi and Aubin-Yau functionals
over compact complex manifolds so that these functionals coincide
with the original definitions and satisfy the same basic properties?
\end{question}
In \cite{L1,L2}, the author answered this question in the complex dimension two and
three, respectively, and proved similar results in the K\"ahler setting. By
carefully checking and using a similar method in \cite{L2}, we can
construct those functionals in higher dimension cases. So, now, we
give an affirmative answer to Question \ref{q1.1}.

\subsection{Mabuchi and Aubin-Yau functionals on K\"ahler manifolds}
In this subsection we review Mabuchi and Aubin-Yau functionals on
K\"ahler manifolds, and describe some basic properties of these
functionals which also hold in any complex manifolds. Let $(X,\omega)$ be a compact
K\"ahler manifold of the complex dimension $n$. Then the volume
\begin{equation}
V_{\omega}:=\int_{X}\omega^{n}\label{1.1}
\end{equation}
depends only on the K\"ahler class of $\omega$. Let
$\mathcal{P}^{{\rm Kahler}}_{\omega}$ denote the space of K\"ahler
potentials and define the Mabuchi functional, for any smooth
functions $\varphi',\varphi''\in\mathcal{P}^{{\rm
Kahler}}_{\omega}$, by
\begin{equation}
\mathcal{L}^{{\rm M},{\rm
Kahler}}_{\omega}(\varphi',\varphi''):=\frac{1}{V_{\omega}}\int^{1}_{0}\int_{X}\dot{\varphi}_{t}
\omega^{n}_{\varphi_{t}}dt\label{1.2}
\end{equation}
where $\varphi_{t}$ is any smooth path in $\mathcal{P}^{{\rm
Kahler}}_{\omega}$ from $\varphi'$ to $\varphi''$. Mabuchi \cite{M}
showed that (\ref{1.2}) is well-defined.

Using (\ref{1.2}) we can define Aubin-Yau functionals, for any
smooth function $\varphi\in\mathcal{P}^{{\rm Kahler}}_{\omega}$, as
follows:
\begin{eqnarray}
\mathcal{I}^{{\rm AY},{\rm
Kahler}}_{\omega}(\varphi)&=&\frac{1}{V_{\omega}}\int_{X}\varphi(\omega^{n}-\omega^{n}_{\varphi}),
\label{1.3}\\
\mathcal{J}^{{\rm AY},{\rm
Kahler}}_{\omega}(\varphi)&=&-\mathcal{L}^{{\rm M},{\rm
Kahler}}_{\omega}(0,\varphi)+\frac{1}{V_{\omega}}\int_{X}\varphi\omega^{n}\label{1.4}.
\end{eqnarray}
So Aubin-Yau functionals are also well-defined. The basic and often
useful inequalities, for any smooth function
$\varphi\in\mathcal{P}^{{\rm Kahler}}_{\omega}$, are
\begin{eqnarray}
\frac{n}{n+1}\mathcal{I}^{{\rm AY},{\rm
Kahler}}_{\omega}(\varphi)-\mathcal{J}^{{\rm AY},{\rm
Kahler}}_{\omega}(\varphi)&\geq&0, \label{1.5}\\
(n+1)\mathcal{J}^{{\rm AY},{\rm
Kahler}}_{\omega}(\varphi)-\mathcal{I}^{{\rm AY},{\rm
Kahler}}_{\omega}(\varphi)&\geq&0. \label{1.6}
\end{eqnarray}
An important consequence is that we will use the inequalities
(\ref{1.5}) and (\ref{1.6}) to determine the extra terms on the
definitions of $\mathcal{I}^{{\rm AY}}_{\omega}(\varphi)$ and
$\mathcal{J}^{{\rm AY}}_{\omega}(\varphi)$, which are Aubin-Yau functionals
over complex manifolds.

However, if $\omega$ is not closed, then the above definitions
(\ref{1.2}), (\ref{1.3}), and (\ref{1.4}) do not make any sense.
Hence we should add some extra terms on the definitions of those
functionals; these extra terms should involve $\partial\omega$ and
$\overline{\partial}\omega$, but, the essential question is to find
the structure of the extra terms. Roughly speaking, $\mathcal{I}^{{\rm AY}}_{\omega}(\varphi)$
and $\mathcal{J}^{{\rm AY}}_{\omega}(\varphi)$ can be written as
\begin{eqnarray*}
\mathcal{I}^{{\rm AY}}_{\omega}(\varphi)&=&\mathcal{I}^{{\rm AY},{\rm Kahler}}_{\omega}(\varphi)
+\text{terms involving} \ \partial\omega, \overline{\partial}\varphi
+\text{terms invloving} \ \overline{\partial}\omega, \partial\varphi, \\
\mathcal{J}^{{\rm AY}}_{\omega}(\varphi)&=&\mathcal{J}^{{\rm AY},{\rm Kahler}}_{\omega}(\varphi)
+\text{terms involving} \ \partial\omega,\overline{\partial}\varphi+
\text{terms involving} \ \overline{\partial}\omega, \partial\varphi.
\end{eqnarray*}
In the following sections, we will explicitly determine the extra terms.

Throughout the rest part of this paper, we denote by $(X,g)$ a
compact complex manifold of the complex dimension $n\geq2$, and
$\omega$ be the associated real $(1,1)$-form. Let
\begin{equation}
\mathcal{P}_{\omega}:=\left\{\varphi\in C^{\infty}(X)_{\mathbb{R}}\Big|
\omega_{\varphi}:=\omega+\sqrt{-1}\partial\overline{\partial}\varphi>0\right\}\label{1.7}
\end{equation}
be the space of all real-valued smooth functions on $X$ whose
associated real $(1,1)$-forms are positive.

\subsection{Mabuchi and Aubin-Yau functionals on complex
surfaces}
In this subsection we recall the main result in \cite{L1}. Let $(X,g)$ be a compact
complex manifold of the complex dimension $2$ and $\omega$ be its
associated real $(1,1)$-form. For any
$\varphi',\varphi''\in\mathcal{P}_{\omega}$, we define
\begin{eqnarray}
\mathcal{L}^{{\rm
M}}_{\omega}(\varphi',\varphi'')&:=&\frac{1}{V_{\omega}}\int^{1}_{0}\int_{X}\dot{\varphi}_{t}\cdot
\omega^{2}_{\varphi_{t}}dt \label{1.8}\\
&-&\frac{1}{V_{\omega}}\int^{1}_{0}\int_{X}\sqrt{-1}\partial\omega\wedge(\overline{\partial}\dot{\varphi}_{t}\cdot\varphi_{t})dt
\nonumber\\
&+&\frac{1}{V_{\omega}}\int^{1}_{0}\int_{X}\sqrt{-1}\overline{\partial}\omega\wedge(\partial\dot{\varphi}_{t}\cdot\varphi_{t})dt,\nonumber
\end{eqnarray}
where $\{\varphi_{t}\}_{0\leq t\leq1}$ is any smooth path in
$\mathcal{P}_{\omega}$ from $\varphi'$ to $\varphi''$. Then in
\cite{L1} we showed that the functional (\ref{1.8}) is independent
of the choice of the smooth path $\{\varphi_{t}\}_{0\leq t\leq1}$.
If we set
\begin{equation*}
\mathcal{L}^{{\rm M}}_{\omega}(\varphi):=\mathcal{L}^{{\rm
M}}_{\omega}(0,\varphi),
\end{equation*}
then we have an explicit formula \cite{L1} of $\mathcal{L}^{{\rm M}}_{\omega}(\varphi)$:
\begin{eqnarray}
\mathcal{L}^{{\rm
M}}_{\omega}(\varphi)&=&\frac{1}{3V_{\omega}}\int_{X}\varphi(\omega^{2}+\omega\wedge\omega_{\varphi}+\omega^{2}_{\varphi})
\label{1.9}
\\
&+&\frac{1}{2V_{\omega}}\int_{X}\varphi\left(-\sqrt{-1}\partial\omega\wedge\overline{\partial}\varphi
+\sqrt{-1}\overline{\partial}\omega\wedge\partial\varphi\right).\nonumber
\end{eqnarray}

Now Aubin-Yau functionals are defined by
\begin{eqnarray}
\mathcal{I}^{{\rm
AY}}_{\omega}(\varphi)&:=&\frac{1}{V_{\omega}}\int_{X}\varphi(\omega^{2}-\omega^{2}_{\varphi})
\label{1.10} \\
&-&\frac{1}{V_{\omega}}\int_{X}\varphi\cdot\sqrt{-1}\partial\omega\wedge\overline{\partial}\varphi
+\frac{1}{V_{\omega}}\int_{X}\varphi\cdot\sqrt{-1}\overline{\partial}\omega\wedge\partial\varphi,\nonumber
\\
\mathcal{J}^{{\rm AY}}_{\omega}(\varphi)&:=&-\mathcal{L}^{{\rm
M}}_{\omega}(\varphi)+\frac{1}{V_{\omega}}\int_{X}\varphi\omega^{2}
\label{1.11} \\
&-&\frac{1}{V_{\omega}}\int_{X}\varphi\cdot\sqrt{-1}\partial\omega\wedge\overline{\partial}\varphi
+\frac{1}{V_{\omega}}\int_{X}\varphi\cdot\sqrt{-1}\overline{\partial}\omega\wedge\partial\varphi.
\nonumber
\end{eqnarray}
Moreover they also satisfy the inequalities (\ref{1.5}) and (\ref{1.6}); that is
\begin{eqnarray}
\frac{2}{3}\mathcal{I}^{{\rm
AY}}_{\omega}(\varphi)-\mathcal{J}^{{\rm
AY}}_{\omega}(\varphi)&\geq&0, \label{1.12}\\
3\mathcal{J}^{{\rm
AY}}_{\omega}-\mathcal{I}^{{\rm
AY}}_{\omega}(\varphi)&\geq&0.\nonumber
\end{eqnarray}
\subsection{Mabuchi and Aubin-Yau functionals on complex
three-folds}
The functionals over complex three-folds are very different with
these over complex surfaces. For any
$\varphi',\varphi''\in\mathcal{P}_{\omega}$, we define
\begin{eqnarray}
\mathcal{L}^{{\rm
M}}_{\omega}(\varphi',\varphi'')&:=&\frac{1}{V_{\omega}}\int^{1}_{0}
\int_{X}\dot{\varphi}_{t}\omega^{3}_{\varphi_{t}}dt
\label{1.13}\\
&-&\frac{3}{V_{\omega}}\int^{1}_{0}\int_{X}\sqrt{-1}\partial\omega
\wedge\omega_{\varphi_{t}}
\wedge(\overline{\partial}\dot{\varphi}_{t}\cdot\varphi_{t})dt
\nonumber\\
&+&\frac{3}{V_{\omega}}\int^{1}_{0}\int_{X}\sqrt{-1}\overline{\partial}\omega
\wedge\omega_{\varphi_{t}}\wedge(\partial\dot{\varphi}_{t}\cdot\varphi_{t})dt\nonumber\\
&-&\frac{1}{V_{\omega}}\int^{1}_{0}\int_{X}\partial\varphi_{t}\wedge
\overline{\partial}\varphi_{t}\wedge\partial\omega\wedge\overline{\partial}\dot{\varphi}_{t}\nonumber
\\
&-&\frac{1}{V_{\omega}}\int^{1}_{0}\int_{X}\overline{\partial}\varphi_{t}\wedge\partial\varphi_{t}\wedge\overline{\partial}
\omega\wedge\partial\dot{\varphi}_{t},\nonumber
\end{eqnarray}
where $\{\varphi_{t}\}_{0\leq t\leq1}$ is any smooth path in
$\mathcal{P}_{\omega}$ from $\varphi'$ to $\varphi''$. In \cite{L2},
we proved that (\ref{1.13}) is well-defined.

For any $\varphi\in\mathcal{P}_{\omega}$ we also set
$\mathcal{L}^{{\rm M}}_{\omega}(\varphi):=\mathcal{L}^{{\rm
M}}_{\omega}(0,\varphi)$. If we chose $\varphi_{t}=t\cdot\varphi$, then we have an explicit formula \cite{L2} of
$\mathcal{L}^{{\rm M}}_{\omega}(\varphi)$:
\begin{eqnarray}
\mathcal{L}^{{\rm M}}_{\omega}(\varphi)&=&
\frac{1}{4V_{\omega}}\sum^{3}_{i=0}\int_{X}
\varphi\omega^{i}_{\varphi}\wedge\omega^{3-i}
\nonumber\\
&-&\sum^{1}_{i=0}\frac{i+1}{2V_{\omega}}\int_{X}\varphi\omega^{i}_{\varphi}
\wedge\omega^{1-i}\wedge\sqrt{-1}\partial\omega\wedge\overline{\partial}\varphi
\label{1.15}\\
&+&\sum^{1}_{i=0}\frac{i+1}{2V_{\omega}}\int_{X}\varphi\omega^{i}_{\varphi}
\wedge\omega^{1-i}\wedge\sqrt{-1}\overline{\partial}\omega\wedge\partial\varphi.\nonumber
\end{eqnarray}

Now we define Aubin-Yau functionals $\mathcal{I}^{{\rm
AY}}_{\omega}, \mathcal{J}^{{\rm AY}}_{\omega}$ for any compact
complex three-fold $(X,\omega)$:
\begin{eqnarray}
& & \mathcal{I}^{{\rm AY}}_{\omega}(\varphi) \ \ := \ \
\frac{1}{V_{\omega}}\int_{X}\varphi(\omega^{3}-\omega^{3}_{\varphi})
\label{1.15}\\
&-&\frac{3}{2V_{\omega}}\int_{X}\varphi\omega_{\varphi}
\wedge\sqrt{-1}\partial\omega\wedge\overline{\partial}\varphi-
\frac{3}{2V_{\omega}}\int_{X}\varphi\omega\wedge
\sqrt{-1}\partial\omega\wedge \overline{\partial}\varphi\nonumber
\\
&+&\frac{3}{2V_{\omega}}\int_{X}\varphi\omega_{\varphi}
\wedge\sqrt{-1}\overline{\partial}\omega\wedge\partial\varphi
+\frac{3}{2V_{\omega}}\int_{X}\varphi\omega\wedge\sqrt{-1}\overline{\partial}\omega\wedge\partial\varphi
\nonumber,\\
& & \mathcal{J}^{{\rm AY}}_{\omega}(\varphi) \ \ := \ \
-\mathcal{L}^{{\rm
M}}_{\omega}(\varphi)+\frac{1}{V_{\omega}}\int_{X}\varphi\omega^{3}
\label{1.16}\\
&-&\frac{3}{2V_{\omega}}\int_{X}\varphi\omega_{\varphi}
\wedge\sqrt{-1}\partial\omega\wedge\overline{\partial}\varphi
-\frac{3}{2V_{\omega}}\int_{X}\varphi\wedge\sqrt{-1}\partial
\omega\wedge\overline{\partial}\varphi\nonumber\\
&+&\frac{3}{2V_{\omega}}\int_{X}\varphi\omega_{\varphi}
\wedge\sqrt{-1}\overline{\partial}\omega\wedge\partial\varphi
+\frac{3}{2V_{\omega}}\int_{X}\varphi\omega\wedge\sqrt{-1}\overline{\partial}\omega\wedge\partial\varphi.\nonumber
\end{eqnarray}
Similarly, we have
\begin{eqnarray}
\frac{3}{4}\mathcal{I}^{{\rm
AY}}_{\omega}(\varphi)-\mathcal{J}^{{\rm
AY}}_{\omega}(\varphi)&\geq&0, \label{1.17}\\
\ \ \  4\mathcal{J}^{{\rm
AY}}_{\omega}(\varphi)-\mathcal{I}^{{\rm
AY}}_{\omega}(\varphi)&\geq&0.\nonumber
\end{eqnarray}

\subsection{Mabuchi and Aubin-Yau functionals on
complex manifolds}
In this subsection we assume that the complex dimension $n$ of the
compact complex manifold $(X,\omega)$ is bigger than or equal to
$3$. For $\varphi',\varphi''\in\mathcal{P}_{\omega}$, define
\begin{eqnarray}
&&\mathcal{L}^{{\rm
M}}_{\omega}(\varphi',\varphi'') \ \ := \ \ \frac{1}{V_{\omega}}\int^{1}_{0}\int_{X}\dot{\varphi}_{t}\omega^{n}_{\varphi_{t}}dt\label{1.18}\\
&-&\frac{1}{V_{\omega}}\int^{1}_{0}\int_{X}\frac{n(n-1)\sqrt{-1}}{2}\partial\omega\wedge\omega^{n-2}_{\varphi_{t}}\wedge(\overline{\partial}\dot{\varphi}_{t}\cdot\varphi_{t})dt
\nonumber\\
&+&\frac{1}{V_{\omega}}\int^{1}_{0}\int_{X}\frac{n(n-1)\sqrt{-1}}{2}\overline{\partial}\omega\wedge\omega^{n-2}_{\varphi_{t}}\wedge(\partial\dot{\varphi}_{t}\cdot\varphi_{t})dt
\nonumber\\
&+&\sum^{n-2}_{i=1}\frac{1}{V_{\omega}}\int^{1}_{0}\int_{X}(-1)^{i}\binom{n}{i+2}\partial\varphi_{t}\wedge\partial\omega\wedge\overline{\partial}\dot{\varphi}_{t}\wedge\overline{\partial}
\varphi_{t}\wedge\omega^{n-i-2}_{\varphi_{t}}\wedge(\sqrt{-1}\partial\overline{\partial}\varphi_{t})^{i-1}\nonumber\\
&+&\sum^{n-2}_{i=1}\frac{1}{V_{\omega}}\int^{1}_{0}\int_{X}(-1)^{i}\binom{n}{i+2}\overline{\partial}\varphi_{t}\wedge\overline{\partial}
\omega\wedge\partial\dot{\varphi}_{t}\wedge
\partial\varphi_{t}\wedge\omega^{n-i-2}_{\varphi_{t}}\wedge(\sqrt{-1}\partial\overline{\partial}\varphi_{t})^{i-1},\nonumber
\end{eqnarray}
where $\{\varphi_{t}\}_{0\leq t\leq 1}$ is any smooth path in
$\mathcal{P}_{\omega}$ from $\varphi'$ to $\varphi''$. In Section 2,
we prove

\begin{theorem} \label{thm1.2} For any $n\geq3$, the functional (\ref{1.18})
is independent of the choice of the smooth path
$\{\varphi_{t}\}_{0\leq t\leq1}$ in $\mathcal{P}_{\omega}$ from
$\varphi'$ to $\varphi''$.
\end{theorem}

As a consequence of Theorem \ref{thm1.2}, by taking
$\varphi_{t}=t\cdot\varphi$, we have, for any
$\varphi\in\mathcal{P}_{\omega}$,
\begin{eqnarray}
\mathcal{L}^{{\rm M}}_{\omega}(\varphi)&:=&\mathcal{L}^{{\rm
M}}_{\omega}(0,\varphi) \ \ = \ \
\frac{1}{V_{\omega}}\sum^{n}_{i=0}\int_{X}\frac{1}{n+1}\varphi\omega^{i}_{\varphi}\wedge\omega^{n-i}
\nonumber\\
&-&\sum^{n-2}_{i=0}\frac{i+1}{2V_{\omega}}\int_{X}\varphi\omega^{i}_{\varphi}\wedge\omega^{n-2-i}\wedge\sqrt{-1}\partial\omega\wedge\overline{\partial}\varphi
\label{1.19}\\
&+&\sum^{n-2}_{i=0}\frac{i+1}{2V_{\omega}}\int_{X}\varphi\omega^{i}_{\varphi}\wedge\omega^{n-2-i}\wedge\sqrt{-1}\overline{\partial}\omega\wedge\partial\varphi.\nonumber
\end{eqnarray}

As before, we define Aubin-Yau functionals for any complex manifolds
by
\begin{eqnarray}
\mathcal{I}^{{\rm AY}}_{\omega}(\varphi) &:=&
\frac{1}{V_{\omega}}\int_{X}\varphi(\omega^{n}-\omega^{n}_{\varphi})\label{1.20}-\frac{n}{2V_{\omega}}\sum^{n-2}_{i=0}
\int_{X}\varphi\omega^{i}_{\varphi}\wedge\omega^{n-2-i}
\wedge\sqrt{-1}\partial\omega\wedge\overline{\partial}\varphi\nonumber\\
&+&\frac{n}{2V_{\omega}}\sum^{n-2}_{i=0}\int_{X}\varphi\omega^{i}_{\varphi}\wedge\omega^{n-2-i}
\wedge\sqrt{-1}\overline{\partial}\omega\wedge\partial\varphi\label{1.20}\\
&=&\frac{1}{V_{\omega}}\sum^{n-1}_{i=0}\int_{X}\sqrt{-1}\partial\varphi\wedge\overline{\partial}\varphi\wedge\omega^{i}_{\varphi}\wedge\omega^{n-1-i},\nonumber\\
\mathcal{J}^{{\rm AY}}_{\omega}(\varphi)&:=&-\mathcal{L}^{{\rm
M}}_{\omega}(\varphi)+\frac{1}{V_{\omega}}\int_{X}\varphi\omega^{n}
-\frac{n}{2V_{\omega}}\sum^{n-2}_{i=0}\int_{X}\varphi\omega^{i}_{\varphi}\wedge\omega^{n-2-i}
\wedge\sqrt{-1}\partial\omega\wedge\overline{\partial}\varphi\nonumber\\
&+&\frac{n}{2V_{\omega}}\sum^{n-2}_{i=0}\int_{X}\varphi\omega^{i}_{\varphi}\wedge\omega^{n-2-i}
\wedge\sqrt{-1}\overline{\partial}\omega\wedge\partial\varphi\label{1.21}\\
&=&\frac{1}{V_{\omega}}\sum^{n-1}_{i=0}\int_{X}\frac{n-i}{n+1}\sqrt{-1}\partial\varphi\wedge\overline{\partial}\varphi\wedge\omega^{i}_{\varphi}\wedge\omega^{n-1-i}.\nonumber
\end{eqnarray}

In Section 3, we shall show the following

\begin{theorem} \label{thm1.3} For any $\varphi\in\mathcal{P}_{\omega}$, one has
\begin{eqnarray}
\frac{n}{n+1}\mathcal{I}^{{\rm
AY}}_{\omega}(\varphi)-\mathcal{J}^{{\rm
AY}}_{\omega}(\varphi)&\geq&0, \label{1.22}\\
(n+1)\mathcal{J}^{{\rm AY}}_{\omega}(\varphi)-\mathcal{I}^{{\rm
AY}}_{\omega}(\varphi)&\geq&0.\label{1.23}
\end{eqnarray}
In particular, $\mathcal{I}^{{\rm AY}}_{\omega}(\varphi),\mathcal{J}^{{\rm AY}}_{\omega}(\varphi)$ are nonnegative, and 
\begin{eqnarray}
\frac{1}{n+1}\mathcal{I}^{{\rm
AY}}_{\omega}(\varphi)&\leq&\mathcal{J}^{{\rm AY}}_{\omega}(\varphi)
\ \ \leq \ \ \frac{n}{n+1}\mathcal{I}^{{\rm AY}}_{\omega}(\varphi),
\label{1.24}\\
\frac{n+1}{n}\mathcal{J}^{{\rm
AY}}_{\omega}(\varphi)&\leq&\mathcal{I}^{{\rm AY}}_{\omega}(\varphi)
\ \ \leq \ \ (n+1)\mathcal{J}^{{\rm AY}}_{\omega}(\varphi), \label{1.25}\\
\frac{1}{n}\mathcal{J}^{{\rm AY}}_{\omega}(\varphi) \ \ \leq \ \
\frac{1}{n+1}\mathcal{J}^{{\rm AY}}_{\omega}(\varphi)&\leq&
\mathcal{I}^{{\rm
AY}}_{\omega}(\varphi)-\mathcal{J}^{{\rm AY}}_{\omega}(\varphi)\label{1.26}\\
&\leq&\frac{n}{n+1}\mathcal{I}^{{\rm AY}}_{\omega}(\varphi) \ \ \leq
\ \ n\mathcal{J}^{{\rm AY}}_{\omega}(\varphi).\nonumber
\end{eqnarray}
\end{theorem}
\subsection{Further questions}
Up to now we consider functionals over compact complex manifolds without boundary, and
we hope that the similar constructions can be achieved for compact complex
manifolds with boundary.

\begin{question} \label{q1.4} Can we define Mabuchi and Aubin-Yau functionals over compact complex
manifolds with boundary so that these functionals coincide with the original definitions
and satisfy the same basic properties?
\end{question}

There are other functionals, for example, Mabuchi $\mathcal{K}^{{\rm M}}_{\omega}$ \cite{M} functional, Chen-Tian functionals  \cite{CT}, etc. We can ask the following

\begin{question} \label{q1.5} Can we define the analogy Mabuchi $\mathcal{K}^{{\rm M}}_{\omega}$ and
Chen-Tian functionals over complex manifolds with(out) boundary?
\end{question}

In the future, we will study those two questions.\\

{\bf Acknowledgements.} The author would like to thank Valentino
Tosatti who read this note and pointed out a serious mistake in the
first version.

\section{Mabuchi $\mathcal{L}^{{\rm M}}_{\omega}$ functional on
complex manifolds}
In this section we assume that $(X,\omega)$ is a compact complex
manifold of the complex dimension $n\geq3$. For any two complex
forms $\alpha$ and $\beta$, we frequently use the following
formulas: if $|\alpha|+|\beta|=2n-1$, then
\begin{eqnarray}
\int_{X}\alpha\wedge\partial\beta&=&(-1)^{|\beta|}\int_{X}\partial\alpha\wedge\beta
\ \ = \ \ -(-1)^{|\alpha|}\int_{X}\partial\alpha\wedge\beta, \label{2.1}\\
\int_{X}\alpha\wedge\overline{\partial}\beta&=&(-1)^{|\beta|}\int_{X}\overline{\partial}\alpha\wedge\beta
\ \ = \ \
-(-1)^{|\alpha|}\int_{X}\overline{\partial}\alpha\wedge\beta.
\label{2.2}
\end{eqnarray}
Another useful formula is
\begin{equation}
\alpha\wedge\alpha=0, \ \ \ \text{if} \ |\alpha| \ \text{is odd}.
\label{2.3}
\end{equation}
By the definition of operators $\partial$ and $\overline{\partial}$,
one has
\begin{equation}
\partial\overline{\partial}+\overline{\partial}\partial=0.\label{2.4}
\end{equation}
Hence the complex conjugate of the operator $\sqrt{-1}\partial\overline{\partial}$ is itself.

\subsection{The main idea}
Similarly in \cite{L2}, we consider $\psi(s,t)=s\cdot\varphi_{t}$
and we can show that (see (\ref{2.49}))
\begin{eqnarray}
\frac{2I^{0}}{n(n-1)\sqrt{-1}}&=&\frac{I^{1}}{a_{1}}-\frac{I^{2}}{a_{2}}+c_{1},
\label{2.5}\\
\frac{I^{3}}{a_{3}}+\frac{I^{4}}{a_{4}}&=&\frac{3}{-(n-2)\sqrt{-1}}c_{1}+c_{2},\label{2.6}
\end{eqnarray}
where $I^{i}$ are functionals which can be determined\footnote{A functional $\mathcal{I}$
is said to be determined if $d\Psi=\mathcal{I}\cdot dt\wedge ds$ for some $1$-form $\Psi$ on $[0,1]
\times[0,1]$.}, $c_{j}$ are
also functionals but may not be determined, and $a_{k}$ are nonzero
constants which can be determined later. We can use equation
(\ref{2.6}) to eliminate the undetermined term $c_{1}$, but there
arises another undetermined term $c_{2}$. Our strategy is to find a
determined expression for $c_{2}$. To achieve this we construct two
sequences $\{I^{2i+1}\}_{2\leq i\leq n-2}$, which can be determined,
and $\{c_{i}\}_{2\leq i\leq n-2}$, which may not be determined,
satisfying
\begin{equation}
\frac{I^{2i+1}}{a_{2i+1}}+\frac{I^{2i+2}}{a_{2i+2}}=\frac{i+2}{n-(i+1)}c_{i}+c_{i+1},
\ \ \ 2\leq i\leq n-2,\label{2.7}
\end{equation}
where $I^{2i+2}:=\overline{I^{2i+1}}$,
$a_{2i+2}:=\overline{a_{2i+1}}$ and $a_{2i+1}$ can be determined
later. By our construction we have $c_{n-1}\equiv0$ which gives us
the determined and explicit formula for $c_{2}$ in terms of
$I^{2i+1}$ and $a_{2i+1}$. More precisely, setting
\begin{equation}
J_{i}:=\frac{I^{2i+1}}{a_{2i+1}}+\frac{I^{2i+2}}{a_{2i+2}},
\label{2.8}
\end{equation}
yields
\begin{equation}
c_{i+1}=-\frac{i+2}{n-(i+1)}c_{i}+J_{i}, \ \ \ 2\leq i\leq n-2.
\label{2.9}
\end{equation}
By induction on $i$ we obtain
\begin{equation}
c_{i}=(-1)^{i-2}\frac{(i+1)!(n-i-1)!}{3!(n-3)!}c_{2}+\sum^{i-1}_{k=2}(-1)^{i-1-k}
\frac{(i+1)!(n-i-1)!}{(k+2)!(n-k-2)!}J_{k}.\label{2.10}
\end{equation}
Since $c_{n-1}\equiv0$ it follows that
\begin{eqnarray}
c_{2}&=&(-1)^{n-3}\frac{3!(n-3)!}{n!}\left[0-\sum^{n-2}_{k=2}(-1)^{n-2-k}\frac{n!}{(k+2)!(n-k-2)!}J_{k}\right]
\nonumber\\
&=&\sum^{n-2}_{k=2}(-1)^{k}\frac{3!(n-3)!}{(k+2)!(n-k-2)!}J_{k}, \ \
\ n\geq4.\label{2.11}
\end{eqnarray}

When $X$ is a three-fold, we knew that $c_{2}=0$ \cite{L2}, but this can be seen from
(\ref{2.11}) if we take $n=3$. Hence the formula (\ref{2.11}) holds for any $n\geq3$.

\subsection{The definitions of $c_{1}$ and $c_{2}$}
Firstly, we consider the "K\"ahler part" of Mabuchi functional. Let
\begin{equation}
\mathcal{L}^{0}_{\omega}(\varphi',\varphi''):=
\frac{1}{V_{\omega}}\int^{1}_{0}\int_{X}\dot{\varphi}_{t}
\omega^{n}_{\varphi_{t}}dt.\label{2.12}
\end{equation}
As in \cite{L1, L2}, we set $\psi(s,t):=s\cdot\varphi_{t}$, $0\leq t,s\leq 1$, and
consider the corresponding $1$-form on
$[0,1]\times[0,1]$,
\begin{equation}
\Psi^{0}=\left(\int_{X}\frac{\partial\psi}{\partial
s}\omega^{n}_{\psi}\right)ds+\left(\int_{X}\frac{\partial\psi}{\partial
t}\omega^{n}_{\psi}\right)dt.\label{2.13}
\end{equation}
Taking the differential on both sides implies
\begin{equation}
d\Psi^{0}=I^{0}\cdot dt\wedge ds\label{2.14}
\end{equation}
where
\begin{equation}
I^{0}:=\int_{X}\frac{\partial}{\partial
t}\left(\frac{\partial\psi}{\partial
s}\omega^{n}_{\psi}\right)-\int_{X}\frac{\partial}{\partial
s}\left(\frac{\partial\psi}{\partial t}\omega^{n}_{\psi}\right).\label{2.15}
\end{equation}
The explicit expression of $I^{0}$ can be determined as follows:
\begin{eqnarray*}
I^{0}&=&\int_{X}\left[\frac{\partial\psi}{\partial
s}n\omega^{n-1}_{\psi}\wedge\sqrt{-1}\partial\overline{\partial}\left(\frac{\partial\psi}{\partial
t}\right)-\frac{\partial\psi}{\partial
t}n\omega^{n-1}_{\psi}\wedge\sqrt{-1}\partial\overline{\partial}\left(\frac{\partial\psi}{\partial
s}\right)\right] \\
&=&\int_{X}n\frac{\partial\psi}{\partial
s}\omega^{n-1}_{\psi}\wedge\sqrt{-1}\partial\overline{\partial}\left(\frac{\partial\psi}{\partial
t}\right)-\int_{X}n\frac{\partial\psi}{\partial
t}\omega^{n-1}_{\psi}\wedge\sqrt{-1}\partial\overline{\partial}\left(\frac{\partial\psi}{\partial
s}\right) \\
&=&\int_{X}n\frac{\partial\psi}{\partial
s}\omega^{n-1}_{\psi}\wedge\sqrt{-1}\partial\overline{\partial}\left(\frac{\partial\psi}{\partial
t}\right)+\int_{X}n\frac{\partial\psi}{\partial
t}\omega^{n-1}_{\psi}\wedge\sqrt{-1}\overline{\partial}\partial\left(\frac{\partial\psi}{\partial
s}\right) \\
&=&-\int_{X}n\frac{\partial\psi}{\partial
s}\omega^{n-1}_{\psi}\wedge\sqrt{-1}\overline{\partial}\partial\left(\frac{\partial\psi}{\partial
t}\right)-\int_{X}n\frac{\partial\psi}{\partial
t}\omega^{n-1}_{\psi}\wedge\sqrt{-1}\partial\overline{\partial}\left(\frac{\partial\psi}{\partial
s}\right).
\end{eqnarray*}
Here we have two slightly different expressions of $I^{0}$, and in the following
we will use those expressions to find $c_{1}=A_{1}+B_{1}-(\overline{A_{1}}+\overline{B_{1}})$,
where $A_{1},B_{1},\overline{A_{1}}$ and $\overline{B_{1}}$ are determined later. This
technique will be frequently used in many places. Hence
\begin{eqnarray*}
I^{0}&=&\int_{X}-n\sqrt{-1}\partial\left(\frac{\partial\psi}{\partial
s}\omega^{n-1}_{\psi}\right)\wedge\overline{\partial}\left(\frac{\partial\psi}{\partial
t}\right)+\int_{X}-n\sqrt{-1}\overline{\partial}\left(\frac{\partial\psi}{\partial
t}\omega^{n-1}_{\psi}\right)\wedge\partial\left(\frac{\partial\psi}{\partial
s}\right) \\
&=&\int_{X}-n\sqrt{-1}\left[\partial\left(\frac{\partial\psi}{\partial
s}\right)\wedge\omega^{n-1}_{\psi}+(n-1)\frac{\partial\psi}{\partial
s}\omega^{n-2}_{\psi}\wedge\partial\omega\right]\wedge\overline{\partial}\left(\frac{\partial\psi}{\partial
t}\right) \\
&+&\int_{X}-n\sqrt{-1}\left[\overline{\partial}\left(\frac{\partial\psi}{\partial
t}\right)\wedge\omega^{n-1}_{\psi}+(n-1)\frac{\partial\psi}{\partial
t}\omega^{n-2}_{\psi}\wedge\overline{\partial}\omega\right]\wedge\partial\left(\frac{\partial\psi}{\partial
s}\right) \\
&=&\int_{X}-n(n-1)\sqrt{-1}\frac{\partial\psi}{\partial
s}\omega^{n-2}_{\psi}\wedge\partial\omega\wedge\overline{\partial}\left(\frac{\partial\psi}{\partial
t}\right) \\
&+&\int_{X}-n(n-1)\sqrt{-1}\frac{\partial\psi}{\partial
t}\omega^{n-2}_{\psi}\wedge\overline{\partial}\omega\wedge\partial\left(\frac{\partial\psi}{\partial
s}\right).
\end{eqnarray*}
Thus
\begin{eqnarray}
I^{0}&=&\int_{X}-n(n-1)\sqrt{-1}\frac{\partial\psi}{\partial
s}\omega^{n-2}_{\psi}\wedge\partial\omega\wedge\overline{\partial}\left(\frac{\partial\psi}{\partial
t}\right) \nonumber\\
&+&\int_{X}-n(n-1)\sqrt{-1}\frac{\partial\psi}{\partial
t}\omega^{n-2}_{\psi}\wedge\overline{\partial}\omega\wedge\partial\left(\frac{\partial\psi}{\partial
s}\right).\label{2.16}
\end{eqnarray}
Using another expression of $I^{0}$, or taking the complex conjugate on both sides
of (\ref{2.16}) since $I^{0}$ is real, one has
\begin{eqnarray}
I^{0}&=&\int_{X}n(n-1)\sqrt{-1}\frac{\partial\psi}{\partial s}\omega^{n-2}_{\psi}\wedge
\overline{\partial}\omega\wedge\partial\left(\frac{\partial\psi}{\partial t}\right)\nonumber \\
&+&\int_{X}n(n-1)\sqrt{-1}\frac{\partial\psi}{\partial t}\omega^{n-2}_{\psi}\wedge\partial\omega
\wedge\overline{\partial}\left(\frac{\partial\psi}{\partial s}\right).\label{2.17}
\end{eqnarray}
Hence, adding (\ref{2.17}) to (\ref{2.16}) and dividing by $n(n-1)\sqrt{-1}$
on both sides, we deduce
\begin{eqnarray}
\frac{2I^{0}}{n(n-1)\sqrt{-1}}&=&\int_{X}\overline{\partial}\left(\frac{\partial\psi}{\partial
t}\right)\frac{\partial\psi}{\partial
s}\wedge\omega^{n-2}_{\psi}\wedge\partial\omega+\int_{X}\partial\left(\frac{\partial\psi}{\partial
s}\right)\frac{\partial\psi}{\partial
t}\wedge\omega^{n-2}_{\psi}\wedge\overline{\partial}\omega \nonumber\\
&-&\int_{X}\overline{\partial}\left(\frac{\partial\psi}{\partial
s}\right)\frac{\partial\psi}{\partial
t}\wedge\omega^{n-2}_{\psi}\wedge\partial\omega-\int_{X}\partial\left(\frac{\partial\psi}{\partial
t}\right)\frac{\partial\psi}{\partial
s}\wedge\omega^{n-2}_{\psi}\wedge\overline{\partial}\omega. \label{2.18}
\end{eqnarray}
According to the expression (\ref{2.18}), we introduce two functionals
\begin{eqnarray}
\mathcal{L}^{1}_{\omega}(\varphi',\varphi'')&:=&\frac{1}{V_{\omega}}\int^{1}_{0}\int_{X}a_{1}\partial\omega\wedge\omega^{n-2}_{\varphi_{t}}\wedge(\overline{\partial}\dot{\varphi}_{t}\cdot\varphi_{t})dt,
\label{2.19}\\
\mathcal{L}^{2}_{\omega}(\varphi',\varphi'')&:=&\frac{1}{V_{\omega}}
\int^{1}_{0}\int_{X}a_{2}\overline{\partial}\omega\wedge
\omega^{n-2}_{\varphi_{t}}\wedge(\partial\dot{\varphi}_{t}\cdot\varphi_{t})dt.\label{2.20}
\end{eqnarray}
Here $a_{1},a_{2}$ are non-zero constants determined later and we require
$\overline{a_{1}}=a_{2}$. Satisfying this condition, $a_{1}$ and $a_{2}$
have lots of solutions. In the following we will see that we take only
two special cases: $a_{i}$ are purely complex numbers; $a_{i}$ are real numbers.
Consider the corresponding two $1$-forms on $[0,1]\times[0,1]$,
\begin{eqnarray}
\Psi^{1}&=&\left[\int_{X}a_{1}\partial\omega\wedge\omega^{n-2}_{\psi}\wedge\left(\overline{\partial}\left(\frac{\partial\psi}{\partial
s}\right)\cdot\psi\right)\right]ds\label{2.21} \\
&+&\left[\int_{X}a_{1}\partial\omega\wedge\omega^{n-2}_{\psi}\wedge\left(\overline{\partial}\left(\frac{\partial\psi}{\partial
t}\right)\cdot\psi\right)\right]dt, \nonumber\\
\Psi^{2}&=&\left[\int_{X}a_{2}\overline{\partial}\omega\wedge\omega^{n-2}_{\psi}\wedge\left(\partial\left(\frac{\partial\psi}{\partial
s}\right)\cdot\psi\right)\right]ds\label{2.22} \\
&+&\left[\int_{X}a_{2}\overline{\partial}\omega\wedge\omega^{n-2}_{\psi}\wedge\left(\partial\left(\frac{\partial\psi}{\partial
t}\right)\cdot\psi\right)\right]dt.\nonumber
\end{eqnarray}
Firstly, we compute the differential of $\Psi^{1}$, and the differential $d\Psi^{2}$ can be
easily written down only by taking the complex conjugate on both sides. Calculate
\begin{equation}
d\Psi^{1}=I^{1}\cdot dt\wedge ds \label{2.23}
\end{equation}
where
\begin{eqnarray}
I^{1}&=&\int_{X}a_{1}\frac{\partial}{\partial
t}\left[\partial\omega\wedge\omega^{n-2}_{\psi}\wedge\left(\overline{\partial}\left(\frac{\partial\psi}{\partial
s}\right)\cdot\psi\right)\right]\label{2.24} \\
&-&\int_{X}a_{1}\frac{\partial}{\partial
s}\left[\partial\omega\wedge\omega^{n-2}_{\psi}\wedge\left(\overline{\partial}\left(\frac{\partial\psi}{\partial
t}\right)\cdot\psi\right)\right].\nonumber
\end{eqnarray}
Dividing by $a_{1}$ on both sides of (\ref{2.24}), we have
\begin{eqnarray*}
\frac{I^{1}}{a_{1}}&=&\int_{X}-\frac{\partial}{\partial
t}\left[\left(\psi\cdot\overline{\partial}\left(\frac{\partial\psi}{\partial
s}\right)\right)\wedge\omega^{n-2}_{\psi}\wedge\partial\omega\right]
\\
&+&\int_{X}\frac{\partial}{\partial
s}\left[\left(\psi\cdot\overline{\partial}\left(\frac{\partial\psi}{\partial
t}\right)\right)\wedge\omega^{n-2}_{\psi}\wedge\partial\omega\right]
\\
&=&\int_{X}-\left[\frac{\partial\psi}{\partial
t}\cdot\overline{\partial}\left(\frac{\partial\psi}{\partial
s}\right)+\psi\cdot\overline{\partial}\left(\frac{\partial^{2}\psi}{\partial
t\partial
s}\right)\right]\wedge\omega^{n-2}_{\psi}\wedge\partial\omega \\
&+&\int_{X}\left[\frac{\partial\psi}{\partial
s}\cdot\overline{\partial}\left(\frac{\partial\psi}{\partial
t}\right)+\psi\cdot\overline{\partial}\left(\frac{\partial^{2}\psi}{\partial
s\partial
t}\right)\right]\wedge\omega^{n-2}_{\psi}\wedge\partial\omega \\
&+&\int_{X}\psi\cdot\overline{\partial}\left(\frac{\partial\psi}{\partial
s}\right)\wedge(n-2)\omega^{n-3}_{\psi}\wedge-\sqrt{-1}\partial\overline{\partial}\left(\frac{\partial\psi}{\partial
t}\right)\wedge\partial\omega \\
&+&\int_{X}\psi\cdot\overline{\partial}\left(\frac{\partial\psi}{\partial
t}\right)\wedge(n-2)\omega^{n-3}_{\psi}\wedge\sqrt{-1}\partial\overline{\partial}\left(\frac{\partial\psi}{\partial
s}\right)\wedge\partial\omega \\
&=&\int_{X}-\frac{\partial\psi}{\partial
t}\cdot\overline{\partial}\left(\frac{\partial\psi}{\partial
s}\right)\wedge\omega^{n-2}_{\psi}\wedge\partial\omega+\int_{X}\frac{\partial\psi}{\partial
s}\cdot\overline{\partial}\left(\frac{\partial\psi}{\partial
t}\right)\wedge\omega^{n-2}_{\psi}\wedge\partial\omega \\
&+&\int_{X}\psi\cdot\overline{\partial}\left(\frac{\partial\psi}{\partial
s}\right)\wedge(n-2)\omega^{n-3}_{\psi}\wedge-\sqrt{-1}\partial\overline{\partial}\left(\frac{\partial\psi}{\partial
t}\right)\wedge\partial\omega \\
&+&\int_{X}\psi\cdot\overline{\partial}\left(\frac{\partial\psi}{\partial
t}\right)\wedge(n-2)\omega^{n-3}_{\psi}\wedge\sqrt{-1}\partial\overline{\partial}\left(\frac{\partial\psi}{\partial
s}\right)\wedge\partial\omega.
\end{eqnarray*}
To simplify the notation, we set
\begin{eqnarray}
A_{1}&:=&\int_{X}\psi(n-2)\overline{\partial}\left(\frac{\partial\psi}{\partial
s}\right)\wedge\omega^{n-3}_{\psi}\wedge\partial\omega\wedge-\sqrt{-1}\partial\overline{\partial}\left(\frac{\partial\psi}{\partial
t}\right), \label{2.25}\\
B_{1}&:=&\int_{X}\psi(n-2)\overline{\partial}\left(\frac{\partial\psi}{\partial
t}\right)\wedge\omega^{n-3}_{\psi}\wedge\partial\omega\wedge\sqrt{-1}\partial\overline{\partial}\left(\frac{\partial\psi}{\partial
s}\right).\label{2.26}
\end{eqnarray}
Consequently, $I^{1}/a_{1}$ can be written as
\begin{eqnarray}
\frac{I^{1}}{a_{1}}&=&\int_{X}-\frac{\partial\psi}{\partial t}\cdot\overline{\partial}
\left(\frac{\partial\psi}{\partial s}\right)\wedge\omega^{n-2}_{\psi}\wedge\partial\omega
+\int_{X}\frac{\partial\psi}{\partial s}\cdot\overline{\partial}\left(\frac{\partial\psi}{\partial t}\right)
\wedge\omega^{n-2}_{\psi}\wedge\partial\omega\nonumber \\
&+&A_{1}+B_{1}.\label{2.27}
\end{eqnarray}
Similarly, we can define $I^{2}$ by
\begin{equation}
d\Psi^{2}=I^{2}\cdot dt\wedge ds, \label{2.28}
\end{equation}
where
\begin{eqnarray}
\frac{I^{2}}{a_{2}}&:=&\int_{X}-\frac{\partial\psi}{\partial
t}\cdot\partial\left(\frac{\partial\psi}{\partial
s}\right)\wedge\omega^{n-2}_{\psi}\wedge\overline{\partial}\omega+\int_{X}\frac{\partial\psi}{\partial
s}\cdot\partial\left(\frac{\partial\psi}{\partial
t}\right)\wedge\omega^{n-2}_{\psi}\wedge\overline{\partial}\omega \nonumber\\
&+&\overline{A_{1}}+\overline{B_{1}}.\label{2.29}
\end{eqnarray}
Consequently, from (\ref{2.18}), (\ref{2.27}) and (\ref{2.29}),
\begin{equation}
\frac{2I^{0}}{n(n-1)\sqrt{-1}}=\frac{I^{1}}{a_{1}}-\frac{I^{2}}{a_{2}}+(A_{1}+B_{1})
-(\overline{A_{1}}+\overline{B_{1}}).\label{2.30}
\end{equation}
By a direct computation and using (\ref{2.1}) and (\ref{2.3}), one can show that the sum $A_{1}+B_{1}$ has a nice form:
\begin{eqnarray*}
A_{1}&=&\int_{X}\sqrt{-1}\partial\left[(n-2)\psi\cdot\overline{\partial}\left(\frac{\partial\psi}{\partial
s}\right)\wedge\omega^{n-3}_{\psi}\wedge\partial\omega\right]\wedge\overline{\partial}\left(\frac{\partial\psi}{\partial
t}\right) \\
&=&\int_{X}\sqrt{-1}(n-2)\left[\partial\left(\psi\cdot\overline{\partial}\left(\frac{\partial\psi}{\partial
s}\right)\right)\wedge\omega^{n-3}_{\psi}\wedge\partial\omega\right]\wedge\overline{\partial}\left(\frac{\partial\psi}{\partial
t}\right)\\
&=&\int_{X}\sqrt{-1}(n-2)\left[\partial\psi\wedge\overline{\partial}\left(\frac{\partial\psi}{\partial
s}\right)+\psi\cdot\partial\overline{\partial}\left(\frac{\partial\psi}{\partial
s}\right)\right]\wedge\omega^{n-3}_{\psi}\wedge\partial\omega\wedge\overline{\partial}\left(\frac{\partial\psi}{\partial
t}\right) \\
&=&\int_{X}(n-2)\psi\cdot\sqrt{-1}\partial\overline{\partial}\left(\frac{\partial\psi}{\partial
s}\right)\wedge\omega^{n-3}_{\psi}\wedge\partial\omega\wedge\overline{\partial}\left(\frac{\partial\psi}{\partial
t}\right) \\
&+&\int_{X}(n-2)\sqrt{-1}\partial\psi\wedge\overline{\partial}\left(\frac{\partial\psi}{\partial
s}\right)\wedge\omega^{n-3}_{\psi}\wedge\partial\omega\wedge\overline{\partial}\left(\frac{\partial\psi}{\partial
t}\right) \\
&=&\int_{X}-(n-2)\psi\cdot\overline{\partial}\left(\frac{\partial\psi}{\partial
t}\right)\wedge\omega^{n-3}_{\psi}\wedge\partial\omega\wedge\sqrt{-1}\partial\overline{\partial}\left(\frac{\partial\psi}{\partial
s}\right) \\
&+&\int_{X}-(n-2)\sqrt{-1}\partial\psi\wedge\partial\omega\wedge\overline{\partial}\left(\frac{\partial\psi}{\partial
s}\right)\wedge\overline{\partial}\left(\frac{\partial\psi}{\partial
t}\right)\wedge\omega^{n-3}_{\psi} \\
&=&-B_{1}-\int_{X}(n-2)\sqrt{-1}\partial\psi\wedge\partial\omega\wedge\overline{\partial}\left(\frac{\partial\psi}{\partial
s}\right)\wedge\overline{\partial}\left(\frac{\partial\psi}{\partial
t}\right)\wedge\omega^{n-3}_{\psi}.
\end{eqnarray*}
Adding the term $B_{1}$ on both sides gives
\begin{equation}
A_{1}+B_{1}=\int_{X}-(n-2)\sqrt{-1}\partial\psi\wedge\partial\omega\wedge\overline{\partial}\left(\frac{\partial\psi}{\partial
s}\right)\wedge\overline{\partial}\left(\frac{\partial\psi}{\partial
t}\right)\wedge\omega^{n-3}_{\psi}. \label{2.31}
\end{equation}

According to (\ref{2.31}), we define
\begin{eqnarray}
\mathcal{L}^{3}_{\omega}(\varphi',\varphi'')&:=&\frac{1}{V_{\omega}}\int^{1}_{0}\int_{X}a_{3}\partial\varphi_{t}\wedge\partial\omega\wedge\overline{\partial}\dot{\varphi}_{t}\wedge
\overline{\partial}\varphi_{t}\wedge\omega^{n-3}_{\varphi_{t}},\label{2.32} \\
\mathcal{L}^{4}_{\omega}(\varphi',\varphi'')&:=&\frac{1}{V_{\omega}}\int^{1}_{0}\int_{X}a_{4}\overline{\partial}\varphi_{t}\wedge\overline{\partial}\omega
\wedge\partial\dot{\varphi}_{t}\wedge\partial\varphi_{t}\wedge\omega^{n-3}_{\varphi_{t}},\label{2.33}
\end{eqnarray}
where $a_{3},a_{4}$ are nonzero constants determined later and
we require $\overline{a_{3}}=a_{4}$. Consider
\begin{eqnarray}
\Psi^{3}&=&\left[\int_{X}a_{3}\partial\psi\wedge\partial\omega\wedge\overline{\partial}\left(\frac{\partial\psi}{\partial
s}\right)\wedge\overline{\partial}\psi\wedge\omega^{n-3}_{\psi}\right]ds
\label{2.34}\\
&+&\left[\int_{X}a_{3}\partial\psi\wedge\partial\omega\wedge\overline{\partial}\left(\frac{\partial\psi}{\partial
t}\right)\wedge\overline{\partial}\psi\wedge\omega^{n-3}_{\psi}\right]dt, \nonumber\\
\Psi^{4}&=&\left[\int_{X}a_{4}\overline{\partial}\psi\wedge\overline{\partial}\omega
\wedge\partial\left(\frac{\partial\psi}{\partial s}\right)\wedge\partial\psi\wedge
\omega^{n-3}_{\psi}\right]ds \label{2.35}\\
&+&\left[\int_{X}a_{4}\overline{\partial}\psi\wedge\overline{\partial}\omega
\wedge\partial\left(\frac{\partial\psi}{\partial t}\right)\wedge\partial\psi\wedge
\omega^{n-3}_{\psi}\right]dt.\nonumber
\end{eqnarray}
Calculate
\begin{equation}
d\Psi^{3}=I^{3}\cdot dt\wedge ds,\label{2.36}
\end{equation}
where
\begin{eqnarray}
I^{3}&:=&\int_{X}a_{3}\frac{\partial}{\partial
t}\left[\partial\psi\wedge\partial\omega\wedge\overline{\partial}\left(\frac{\partial\psi}{\partial
s}\right)\wedge\overline{\partial}\psi\wedge\omega^{n-3}_{\psi}\right]
\label{2.37}\\
&-&\int_{X}a_{3}\frac{\partial}{\partial
s}\left[\partial\psi\wedge\partial\omega\wedge\overline{\partial}\left(\frac{\partial\psi}{\partial
t}\right)\wedge\overline{\partial}\psi\wedge\omega^{n-3}_{\psi}\right].\nonumber
\end{eqnarray}
Also, we can calculate
\begin{equation}
d\Psi^{4}=I^{4}\cdot dt\wedge ds,
\end{equation}
where
\begin{eqnarray}
I^{4}&:=&\int_{X}a_{4}\frac{\partial}{\partial t}\left[\overline{\partial}\psi\wedge
\overline{\partial}\omega\wedge\partial\left(\frac{\partial\psi}{\partial s}\right)\wedge\partial\psi
\wedge\omega^{n-3}_{\psi}\right] \label{2.38} \\
&-&\int_{X}a_{4}\frac{\partial}{\partial s}\left[\overline{\partial}\psi\wedge\overline{\partial}
\omega\wedge\partial\left(\frac{\partial\psi}{\partial t}\right)\wedge
\partial\psi\wedge\omega^{n-3}_{\psi}\right].\nonumber
\end{eqnarray}
Hence
\begin{eqnarray*}
\frac{I^{3}}{a_{3}}&=&\int_{X}\left[\partial\left(\frac{\partial\psi}{\partial
t}\right)\wedge\partial\omega\wedge\overline{\partial}\left(\frac{\partial\psi}{\partial
s}\right)\wedge\overline{\partial}\psi\wedge\omega^{n-3}_{\psi}\right.\\
&+&\partial\psi\wedge\partial\omega\wedge\overline{\partial}\left(\frac{\partial^{2}\psi}{\partial
t\partial
s}\right)\wedge\overline{\partial}\psi\wedge\omega^{n-3}_{\psi}+\partial\psi\wedge\partial\omega\wedge\overline{\partial}\left(\frac{\partial\psi}{\partial
s}\right)
\wedge\overline{\partial}\left(\frac{\partial\psi}{\partial
t}\right)\wedge\omega^{n-3}_{\psi} \\
&+&\left.\partial\psi\wedge\partial\omega\wedge\overline{\partial}\left(\frac{\partial\psi}{\partial
s}\right)\wedge\overline{\partial}\psi\wedge(n-3)\omega^{n-4}_{\psi}\wedge\sqrt{-1}\partial\overline{\partial}\left(\frac{\partial\psi}{\partial
t}\right)\right] \\
&-&\int_{X}\left[\partial\left(\frac{\partial\psi}{\partial
s}\right)\wedge\partial\omega\wedge\overline{\partial}\left(\frac{\partial\psi}{\partial
t}\right)\wedge\overline{\partial}\psi\wedge\omega^{n-3}_{\psi}\right.
\\
&+&\partial\psi\wedge\partial\omega\wedge\overline{\partial}\left(\frac{\partial^{2}\psi}{\partial
s\partial
t}\right)\wedge\overline{\partial}\psi\wedge\omega^{n-3}_{\psi}+\partial\psi\wedge\partial\omega\wedge\overline{\partial}\left(\frac{\partial\psi}{\partial
t}\right)\wedge\overline{\partial}\left(\frac{\partial\psi}{\partial
s}\right)\wedge\omega^{n-3}_{\psi} \\
&+&\left.\partial\psi\wedge\partial\omega\wedge\overline{\partial}\left(\frac{\partial\psi}{\partial
t}\right)\wedge\overline{\partial}\psi\wedge(n-3)\omega^{n-4}_{\psi}\wedge\sqrt{-1}\partial\overline{\partial}\left(\frac{\partial\psi}{\partial
s}\right)\right].
\end{eqnarray*}
The second term and the sixth term cancel with each other, and the third term and the
seventh term are the same, so we have
\begin{eqnarray*}
\frac{I^{3}}{a_{3}}&=&\int_{X}-\partial\left(\frac{\partial\psi}{\partial
t}\right)\wedge\overline{\partial}\left(\frac{\partial\psi}{\partial
s}\right)\wedge\partial\omega\wedge\overline{\partial}\psi\wedge\omega^{n-3}_{\psi}
\\
&+&\int_{X}\partial\left(\frac{\partial\psi}{\partial
s}\right)\wedge\overline{\partial}\left(\frac{\partial\psi}{\partial
t}\right)\wedge\partial\omega\wedge\overline{\partial}\psi\wedge\omega^{n-3}_{\psi}
\\
&+&2\int_{X}\partial\psi\wedge\partial\omega\wedge\overline{\partial}\left(\frac{\partial\psi}{\partial
s}\right)\wedge\overline{\partial}\left(\frac{\partial\psi}{\partial
t}\right)\wedge\omega^{n-3}_{\psi} \\
&+&\int_{X}(n-3)\partial\psi\wedge\overline{\partial}\psi\wedge\partial\omega\wedge\overline{\partial}\left(\frac{\partial\psi}{\partial
s}\right)\wedge\omega^{n-4}_{\psi}\wedge\sqrt{-1}\partial\overline{\partial}\left(\frac{\partial\psi}{\partial
t}\right) \\
&-&\int_{X}(n-3)\partial\psi\wedge\overline{\partial}\psi\wedge\partial\omega\wedge\overline{\partial}\left(\frac{\partial\psi}{\partial
t}\right)\wedge\omega^{n-4}_{\psi}\wedge\sqrt{-1}\partial\overline{\partial}\left(\frac{\partial\psi}{\partial
s}\right) \\
&=&H_{1}+\frac{2}{-(n-2)\sqrt{-1}}(A_{1}+B_{1})+A_{2}+B_{2},
\end{eqnarray*}
where
\begin{eqnarray}
H_{1}&:=&\int_{X}-\partial\left(\frac{\partial\psi}{\partial
t}\right)\wedge\overline{\partial}\left(\frac{\partial\psi}{\partial
s}\right)\wedge\partial\omega\wedge\overline{\partial}\psi\wedge\omega^{n-3}_{\psi}
\label{2.40}\\
&+&\int_{X}\partial\left(\frac{\partial\psi}{\partial
s}\right)\wedge\overline{\partial}\left(\frac{\partial\psi}{\partial
t}\right)\wedge\partial\omega\wedge\overline{\partial}\psi\wedge\omega^{n-3}_{\psi}
\nonumber\\
A_{2}&:=&\int_{X}(n-3)\partial\psi\wedge\overline{\partial}\psi\wedge\partial\omega\wedge\overline{\partial}\left(\frac{\partial\psi}{\partial
s}\right)\label{2.41}\\
&&\wedge\omega^{n-4}_{\psi}\wedge\sqrt{-1}\partial\overline{\partial}\left(\frac{\partial\psi}{\partial
t}\right) \nonumber\\
B_{2}&:=&-\int_{X}(n-3)\partial\psi\wedge\overline{\partial}\psi\wedge\partial\omega\wedge\overline{\partial}\left(\frac{\partial\psi}{\partial
t}\right)\label{2.42}\\
&&\wedge\omega^{n-4}_{\psi}\wedge\sqrt{-1}\partial\overline{\partial}\left(\frac{\partial\psi}{\partial
s}\right).\nonumber
\end{eqnarray}
Similarly,
\begin{equation}
\frac{I^{4}}{a_{4}}=\overline{H_{1}}+\frac{2}{(n-2)\sqrt{-1}
}(\overline{A_{1}}+\overline{B_{1}})+\overline{A_{2}}+\overline{B_{2}}.\label{2.43}
\end{equation}

The hard part is to find some suitable expression of $H_{1}$. In the following
we will see that $H_{1}+\overline{H_{1}}$ has a nice form which contains only
$A_{1},B_{1}, \overline{A_{1}}$, and $\overline{B_{1}}$.

Now we compute $H_{1}$, using (\ref{2.3}) and (\ref{2.4}):
\begin{eqnarray*}
H_{1}&=&\int_{X}\partial\left[\overline{\partial}\left(\frac{\partial\psi}{\partial
s}\right)\wedge\partial\omega\wedge\overline{\partial}\psi\wedge\omega^{n-3}_{\psi}\right]\frac{\partial\psi}{\partial
t}\\
&+&\int_{X}\overline{\partial}\left[\partial\left(\frac{\partial\psi}{\partial
s}\right)\wedge\partial\omega\wedge\overline{\partial}\psi\wedge\omega^{n-3}_{\psi}\right]\frac{\partial\psi}{\partial
t}\\
&=&\int_{X}\frac{\partial\psi}{\partial
t}\left[\partial\overline{\partial}\left(\frac{\partial\psi}{\partial
s}\right)\wedge\partial\omega\wedge\overline{\partial}\psi\wedge\omega^{n-3}_{\psi}-\overline{\partial}\left(\frac{\partial\psi}{\partial
s}\right)\wedge\partial\left(\partial\omega\wedge\overline{\partial}\psi\wedge\omega^{n-3}_{\psi}\right)\right]
\\
&+&\int_{X}\frac{\partial\psi}{\partial
t}\left[\overline{\partial}\partial\left(\frac{\partial\psi}{\partial
s}\right)\wedge\partial\omega\wedge\overline{\partial}\psi\wedge\omega^{n-3}_{\psi}-\partial\left(\frac{\partial\psi}{\partial
s}\right)\wedge\overline{\partial}\left(\partial\omega\wedge\overline{\partial}\psi\wedge\omega^{n-3}_{\psi}\right)\right]
\\
&=&\int_{X}\frac{\partial\psi}{\partial
t}\cdot\overline{\partial}\left(\frac{\partial\psi}{\partial
s}\right)\wedge\partial\omega\wedge\left(\partial
\overline{\partial}\psi\wedge\omega^{n-3}_{\psi}
-\overline{\partial}\psi\wedge(n-3)\omega^{n-4}_{\psi}\wedge\partial\omega\right)
\\
&+&\int_{X}-\frac{\partial\psi}{\partial
t}\cdot\partial\left(\frac{\partial\psi}{\partial
s}\right)\wedge\overline{\partial}\partial\omega\wedge\overline{\partial}\psi\wedge\omega^{n-3}_{\psi}\\
&+&\int_{X}-\frac{\partial\psi}{\partial
t}\cdot\partial\left(\frac{\partial\psi}{\partial
s}\right)\wedge\partial\omega\wedge\overline{\partial}\psi\wedge(n-3)\omega^{n-4}_{\psi}\wedge\overline{\partial}\omega
\\
&=&\int_{X}\overline{\partial}\left(\frac{\partial\psi}{\partial
t}\right)\wedge\overline{\partial}\left(\frac{\partial\psi}{\partial
s}\right)\wedge\partial\omega\wedge\omega^{n-3}_{\psi}\wedge\partial\psi
\\
&+&\int_{X}\frac{\partial\psi}{\partial
t}\cdot-\overline{\partial}\left(\frac{\partial\psi}{\partial
s}\right)\wedge\left[\overline{\partial}\partial\omega\wedge\omega^{n-3}_{\psi}
-\partial\omega\wedge(n-3)\omega^{n-4}_{\psi}\wedge\overline{\partial}\omega\right]\wedge\partial\psi
\\
&-&\int_{X}\frac{\partial\psi}{\partial
t}\cdot\partial\left(\frac{\partial\psi}{\partial
s}\right)\wedge\overline{\partial}\partial\omega\wedge\overline{\partial}\psi\wedge\omega^{n-3}_{\psi}\\
&-&\int_{X}\frac{\partial\psi}{\partial
t}\cdot\partial\left(\frac{\partial\psi}{\partial
s}\right)\wedge\partial\omega\wedge\overline{\partial}\psi\wedge(n-3)\omega^{n-4}_{\psi}\wedge\overline{\partial}\omega
\end{eqnarray*}
Therefore we obtain
\begin{eqnarray}
H_{1}&=&\int_{X}\partial\psi\wedge\partial\omega\wedge\overline{\partial}\left(\frac{\partial\psi}{\partial
s}\right)\wedge\overline{\partial}\left(\frac{\partial\psi}{\partial
t}\right)\wedge\omega^{n-3}_{\psi} \label{2.44}\\
&-&\int_{X}\frac{\partial\psi}{\partial
t}\cdot\overline{\partial}\left(\frac{\partial\psi}{\partial
s}\right)\wedge\overline{\partial}\partial\omega\wedge\partial\psi\wedge\omega^{n-3}_{\psi}
\nonumber\\
&+&\int_{X}\frac{\partial\psi}{\partial
t}\cdot\overline{\partial}\left(\frac{\partial\psi}{\partial
s}\right)\wedge\partial\omega\wedge(n-3)\omega^{n-4}_{\psi}\wedge\overline{\partial}\omega\wedge\partial\psi\nonumber\\
&-&\int_{X}\frac{\partial\psi}{\partial
t}\cdot\partial\left(\frac{\partial\psi}{\partial
s}\right)\wedge\overline{\partial}\partial\omega\wedge\overline{\partial}\psi\wedge\omega^{n-3}_{\psi}
\nonumber\\
&+&\int_{X}\frac{\partial\psi}{\partial
t}\cdot\partial\left(\frac{\partial\psi}{\partial
s}\right)\wedge\partial\omega\wedge(n-3)\omega^{n-4}_{\psi}\wedge\overline{\partial}\omega\wedge\overline{\partial}\psi,\nonumber
\end{eqnarray}
and, taking the complex conjugate yields, using (\ref{2.3}) and (\ref{2.4})
\begin{eqnarray}
\overline{H_{1}}&=&\int_{X}\overline{\partial}\psi\wedge\overline{\partial}\omega\wedge\partial\left(\frac{\partial\psi}{\partial
s}\right)\wedge\partial\left(\frac{\partial\psi}{\partial
t}\right)\wedge\omega^{n-3}_{\psi} \label{2.45}\\
&-&\int_{X}\frac{\partial\psi}{\partial
t}\cdot\partial\left(\frac{\partial\psi}{\partial
s}\right)\wedge\partial\overline{\partial}\omega\wedge\overline{\partial}\psi\wedge\omega^{n-3}_{\psi}
\nonumber\\
&-&\int_{X}\frac{\partial\psi}{\partial
t}\cdot\overline{\partial}\left(\frac{\partial\psi}{\partial
s}\right)\wedge\partial\overline{\partial}\omega\wedge\partial\psi\wedge\omega^{n-3}_{\psi}
\nonumber\\
&+&\int_{X}\frac{\partial\psi}{\partial
t}\cdot\partial\left(\frac{\partial\psi}{\partial
s}\right)\wedge\overline{\partial}\omega\wedge(n-3)\omega^{n-4}_{\psi}\wedge\partial\omega\wedge\overline{\partial}\psi
\nonumber\\
&+&\int_{X}\frac{\partial\psi}{\partial
t}\cdot\overline{\partial}\left(\frac{\partial\psi}{\partial
s}\right)\wedge\overline{\partial}\omega\wedge(n-3)\omega^{n-4}_{\psi}\wedge\partial\omega\wedge\partial\psi.\nonumber
\end{eqnarray}
Adding (\ref{2.45}) to (\ref{2.45}), it follows that
\begin{equation}
H_{1}+\overline{H_{1}}=\frac{A_{1}+B_{1}}{-(n-2)\sqrt{-1}}+
\frac{\overline{A_{1}}+\overline{B_{1}}}{(n-2)\sqrt{-1}}.\label{2.46}
\end{equation}
and, hence,
\begin{eqnarray}
\frac{I^{3}}{a_{3}}+\frac{I^{4}}{a_{4}}&=&H_{1}+\overline{H_{1}}+\frac{2}{-(n-2)\sqrt{-1}}[(A_{1}+B_{1})-(\overline{A_{1}}+\overline{B_{1}})]\nonumber\\
&+&A_{2}+B_{2}+\overline{A_{2}}+\overline{B_{2}}
\nonumber\\
&=&\frac{3}{-(n-2)\sqrt{-1}}[(A_{1}+B_{1})-(\overline{A_{1}}+\overline{B_{1}})]+(A_{2}+B_{2})+(\overline{A_{2}}+\overline{B_{2}}).\label{2.47}
\end{eqnarray}
Set
\begin{equation}
c_{1}:=A_{1}+B_{1}-(\overline{A_{1}}+\overline{B_{1}}), \ \ \ c_{2}:=A_{2}+B_{2}+\overline{A_{2}}
+\overline{B_{2}}.
\end{equation}
we deduce
\begin{equation}
\frac{2I^{0}}{n(n-1)\sqrt{-1}}=\frac{I^{1}}{a_{1}}-\frac{I^{2}}{a_{2}}+c_{1}, \ \ \
\frac{I^{3}}{a_{3}}+\frac{I^{4}}{a_{4}}=\frac{3}{-(n-2)\sqrt{-1}}c_{1}+c_{2}.\label{2.49}
\end{equation}

\subsection{The constructions of $I^{5}$ and $I^{6}$}
To give the general construction of $I^{2i+1}$ and $I^{2i+2}$, we firstly consider some
special cases. In this subsection we give the construction of $I^{5}$ and $I^{6}$,
and in the next subsection the construction of $I^{7}$ and $I^{8}$. Finally, we give
the general construction.

From (\ref{2.41}), it follows that
\begin{eqnarray*}
A_{2}&=&\int_{X}(n-3)\partial\psi\wedge\overline{\partial}\psi\wedge\overline{\partial}
\left(\frac{\partial\psi}{\partial
s}\right)\wedge\omega^{n-4}_{\psi}\wedge\partial\omega\wedge-\sqrt{-1}\partial\overline{\partial}\left(\frac{\partial\psi}{\partial
t}\right) \\
&=&\int_{X}\sqrt{-1}\partial\left[(n-3)\partial\psi\wedge\overline{\partial}\psi
\wedge\overline{\partial}\left(\frac{\partial\psi}{\partial
s}\right)\wedge\omega^{n-4}_{\psi}\wedge\partial\omega\right]\wedge\overline{\partial}
\left(\frac{\partial\psi}{\partial t}\right) \\
&=&\int_{X}\sqrt{-1}(n-3)\left[\partial\left(\partial\psi\wedge\overline{\partial}\psi\wedge
\overline{\partial}\left(\frac{\partial\psi}{\partial
s}\right)\right)\wedge\omega^{n-4}_{\psi}\wedge\partial\omega\right]\wedge
\overline{\partial}\left(\frac{\partial\psi}{\partial t}\right)\\
&=&\int_{X}-\sqrt{-1}(n-3)\left[\partial\psi\wedge\left(\partial\overline{\partial}\psi
\wedge\overline{\partial}\left(\frac{\partial\psi}{\partial
s}\right)-\overline{\partial}\psi\wedge\partial\overline{\partial}\left(\frac{\partial\psi}{\partial
s}\right)\right)\right]\\
&&\wedge\omega^{n-4}_{\psi}\wedge\partial\omega\wedge
\overline{\partial}\left(\frac{\partial\psi}{\partial t}\right) \\
&=&\int_{X}(n-3)\partial\psi\wedge-\sqrt{-1}\partial\overline{\partial}\psi\wedge
\overline{\partial}\left(\frac{\partial\psi}{\partial
s}\right)\wedge\omega^{n-4}_{\psi}\wedge\partial\omega\wedge\overline{\partial}\left(\frac{\partial\psi}{\partial
t}\right)\\
&+&\int_{X}(n-3)\partial\psi\wedge\overline{\partial}\psi\wedge\sqrt{-1}
\partial\overline{\partial}\left(\frac{\partial\psi}{\partial
s}\right)\wedge\omega^{n-4}_{\psi}\wedge\partial\omega\wedge\overline{\partial}
\left(\frac{\partial\psi}{\partial t}\right) \\
&=&\int_{X}(n-3)\partial\psi\wedge\partial\omega\wedge\overline{\partial}
\left(\frac{\partial\psi}{\partial
s}\right)\wedge\overline{\partial}\left(\frac{\partial\psi}{\partial
t}\right)\wedge\omega^{n-4}_{\psi}\wedge\sqrt{-1}\partial\overline{\partial}\psi-B_{2};
\end{eqnarray*}
hence, by the definition (\ref{2.42}), we have
\begin{equation}
A_{2}+B_{2}=\int_{X}(n-3)\partial\psi\wedge\partial\omega\wedge\overline{\partial}
\left(\frac{\partial\psi}{\partial
s}\right)\wedge\overline{\partial}\left(\frac{\partial\psi}{\partial
t}\right)\wedge\omega^{n-4}_{\psi}\wedge\sqrt{-1}\partial\overline{\partial}\psi.\label{2.50}
\end{equation}
Motivated by (\ref{2.50}), we set
\begin{equation}
\mathcal{L}^{5}_{\omega}(\varphi',\varphi''):=\frac{1}{V_{\omega}}\int^{1}_{0}
\int_{X}a_{5}\partial\varphi_{t}\wedge\partial\omega\wedge\overline{\partial}\dot{\varphi}_{t}
\wedge\overline{\partial}\varphi_{t}\wedge\omega^{n-4}_{\varphi_{t}}\wedge\sqrt{-1}\partial\overline{\partial}\varphi_{t},
\label{2.51}\\
\end{equation}
\begin{equation}
\mathcal{L}^{6}_{\omega}(\varphi',\varphi''):=\frac{1}{V_{\omega}}\int^{1}_{0}\int_{X}a_{6}
\overline{\partial}\varphi_{t}\wedge\overline{\partial}\omega\wedge\partial\dot{\varphi}_{t}\wedge
\partial\varphi_{t}\wedge\omega^{n-4}_{\varphi_{t}}\wedge\sqrt{-1}
\partial\overline{\partial}\varphi_{t}.\label{2.52}
\end{equation}
Consider again the $1$-forms
\begin{eqnarray}
\Psi^{5}&=&\left[\int_{X}a_{5}\partial\psi\wedge\partial\omega\wedge\overline{\partial}\left(\frac{\partial\psi}{\partial
s}\right)\wedge\overline{\partial}\psi\wedge\omega^{n-4}_{\psi}\wedge\sqrt{-1}\partial\overline{\partial}\psi\right]ds
\nonumber\\
&+&\left[\int_{X}a_{5}\partial\psi\wedge\partial\omega\wedge\overline{\partial}
\left(\frac{\partial\psi}{\partial
t}\right)\wedge\overline{\partial}\psi\wedge\omega^{n-4}_{\psi}\wedge\sqrt{-1}\partial\overline{\partial}\psi\right]dt, \label{2.53} \\
\Psi^{6}&=&\left[\int_{X}a_{6}\overline{\partial}\psi\wedge\overline{\partial}\omega
\wedge\partial\left(\frac{\partial\psi}{\partial s}\right)\wedge\partial\psi\wedge
\omega^{n-4}_{\psi}\wedge\sqrt{-1}\partial\overline{\partial}\psi\right]ds\nonumber \\
&+&\left[\int_{X}a_{6}\overline{\partial}\psi\wedge\overline{\partial}\omega
\wedge\partial\left(\frac{\partial\psi}{\partial t}\right)\wedge\partial\psi
\wedge\omega^{n-4}_{\psi}\wedge\sqrt{-1}\partial\overline{\partial}\psi\right]dt.\label{2.54}
\end{eqnarray}
The differential of $\Psi^{5}$ is given by
\begin{equation}
d\Psi^{5}=I^{5}\cdot dt\wedge ds,\label{2.55}
\end{equation}
where
\begin{eqnarray}
\frac{I^{5}}{a_{5}}&:=&\int_{X}\frac{\partial}{\partial
t}\left[\partial\psi\wedge\partial\omega\wedge\overline{\partial}\left(\frac{\partial\psi}{\partial
s}\right)\wedge\overline{\partial}\psi\wedge\omega^{n-4}_{\psi}\wedge\sqrt{-1}
\partial\overline{\partial}\psi\right]\nonumber\\
&-&\int_{X}\frac{\partial}{\partial
s}\left[\partial\psi\wedge\partial\omega\wedge\overline{\partial}\left(\frac{\partial\psi}{\partial
t}\right)\wedge\overline{\partial}\psi\wedge\omega^{n-4}_{\psi}
\wedge\sqrt{-1}\partial\overline{\partial}\psi\right].\label{2.56}
\end{eqnarray}
Hence, using (\ref{2.50}),
\begin{eqnarray*}
\frac{I^{5}}{a_{5}}&=&\int_{X}\left[\partial\left(\frac{\partial\psi}{\partial
t}\right)\wedge\partial\omega\wedge\overline{\partial}\left(\frac{\partial\psi}{\partial
s}\right)\wedge\overline{\partial}\psi\wedge\omega^{n-4}_{\psi}\wedge\sqrt{-1}\partial\overline{\partial}\psi\right.
\\
&+&\partial\psi\wedge\partial\omega\wedge\overline{\partial}\left(\frac{\partial^{2}\psi}{\partial
t\partial
s}\right)\wedge\overline{\partial}\psi\wedge\omega^{n-4}_{\psi}\wedge\sqrt{-1}\partial\overline{\partial}\psi
\\
&+&\partial\psi\wedge\partial\omega\wedge\overline{\partial}\left(\frac{\partial\psi}{\partial
s}\right)\wedge\overline{\partial}\left(\frac{\partial\psi}{\partial
t}\right)\wedge\omega^{n-4}_{\psi}\wedge\sqrt{-1}\partial\overline{\partial}\psi
\\
&+&\partial\psi\wedge\partial\omega\wedge\overline{\partial}\left(\frac{\partial\psi}{\partial
s}\right)\wedge\overline{\partial}\psi\wedge(n-4)\omega^{n-5}_{\psi}\wedge\sqrt{-1}\partial\overline{\partial}\left(\frac{\partial\psi}{\partial
t}\right)\wedge\sqrt{-1}\partial\overline{\partial}\psi \\
&+&\left.\partial\psi\wedge\partial\omega\wedge\overline{\partial}\left(\frac{\partial\psi}{\partial
s}\right)\wedge\overline{\partial}\psi\wedge\omega^{n-4}_{\psi}\wedge\sqrt{-1}\partial\overline{\partial}\left(\frac{\partial\psi}{\partial
t}\right)\right] \\
&-&\int_{X}\left[\partial\left(\frac{\partial\psi}{\partial
s}\right)\wedge\partial\omega\wedge\overline{\partial}\left(\frac{\partial\psi}{\partial
t}\right)\wedge\overline{\partial}\psi\wedge\omega^{n-4}_{\psi}\wedge\sqrt{-1}\partial\overline{\partial}\psi\right.
\\
&+&\partial\psi\wedge\partial\omega\wedge\overline{\partial}\left(\frac{\partial^{2}\psi}{\partial
s\partial
t}\right)\wedge\overline{\partial}\psi\wedge\omega^{n-4}_{\psi}\wedge\sqrt{-1}\partial\overline{\partial}\psi
\\
&+&\partial\psi\wedge\partial\omega\wedge\overline{\partial}\left(\frac{\partial\psi}{\partial
t}\right)\wedge\overline{\partial}\left(\frac{\partial\psi}{\partial
s}\right)\wedge\omega^{n-4}_{\psi}\wedge\sqrt{-1}\partial\overline{\partial}\psi
\\
&+&\partial\psi\wedge\partial\omega\wedge\overline{\partial}\left(\frac{\partial\psi}{\partial
t}\right)\wedge\overline{\partial}\psi\wedge(n-4)\omega^{n-5}_{\psi}\wedge
\sqrt{-1}\partial\overline{\partial}\left(\frac{\partial\psi}{\partial
s}\right)\wedge\sqrt{-1}\partial\overline{\partial}\psi \\
&+&\left.\partial\psi\wedge\partial\omega\wedge\overline{\partial}\left(\frac{\partial\psi}{\partial
t}\right)\wedge\overline{\partial}\psi\wedge\omega^{n-4}_{\psi}\wedge\sqrt{-1}\partial\overline{\partial}\left(\frac{\partial\psi}{\partial
s}\right)\right] \\
&=&\int_{X}-\partial\left(\frac{\partial\psi}{\partial
t}\right)\wedge\overline{\partial}\left(\frac{\partial\psi}{\partial
s}\right)\wedge\partial\omega\wedge\overline{\partial}\psi\wedge\omega^{n-4}_{\psi}\wedge\sqrt{-1}\partial\overline{\partial}\psi
\\
&+&\int_{X}\partial\left(\frac{\partial\psi}{\partial
s}\right)\wedge\overline{\partial}\left(\frac{\partial\psi}{\partial
t}\right)\wedge\partial\omega\wedge\overline{\partial}\psi\wedge\omega^{n-4}_{\psi}\wedge\sqrt{-1}\partial\overline{\partial}\psi
\\
&+&2\int_{X}\partial\psi\wedge\partial\omega\wedge\overline{\partial}\left(\frac{\partial\psi}{\partial
s}\right)\wedge\overline{\partial}\left(\frac{\partial\psi}{\partial
t}\right)\wedge\omega^{n-4}_{\psi}\wedge\sqrt{-1}\partial\overline{\partial}\psi
\\
&+&\int_{X}\partial\psi\wedge\partial\omega\wedge\overline{\partial}\left(\frac{\partial\psi}{\partial
s}\right)\wedge\overline{\partial}\psi\wedge(n-4)\omega^{n-5}_{\psi}\wedge\sqrt{-1}\partial\overline{\partial}
\left(\frac{\partial\psi}{\partial
t}\right)\wedge\sqrt{-1}\partial\overline{\partial}\psi \\
&-&\int_{X}\partial\psi\wedge\partial\omega\wedge\overline{\partial}\left(\frac{\partial\psi}{\partial
t}\right)\wedge\overline{\partial}\psi\wedge(n-4)\omega^{n-5}_{\psi}\wedge\sqrt{-1}\partial\overline{\partial}\left(\frac{\partial\psi}{\partial
s}\right)\wedge\sqrt{-1}\partial\overline{\partial}\psi\\
&+&\frac{A_{2}}{n-3}+\frac{B_{2}}{n-3} \\
&=&H_{2}+\frac{3}{n-3}(A_{2}+B_{2})+A_{3}+B_{3},
\end{eqnarray*}
where
\begin{eqnarray}
H_{2}&:=&\int_{X}-\partial\left(\frac{\partial\psi}{\partial
t}\right)\wedge\overline{\partial}\left(\frac{\partial\psi}{\partial
s}\right)\wedge\partial\omega\wedge\overline{\partial}\psi\wedge\omega^{n-4}_{\psi}\wedge\sqrt{-1}\partial\overline{\partial}\psi
\nonumber\\
&+&\int_{X}\partial\left(\frac{\partial\psi}{\partial
s}\right)\wedge\overline{\partial}\left(\frac{\partial\psi}{\partial
t}\right)\wedge\partial\omega\wedge\overline{\partial}\psi\wedge\omega^{n-4}_{\psi}\wedge\sqrt{-1}\partial\overline{\partial}\psi,
\label{2.57}\\
A_{3}&:=&\int_{X}(n-4)\partial\psi\wedge\overline{\partial}\psi\wedge\partial\omega
\wedge\overline{\partial}\left(\frac{\partial\psi}{\partial
s}\right)\wedge\omega^{n-5}_{\psi}\wedge\sqrt{-1}\partial\overline{\partial}\left(\frac{\partial\psi}{\partial
t}\right)\nonumber\\
&&\wedge\sqrt{-1}\partial\overline{\partial}\psi, \label{2.58}\\
B_{3}&:=&\int_{X}(n-4)\partial\psi\wedge\overline{\partial}\psi\wedge\partial\omega
\wedge\overline{\partial}\left(\frac{\partial\psi}{\partial
t}\right)\wedge\omega^{n-5}_{\psi}\wedge-\sqrt{-1}\partial\overline{\partial}
\left(\frac{\partial\psi}{\partial
s}\right)\nonumber \\
&&\wedge\sqrt{-1}\partial\overline{\partial}\psi.\label{2.59}
\end{eqnarray}
Similarly, for
\begin{equation}
d\Psi^{6}=:I^{6}\cdot dt\wedge ds,\label{2.60}
\end{equation}
we have
\begin{equation}
\frac{I^{6}}{a_{6}}=\overline{H_{2}}+\frac{3}{n-3}(\overline{A_{2}}+
\overline{B_{2}})+\overline{A_{3}}+\overline{B_{3}}.\label{2.61}
\end{equation}
As (\ref{2.44}), the hard part $H_{2}$ is calculated as follows:
\begin{eqnarray*}
H_{2}&=&\int_{X}\partial\left[\overline{\partial}\left(\frac{\partial\psi}{\partial
s}\right)\wedge\partial\omega\wedge\overline{\partial}\psi\wedge\omega^{n-4}_{\psi}
\wedge\sqrt{-1}\partial\overline{\partial}\psi\right]\frac{\partial\psi}{\partial
t} \\
&+&\int_{X}\overline{\partial}\left[\partial\left(\frac{\partial\psi}{\partial
s}\right)\wedge\partial\omega\wedge\overline{\partial}\psi\wedge\omega^{n-4}_{\psi}
\wedge\sqrt{-1}\partial\overline{\partial}\psi\right]\frac{\partial\psi}{\partial
t}\\
&=&\int_{X}\frac{\partial\psi}{\partial
t}\left[\partial\overline{\partial}\left(\frac{\partial\psi}{\partial
s}\right)\wedge\partial\omega\wedge\overline{\partial}\psi\wedge\omega^{n-4}_{\psi}\wedge\sqrt{-1}\partial
\overline{\partial}\psi\right. \\
&-&\left.\overline{\partial}\left(\frac{\partial\psi}{\partial
s}\right)\wedge\partial\left(\partial\omega\wedge\overline{\partial}\psi\wedge\omega^{n-4}_{\psi}
\wedge\sqrt{-1}\partial\overline{\partial}\psi\right)\right] \\
&+&\int_{X}\frac{\partial\psi}{\partial
t}\left[\overline{\partial}\partial\left(\frac{\partial\psi}{\partial
s}\right)\wedge\partial\omega\wedge\overline{\partial}\psi\wedge\omega^{n-4}_{\psi}
\wedge\sqrt{-1}\partial\overline{\partial}\psi\right. \\
&-&\left.\partial\left(\frac{\partial\psi}{\partial s}\right)\wedge
\overline{\partial}\left(\partial\omega\wedge\overline{\partial}\psi\wedge
\omega^{n-4}_{\psi}\wedge\sqrt{-1}\partial\overline{\partial}\psi\right)\right]\\
&=&\int_{X}\frac{\partial\psi}{\partial
t}\cdot\overline{\partial}\left(\frac{\partial\psi}{\partial
s}\right)\wedge\partial\omega\wedge\partial\overline{\partial}\psi\wedge\omega^{n-4}_{\psi}\wedge\sqrt{-1}\partial\overline{\partial}\psi
\\
&-&\int_{X}\frac{\partial\psi}{\partial
t}\cdot\partial\left(\frac{\partial\psi}{\partial
s}\right)\wedge\overline{\partial}\partial\omega\wedge\overline{\partial}\psi\wedge\omega^{n-4}_{\psi}\wedge\sqrt{-1}\partial\overline{\partial}\psi
\\
&-&\int_{X}\frac{\partial\psi}{\partial
t}\cdot\partial\left(\frac{\partial\psi}{\partial
s}\right)\wedge\partial\omega\wedge\overline{\partial}\psi\wedge(n-4)
\omega^{n-5}_{\psi}\wedge\overline{\partial}\omega\wedge\sqrt{-1}\partial\overline{\partial}\psi
\\
&=&\int_{X}\sqrt{-1}\overline{\partial}\left[\frac{\partial\psi}{\partial
t}\cdot\overline{\partial}\left(\frac{\partial\psi}{\partial
s}\right)\wedge\partial\omega\wedge\partial\overline{\partial}\psi\wedge\omega^{n-4}_{\psi}\right]\wedge
\partial\psi \\
&-&\int_{X}\frac{\partial\psi}{\partial
t}\cdot\partial\left(\frac{\partial\psi}{\partial
s}\right)\wedge\overline{\partial}\partial\omega\wedge\overline{\partial}\psi\wedge\omega^{n-4}_{\psi}\wedge\sqrt{-1}\partial\overline{\partial}\psi
\\
&-&\int_{X}\frac{\partial\psi}{\partial
t}\cdot\partial\left(\frac{\partial\psi}{\partial
s}\right)\wedge\partial\omega\wedge\overline{\partial}\psi\wedge(n-4)
\omega^{n-5}_{\psi}\wedge\overline{\partial}\omega\wedge\sqrt{-1}\partial\overline{\partial}\psi.
\end{eqnarray*}
Hence
\begin{eqnarray}
H_{2}&=&\int_{X}\overline{\partial}\left(\frac{\partial\psi}{\partial
s}\right)\wedge\overline{\partial}\left(\frac{\partial\psi}{\partial
t}\right)\wedge\partial\psi\wedge\partial\omega\wedge\omega^{n-4}_{\psi}\wedge\sqrt{-1}\partial
\overline{\partial}\psi \label{2.62}\\
&-&\int_{X}\frac{\partial\psi}{\partial
t}\cdot\overline{\partial}\left(\frac{\partial\psi}{\partial
s}\right)\wedge\partial\psi\wedge\overline{\partial}\partial\omega\wedge
\omega^{n-4}_{\psi}\wedge\sqrt{-1}\partial\overline{\partial}\psi \nonumber\\
&+&\int_{X}\frac{\partial\psi}{\partial
t}\cdot\overline{\partial}\left(\frac{\partial\psi}{\partial
s}\right)\wedge\partial\omega\wedge\overline{\partial}\omega\wedge(n-4)
\omega^{n-5}_{\psi}\wedge\sqrt{-1}\partial\overline{\partial}\psi\wedge\partial\psi
\nonumber\\
&-&\int_{X}\frac{\partial\psi}{\partial
t}\cdot\partial\left(\frac{\partial\psi}{\partial
s}\right)\wedge\overline{\partial}\psi\wedge\overline{\partial}\partial\omega
\wedge\omega^{n-4}_{\psi}\wedge\sqrt{-1}\partial\overline{\partial}\psi
\nonumber\\
&+&\int_{X}\frac{\partial\psi}{\partial
t}\cdot\partial\left(\frac{\partial\psi}{\partial
s}\right)\wedge\partial\omega\wedge\overline{\partial}\omega\wedge(n-4)\omega^{n-5}_{\psi}
\wedge\sqrt{-1}\partial\overline{\partial}\psi\wedge\overline{\partial}\psi\nonumber
\end{eqnarray}
and, consequently, after taking the complex conjugate on both sides,
\begin{eqnarray}
\overline{H_{2}}&=&\int_{X}\partial\left(\frac{\partial\psi}{\partial
s}\right)\wedge\partial\left(\frac{\partial\psi}{\partial
t}\right)\wedge\overline{\partial}\psi\wedge\overline{\partial}\omega\wedge\omega^{n-4}_{\psi}\wedge\sqrt{-1}\partial\overline{\partial}\psi
\label{2.63}\\
&-&\int_{X}\frac{\partial\psi}{\partial
t}\cdot\partial\left(\frac{\partial\psi}{\partial
s}\right)\wedge\overline{\partial}\psi\wedge\partial\overline{\partial}\omega\wedge\omega^{n-4}_{\psi}
\wedge\sqrt{-1}\partial\overline{\partial}\psi \nonumber\\
&+&\int_{X}\frac{\partial\psi}{\partial
t}\cdot\partial\left(\frac{\partial\psi}{\partial
s}\right)\wedge\overline{\partial}\omega\wedge\partial\omega\wedge(n-4)
\omega^{n-5}_{\psi}\wedge\sqrt{-1}\partial\overline{\partial}\psi\wedge\overline{\partial}\psi
\nonumber\\
&-&\int_{X}\frac{\partial\psi}{\partial
t}\cdot\overline{\partial}\left(\frac{\partial\psi}{\partial
s}\right)\wedge\partial\psi\wedge\partial\overline{\partial}\omega\wedge\omega^{n-4}_{\psi}
\wedge\sqrt{-1}\partial\overline{\partial}\psi \nonumber\\
&+&\int_{X}\frac{\partial\psi}{\partial
t}\cdot\overline{\partial}\left(\frac{\partial\psi}{\partial
s}\right)\wedge\overline{\partial}\omega\wedge\partial\omega\wedge(n-4)\omega^{n-5}_{\psi}
\wedge\sqrt{-1}\partial\overline{\partial}\psi\wedge\partial\psi.\nonumber
\end{eqnarray}
Therefore
\begin{eqnarray}
H_{2}+\overline{H_{2}}&=&\frac{A_{2}+B_{2}}{n-3}+\frac{\overline{A_{2}}+\overline{B_{2}}}{n-3},
\label{2.64}\\
\frac{I^{5}}{a_{5}}+\frac{I^{6}}{a_{6}}&=&\frac{4}{n-3}(A_{2}+B_{2}+\overline{A_{2}}+\overline{B_{2}})+(A_{3}+B_{3}+\overline{A_{3}}+
\overline{B_{3}}).\label{2.65}
\end{eqnarray}
If we set
\begin{equation}
c_{3}:=A_{3}+B_{3}+\overline{A_{3}}+\overline{B_{3}}\label{2.66}
\end{equation}
we can rewrite (\ref{2.64}) and (\ref{2.65}) as
\begin{equation}
H_{2}+\overline{H_{2}}=\frac{c_{2}}{n-3}, \ \ \ \frac{I^{5}}{a_{5}}+\frac{I^{6}}{a_{6}}
=\frac{4}{n-3}c_{2}+c_{3}.\label{2.67}
\end{equation}

\subsection{The constructions of $I^{7}$ and $I^{8}$}
Recall
\begin{equation*}
A_{3}=\int_{X}\left[(n-4)\partial\psi\wedge\overline{\partial}\psi\wedge\partial\omega\wedge\overline{\partial}
\left(\frac{\partial\psi}{\partial
s}\right)\wedge\omega^{n-5}_{\psi}\wedge\sqrt{-1}\partial\overline{\partial}\psi\right]\wedge\sqrt{-1}\partial\overline{\partial}\left(\frac{\partial\psi}{\partial
t}\right).
\end{equation*}
By a direct computation, one deduces that
\begin{eqnarray*}
A_{3}&=&\int_{X}-\sqrt{-1}\partial\left[(n-4)\partial\psi\wedge\overline{\partial}\psi\wedge\partial\omega\wedge\overline{\partial}\left(\frac{\partial\psi}{\partial
s}\right)\wedge\omega^{n-5}_{\psi}\wedge\sqrt{-1}\partial\overline{\partial}\psi\right]\\
&&\wedge\overline{\partial}\left(\frac{\partial\psi}{\partial
t}\right) \\
&=&\int_{X}\sqrt{-1}(n-4)\partial\psi\wedge\partial\left(\overline{\partial}\psi\wedge\partial\omega\wedge\overline{\partial}\left(\frac{\partial\psi}{\partial
s}\right)\wedge\omega^{n-5}_{\psi}\wedge\sqrt{-1}\partial\overline{\partial}\psi\right)\\
&&\wedge\overline{\partial}\left(\frac{\partial\psi}{\partial
t}\right) \\
&=&\int_{X}\sqrt{-1}(n-4)\partial\psi\wedge\left[\partial\overline{\partial}\psi\wedge\partial\omega\wedge\overline{\partial}\left(\frac{\partial\psi}{\partial
s}\right)\wedge\omega^{n-5}_{\psi}\wedge\sqrt{-1}\partial\overline{\partial}\psi\right.
\\
&+&\left.\overline{\partial}\psi\wedge\partial\omega\wedge\partial\left(\overline{\partial}\left(\frac{\partial\psi}{\partial
s}\right)\wedge\omega^{n-5}_{\psi}\wedge\sqrt{-1}\partial\overline{\partial}\psi\right)\right]\wedge\overline{\partial}\left(\frac{\partial\psi}{\partial
t}\right) \\
&=&\int_{X}\sqrt{-1}(n-4)\partial\psi\wedge\left[\partial\overline{\partial}\psi\wedge\partial\omega\wedge\overline{\partial}\left(\frac{\partial\psi}{\partial
s}\right)\wedge\omega^{n-5}_{\psi}\wedge\sqrt{-1}\partial\overline{\partial}\psi\right.
\\
&+&\left.\overline{\partial}\psi\wedge\partial\omega\wedge\left(\partial\overline{\partial}\left(\frac{\partial\psi}{\partial
s}\right)\wedge\omega^{n-5}_{\psi}\wedge\sqrt{-1}\partial\overline{\partial}\psi\right)\right]\wedge\overline{\partial}\left(\frac{\partial\psi}{\partial
t}\right) \\
&=&\int_{X}(n-4)\partial\psi\wedge\partial\omega\wedge\overline{\partial}\left(\frac{\partial\psi}{\partial
s}\right)\wedge\overline{\partial}\left(\frac{\partial\psi}{\partial
t}\right)\wedge\omega^{n-5}_{\psi}\wedge(\sqrt{-1}\partial\overline{\partial}\psi)^{2}
\\
&+&\int_{X}(n-4)\partial\psi\wedge\overline{\partial}\psi\wedge\partial\omega\wedge\overline{\partial}\left(\frac{\partial\psi}{\partial
t}\right)\wedge\sqrt{-1}\partial\overline{\partial}\left(\frac{\partial\psi}{\partial
s}\right)\wedge\omega^{n-5}_{\psi}\wedge\sqrt{-1}\partial\overline{\partial}\psi
\\
&=&-B_{3}+\int_{X}(n-4)\partial\psi\wedge\partial\omega\wedge\overline{\partial}\left(\frac{\partial\psi}{\partial
s}\right)\wedge\overline{\partial}\left(\frac{\partial\psi}{\partial
t}\right)\wedge\omega^{n-5}_{\psi}\wedge(\sqrt{-1}\partial\overline{\partial}\psi)^{2}.
\end{eqnarray*}
Consequently, it follows that
\begin{equation}
A_{3}+B_{3}=\int_{X}(n-4)\partial\psi\wedge\partial\omega\wedge\overline{\partial}\left(\frac{\partial\psi}{\partial
s}\right)\wedge\overline{\partial}\left(\frac{\partial\psi}{\partial
t}\right)\wedge\omega^{n-5}_{\psi}
\wedge(\sqrt{-1}\partial\overline{\partial}\psi)^{2}.\label{2.68}
\end{equation}
Now we introduce the corresponding functionals
\begin{equation}
\mathcal{L}^{7}_{\omega}(\varphi',\varphi''):=\frac{1}{V_{\omega}}\int^{1}_{0}\int_{X}a_{7}\partial\varphi_{t}\wedge\partial\omega\wedge\overline{\partial}
\dot{\varphi}_{t}\wedge\overline{\partial}\varphi_{t}\wedge\omega^{n-5}_{\varphi_{t}}\wedge(\sqrt{-1}\partial\overline{\partial}\varphi_{t})^{2},
\label{2.69}
\end{equation}
\begin{equation}
\mathcal{L}^{8}_{\omega}(\varphi',\varphi''):=\frac{1}{V_{\omega}}\int^{1}_{0}\int_{X}a_{8}\overline{\partial}\varphi_{t}\wedge\overline{\partial}\omega
\wedge\partial\dot{\varphi}_{t}\wedge\partial
\varphi_{t}\wedge\omega^{n-5}_{\varphi_{t}}\wedge(\sqrt{-1}
\partial\overline{\partial}\varphi_{t})^{2}\label{2.70}
\end{equation}
and consider the $1$-forms
\begin{eqnarray}
\Psi^{7}&=&\left[\int_{X}a_{7}\partial\psi\wedge\partial\omega\wedge\overline{\partial}\left(\frac{\partial\psi}{\partial
s}\right)\wedge\overline{\partial}\psi\wedge\omega^{n-5}_{\psi}\wedge(\sqrt{-1}\partial\overline{\partial}\psi)^{2}\right]ds
\nonumber\\
&+&\left[\int_{X}a_{7}\partial\psi\wedge\partial\omega\wedge\overline{\partial}\left(\frac{\partial\psi}{\partial
t}\right)\wedge\overline{\partial}\psi\wedge\omega^{n-5}_{\psi}\wedge(\sqrt{-1}\partial\overline{\partial}\psi)^{2}\right]dt,
\label{2.71}\\
\Psi^{8}&=&\left[\int_{X}a_{8}\overline{\partial}\psi\wedge\overline{\partial}\omega\wedge\partial\left(\frac{\partial\psi}{\partial
s}\right)\wedge\partial\psi\wedge
\omega^{n-5}_{\psi}\wedge(\sqrt{-1}\partial\overline{\partial}\psi)^{2}\right]ds
\nonumber\\
&+&\left[\int_{X}a_{8}\overline{\partial}\psi\wedge\overline{\partial}\omega\wedge\partial\left(\frac{\partial\psi}{\partial
t}\right)\wedge\partial\psi\wedge
\omega^{n-5}_{\psi}\wedge(\sqrt{-1}\partial\overline{\partial}\psi)^{2}\right]dt.\label{2.72}
\end{eqnarray}
So, we have the expression,
\begin{equation}
d\Psi^{7}=I^{7}\cdot dt\wedge ds, \ \ \ d\Psi^{8}=I^{8}\cdot\label{2.73}
dt\wedge ds,
\end{equation}
where
\begin{eqnarray*}
\frac{I^{7}}{a_{7}}&:=&\int_{X}\frac{\partial}{\partial
t}\left[\partial\psi\wedge\partial\omega\wedge\overline{\partial}\left(\frac{\partial\psi}{\partial
s}\right)\wedge\overline{\partial}\psi\wedge\omega^{n-5}_{\psi}\wedge(\sqrt{-1}\partial\overline{\partial}\psi)^{2}\right]
\\
&-&\int_{X}\frac{\partial}{\partial
s}\left[\partial\psi\wedge\partial\omega\wedge\overline{\partial}\left(\frac{\partial\psi}{\partial
t}\right)\wedge\overline{\partial}\psi\wedge\omega^{n-5}_{\psi}\wedge(\sqrt{-1}\partial\overline{\partial}\psi)^{2}\right]
\\
&=&\int_{X}\left[\partial\left(\frac{\partial\psi}{\partial
t}\right)\wedge\partial\omega\wedge\overline{\partial}\left(\frac{\partial\psi}{\partial
s}\right)\wedge\overline{\partial}\psi\wedge\omega^{n-5}_{\psi}\wedge(\sqrt{-1}\partial\overline{\partial}\psi)^{2}\right.
\\
&+&\partial\psi\wedge\partial\omega\wedge\overline{\partial}\left(\frac{\partial^{2}\psi}{\partial
t\partial
s}\right)\wedge\overline{\partial}\psi\wedge\omega^{n-5}_{\psi}\wedge(\sqrt{-1}\partial\overline{\partial}\psi)^{2}
\\
&+&\partial\psi\wedge\partial\omega\wedge\overline{\partial}\left(\frac{\partial\psi}{\partial
s}\right)\wedge\overline{\partial}\left(\frac{\partial\psi}{\partial
t}\right)\wedge\omega^{n-5}_{\psi}\wedge(\sqrt{-1}\partial\overline{\partial}\psi)^{2}
\\
&+&\partial\psi\wedge\partial\omega\wedge\overline{\partial}\left(\frac{\partial\psi}{\partial
s}\right)\wedge\overline{\partial}\psi\wedge(n-5)\omega^{n-6}_{\psi}\wedge\sqrt{-1}\partial\overline{\partial}\left(\frac{\partial\psi}{\partial
t}\right)\wedge(\sqrt{-1}\partial\overline{\partial}\psi)^{2} \\
&+&\left.\partial\psi\wedge\partial\omega\wedge\overline{\partial}\left(\frac{\partial\psi}{\partial
s}\right)\wedge\overline{\partial}\psi\wedge\omega^{n-5}_{\psi}\wedge2\sqrt{-1}\partial\overline{\partial}\psi\wedge\sqrt{-1}\partial\overline{\partial}\left(\frac{\partial\psi}{\partial
t}\right)\right] \\
&-&\int_{X}\left[\partial\left(\frac{\partial\psi}{\partial
s}\right)\wedge\partial\omega\wedge\overline{\partial}\left(\frac{\partial\psi}{\partial
t}\right)\wedge\overline{\partial}\psi\wedge\omega^{n-5}_{\psi}\wedge(\sqrt{-1}\partial\overline{\partial}\psi)^{2}\right.\\
&+&\partial\psi\wedge\partial\omega\wedge\overline{\partial}\left(\frac{\partial^{2}\psi}{\partial
s\partial
t}\right)\wedge\overline{\partial}\psi\wedge\omega^{n-5}_{\psi}\wedge(\sqrt{-1}\partial\overline{\partial}\psi)^{2}
\\
&+&\partial\psi\wedge\partial\omega\wedge\overline{\partial}\left(\frac{\partial\psi}{\partial
t}\right)\wedge\overline{\partial}\left(\frac{\partial\psi}{\partial
s}\right)\wedge\omega^{n-5}_{\psi}\wedge(\sqrt{-1}\partial\overline{\partial}\psi)^{2}
\\
&+&\partial\psi\wedge\partial\omega\wedge\overline{\partial}\left(\frac{\partial\psi}{\partial
t}\right)\wedge\overline{\partial\psi}\wedge(n-5)\omega^{n-6}_{\psi}\wedge\sqrt{-1}\partial\overline{\partial}\left(\frac{\partial\psi}{\partial
s}\right)\wedge(\sqrt{-1}\partial\overline{\partial}\psi)^{2} \\
&+&\left.\partial\psi\wedge\partial\omega\wedge\overline{\partial}\left(\frac{\partial\psi}{\partial
t}\right)\wedge\overline{\partial}\psi\wedge\omega^{n-5}_{\psi}\wedge2\sqrt{-1}\partial\overline{\partial}\psi\wedge\sqrt{-1}\partial\overline{\partial}
\left(\frac{\partial\psi}{\partial s}\right)\right] \\
&=&\int_{X}-\partial\left(\frac{\partial\psi}{\partial
t}\right)\wedge\overline{\partial}\left(\frac{\partial\psi}{\partial
s}\right)\wedge\partial\omega\wedge\overline{\partial}\psi\wedge\omega^{n-5}_{\psi}\wedge(\sqrt{-1}\partial\overline{\partial}\psi)^{2}
\\
&+&\int_{X}\partial\left(\frac{\partial\psi}{\partial
s}\right)\wedge\overline{\partial}\left(\frac{\partial\psi}{\partial
t}\right)\wedge\partial\omega\wedge\overline{\partial}\psi\wedge\omega^{n-5}_{\psi}\wedge(\sqrt{-1}\partial\partial\overline{\partial}\psi)^{2}
\\
&+&2\int_{X}\partial\psi\wedge\partial\omega\wedge\overline{\partial}\left(\frac{\partial\psi}{\partial
s}\right)\wedge\overline{\partial}\left(\frac{\partial\psi}{\partial
t}\right)\wedge\omega^{n-5}_{\psi}\wedge(\sqrt{-1}\partial\overline{\partial}\psi)^{2}
\\
&+&\int_{X}\partial\psi\wedge\partial\omega\wedge\overline{\partial}\left(\frac{\partial\psi}{\partial
s}\right)\wedge\overline{\partial}\psi\wedge(n-5)\omega^{n-6}_{\psi}\wedge\sqrt{-1}\partial\overline{\partial}\left(\frac{\partial\psi}{\partial
t}\right)\wedge(\sqrt{-1}\partial\overline{\partial}\psi)^{2} \\
&-&\int_{X}\partial\psi\wedge\partial\omega\wedge\overline{\partial}\left(\frac{\partial\psi}{\partial
t}\right)\wedge\overline{\partial}\psi\wedge(n-5)\omega^{n-6}_{\psi}\wedge\sqrt{-1}\partial\overline{\partial}\left(\frac{\partial\psi}{\partial
s}\right)\wedge(\sqrt{-1}\partial\overline{\partial}\psi)^{2} \\
&+&\frac{2}{n-4}(A_{3}+B_{3}) \\
&=&H_{3}+\frac{4}{n-4}(A_{3}+B_{3})+A_{4}+B_{4}
\end{eqnarray*}
where
\begin{eqnarray}
H_{3}&:=&\int_{X}-\partial\left(\frac{\partial\psi}{\partial
t}\right)\wedge\overline{\partial}\left(\frac{\partial\psi}{\partial
s}\right)\wedge\partial\omega\wedge\overline{\partial}\psi\wedge\omega^{n-5}_{\psi}\wedge(\sqrt{-1}\partial\overline{\partial}\psi)^{2}
\nonumber\\
&+&\int_{X}\partial\left(\frac{\partial\psi}{\partial
s}\right)\wedge\overline{\partial}\left(\frac{\partial\psi}{\partial
t}\right)\wedge\partial\omega\wedge\overline{\partial}\psi\wedge
\omega^{n-5}\wedge(\sqrt{-1}\partial\overline{\partial}\psi)^{2}\label{2.74}
\end{eqnarray}
and
\begin{eqnarray}
A_{4}&:=&\int_{X}(n-5)\partial\psi\wedge\overline{\partial}\psi\wedge\partial\omega\wedge\overline{\partial}\left(\frac{\partial\psi}{\partial
s}\right)\wedge\omega^{n-6}_{\psi}\nonumber\\
&&\wedge\sqrt{-1}\partial\overline{\partial}\left(\frac{\partial\psi}{\partial
t}\right)\wedge(\sqrt{-1}\partial\overline{\partial}\psi)^{2}, \label{2.75}\\
B_{4}&:=&\int_{X}(n-5)\partial\psi\wedge\overline{\partial}\psi\wedge\partial\omega\wedge\overline{\partial}\left(\frac{\partial\psi}{\partial
t}\right)\wedge\omega^{n-6}_{\psi}\nonumber\\
&&\wedge-\sqrt{-1}\partial\overline{\partial}\left(\frac{\partial\psi}{\partial
s}\right)\wedge(\sqrt{-1}\partial\overline{\partial}\psi)^{2}.\label{2.76}
\end{eqnarray}

Hence
\begin{equation}
\frac{I^{7}}{a_{7}}=H_{3}+\frac{4}{n-4}(A_{3}+B_{3})+A_{4}+B_{4}.\label{2.77}
\end{equation}

Similarly
\begin{equation}
\frac{I^{8}}{a_{8}}=\overline{H_{3}}+\frac{4}{n-4}(\overline{A_{3}}
+\overline{B_{3}})+\overline{A_{4}}+\overline{B_{4}}.\label{2.78}
\end{equation}

Calculate
\begin{eqnarray*}
H_{3}&=&\int_{X}\partial\left[\overline{\partial}\left(\frac{\partial\psi}{\partial
s}\right)\wedge\partial\omega\wedge\overline{\partial}\psi\wedge\omega^{n-5}_{\psi}\wedge(\sqrt{-1}\partial\overline{\partial}\psi)^{2}\right]\frac{\partial\psi}{\partial
t} \\
&+&\int_{X}\overline{\partial}\left[\partial\left(\frac{\partial\psi}{\partial
s}\right)\wedge\partial\omega\wedge\overline{\partial}\psi\wedge\omega^{n-5}_{\psi}\wedge(\sqrt{-1}\partial\overline{\partial}\psi)^{2}\right]\frac{\partial\psi}{\partial
t} \\
&=&\int_{X}\frac{\partial\psi}{\partial
t}\left[\partial\overline{\partial}\left(\frac{\partial\psi}{\partial
s}\right)\wedge\partial\omega\wedge\overline{\partial}\psi\wedge\omega^{n-5}_{\psi}\wedge(\sqrt{-1}\partial\overline{\partial}\psi)^{2}\right.
\\
&-&\left.\overline{\partial}\left(\frac{\partial\psi}{\partial
s}\right)\wedge\partial\left(\partial\omega\wedge\overline{\partial}\psi\wedge\omega^{n-5}_{\psi}\wedge(\sqrt{-1}\partial\overline{\partial}\psi\right)^{2})\right]
\\
&+&\int_{X}\frac{\partial\psi}{\partial
t}\left[\overline{\partial}\partial\left(\frac{\partial\psi}{\partial
s}\right)\wedge\partial\omega\wedge\overline{\partial}\psi\wedge\omega^{n-5}_{\psi}\wedge(\sqrt{-1}\partial\overline{\partial}\psi)^{2}\right.
\\
&-&\left.\partial\left(\frac{\partial\psi}{\partial
s}\right)\wedge\overline{\partial}\left(\partial\omega\wedge\overline{\partial}\psi\wedge\omega^{n-5}_{\psi}\wedge(\sqrt{-1}\partial\overline{\partial}\psi)^{2}\right)\right]
\\
&=&\int_{X}\frac{\partial\psi}{\partial
t}\cdot\overline{\partial}\left(\frac{\partial\psi}{\partial
s}\right)\wedge\partial\omega\wedge\partial\overline{\partial}\psi\wedge\omega^{n-5}_{\psi}\wedge(\sqrt{-1}\partial\overline{\partial}\psi)^{2}
\\
&-&\int_{X}\frac{\partial\psi}{\partial
t}\cdot\partial\left(\frac{\partial\psi}{\partial
s}\right)\wedge\overline{\partial}\partial\omega\wedge\overline{\partial}\psi\wedge\omega_{\psi}^{n-5}\wedge(\sqrt{-1}\partial\overline{\partial}\psi)^{2}
\\
&-&\int_{X}\frac{\partial\psi}{\partial
t}\cdot\partial\left(\frac{\partial\psi}{\partial
s}\right)\wedge\partial\omega\wedge\overline{\partial}\psi\wedge(n-5)\omega^{n-6}_{\psi}\wedge\overline{\partial}\omega\wedge(\sqrt{-1}\partial\overline{\partial}\psi)^{2}.
\end{eqnarray*}
Since
\begin{eqnarray}
& & \int_{X}\frac{\partial\psi}{\partial
t}\cdot\overline{\partial}\left(\frac{\partial\psi}{\partial
s}\right)\wedge\partial\omega\wedge\overline{\partial}\psi\wedge\omega^{n-5}_{\psi}\wedge(\sqrt{-1}\partial\overline{\partial}\psi)^{2}
\label{2.79}\\
&=&\int_{X}\left[\frac{\partial\psi}{\partial
t}\cdot\overline{\partial}\left(\frac{\partial\psi}{\partial
s}\right)\wedge\partial\omega\wedge\partial\overline{\partial}\psi\wedge\omega^{n-5}_{\psi}\wedge\sqrt{-1}\partial\overline{\partial}\psi\right]\wedge-\sqrt{-1}\overline{\partial}\partial\psi
\nonumber\\
&=&\int_{X}\sqrt{-1}\overline{\partial}\left(\frac{\partial\psi}{\partial
t}\right)\wedge\overline{\partial}\left(\frac{\partial\psi}{\partial
s}\right)\wedge\partial\omega\wedge\partial\overline{\partial}\psi\wedge\omega^{n-5}_{\psi}\wedge\sqrt{-1}\partial\overline{\partial}\psi\wedge\partial\psi
\nonumber\\
&-&\int_{X}\sqrt{-1}\frac{\partial\psi}{\partial
t}\cdot\overline{\partial}\left(\frac{\partial\psi}{\partial
s}\right)\wedge\left(\overline{\partial}\partial\omega\wedge\partial\overline{\partial}\psi\wedge\omega^{n-5}_{\psi}\wedge\sqrt{-1}\partial\overline{\partial}\psi\right.\nonumber\\
&-&\left.\partial\omega\wedge\partial\overline{\partial}\psi\wedge(n-5)\omega^{n-6}_{\psi}\wedge\overline{\partial}\omega\wedge\sqrt{-1}\partial\overline{\partial}\psi\right)\wedge\partial\psi,\nonumber
\end{eqnarray}
it follows that
\begin{eqnarray}
H_{3}&=&\int_{X}\overline{\partial}\left(\frac{\partial\psi}{\partial
s}\right)\wedge\overline{\partial}\left(\frac{\partial\psi}{\partial
t}\right)\wedge\partial\psi\wedge\partial\omega\wedge\omega^{n-5}_{\psi}\wedge(\sqrt{-1}\partial\overline{\partial}\psi)^{2}
\nonumber\\
&-&\int_{X}\frac{\partial\psi}{\partial
t}\cdot\overline{\partial}\left(\frac{\partial\psi}{\partial
s}\right)\wedge\partial\psi\wedge\overline{\partial}\partial\omega\wedge\omega^{n-5}_{\psi}\wedge(\sqrt{-1}\partial\overline{\partial}\psi)^{2}
\label{2.80}\\
&+&\int_{X}\frac{\partial\psi}{\partial
t}\cdot\overline{\partial}\left(\frac{\partial\psi}{\partial
s}\right)\wedge\partial\omega\wedge\overline{\partial}\omega\wedge(n-5)\omega^{n-6}_{\psi}\wedge(\sqrt{-1}\partial\overline{\partial}\psi)^{2}\wedge\partial\psi
\nonumber\\
&-&\int_{X}\frac{\partial\psi}{\partial
t}\cdot\partial\left(\frac{\partial\psi}{\partial
s}\right)\wedge\overline{\partial}\psi\wedge\overline{\partial}\partial\omega\wedge\omega^{n-5}_{\psi}\wedge(\sqrt{-1}\partial\overline{\partial}\psi)^{2}
\nonumber\\
&+&\int_{X}\frac{\partial\psi}{\partial
t}\cdot\partial\left(\frac{\partial\psi}{\partial
s}\right)\wedge\partial\omega\wedge\overline{\partial}\omega\wedge(n-5)\omega^{n-6}_{\psi}\wedge(\sqrt{-1}\partial\overline{\partial}\psi)^{2}\wedge\overline{\partial}\psi.\nonumber
\end{eqnarray}

Hence
\begin{eqnarray}
H_{3}+\overline{H_{3}}&=&\frac{A_{3}+B_{3}}{n-4}+\frac{\overline{A_{3}}+\overline{B_{3}}}{n-4},
\label{2.81}, \\
\frac{I^{7}}{a_{7}}+\frac{I^{8}}{a_{8}}&=&\frac{5}{n-4}(A_{3}+B_{3}+\overline{A_{3}}+\overline{B_{3}})
+(A_{4}+B_{4}+\overline{A_{4}}+\overline{B_{4}}).\label{2.82}
\end{eqnarray}
Set
\begin{equation}
c_{4}:=A_{4}+B_{4}+\overline{A_{4}}+\overline{B_{4}}.\label{2.83}
\end{equation}
Then
\begin{equation}
\frac{I^{7}}{a_{7}}+\frac{I^{8}}{a_{8}}=\frac{5}{n-4}c_{3}+c_{4}.\label{2.84}
\end{equation}

\subsection{Recursion formula}
Suppose now $n\geq4$. we define, for $2\leq i\leq n-2$,
\begin{eqnarray}
A_{i}&:=&\int_{X}\left[(n-i-1)\partial\psi\wedge\overline{\partial}\psi\wedge\partial\omega\wedge\overline{\partial}\left(\frac{\partial\psi}{\partial
s}\right)\wedge\omega^{n-i-2}_{\psi}\wedge(\sqrt{-1}\partial\overline{\partial}\psi)^{i-2}\right]
\nonumber\\
&&\wedge\sqrt{-1}\partial\overline{\partial}\left(\frac{\partial\psi}{\partial
t}\right), \label{2.85}\\
B_{i}&;=&\int_{X}\left[(n-i-1)\partial\psi\wedge\overline{\partial}\psi\wedge\partial\omega\wedge\overline{\partial}\left(\frac{\partial\psi}{\partial
t}\right)\wedge\omega^{n-i-2}_{\psi}\wedge(\sqrt{-1}\partial\overline{\partial}\psi)^{i-2}\right]
\nonumber\\
&&\wedge-\sqrt{-1}\partial\overline{\partial}\left(\frac{\partial\psi}{\partial
s}\right).\label{2.86}
\end{eqnarray}

So

\begin{eqnarray*}
A_{i}&=&\int_{X}\left[(n-i-1)\partial\psi\wedge\overline{\partial}\psi\wedge\partial\omega\wedge\overline{\partial}\left(\frac{\partial\psi}{\partial
s}\right)\wedge\omega^{n-i-2}_{\psi}\wedge(\sqrt{-1}\partial\overline{\partial}\psi)^{i-2}\right]\\
&&\wedge\sqrt{-1}\partial\overline{\partial}\left(\frac{\partial\psi}{\partial
t}\right) \\
&=&\int_{X}-\sqrt{-1}\partial\left[(n-i-1)\partial\psi\wedge\overline{\partial}\psi\wedge\partial\omega\wedge\overline{\partial}\left(\frac{\partial\psi}{\partial
s}\right)\wedge\omega^{n-i-2}_{\psi}\wedge(\sqrt{-1}\partial\overline{\partial}\psi)^{i-2}\right]\\
&&\wedge\overline{\partial}\left(\frac{\partial\psi}{\partial
t}\right) \\
&=&\int_{X}\sqrt{-1}(n-i-1)\partial\psi\wedge\left[\partial\overline{\partial}\psi\wedge\partial\omega\wedge\overline{\partial}\left(\frac{\partial\psi}{\partial
s}\right)\wedge\omega^{n-i-2}_{\psi}\wedge(\sqrt{-1}\partial\overline{\partial}\psi)^{i-2}\right.
\\
&-&\left.\overline{\partial}\psi\wedge\partial\left(\partial\omega\wedge\overline{\partial}\left(\frac{\partial\psi}{\partial
s}\right)\wedge\omega^{n-i-2}_{\psi}\wedge(\sqrt{-1}\partial\overline{\partial}\psi)^{i-2}\right)\right]\wedge\overline{\partial}\left(\frac{\partial\psi}{\partial
t}\right)\\
&=&\int_{X}\sqrt{-1}(n-i-1)\partial\psi\wedge\partial\overline{\partial}\psi\wedge\partial\omega\wedge\overline{\partial}\left(\frac{\partial\psi}{\partial
s}\right)\wedge\omega^{n-i-2}_{\psi}\wedge(\sqrt{-1}\partial\overline{\partial}\psi)^{i-2}\\
&&\wedge\overline{\partial}\left(\frac{\partial\psi}{\partial
t}\right) \\
&+&\int_{X}\sqrt{-1}(n-i-1)\partial\psi\wedge\overline{\partial}\psi\wedge\partial\omega\wedge\partial\left(\overline{\partial}\left(\frac{\partial\psi}{\partial
s}\right)\wedge\omega^{n-i-2}_{\psi}\wedge(\sqrt{-1}\partial\overline{\partial}\psi)^{i-2}\right)\\
&&\wedge\overline{\partial}\left(\frac{\partial\psi}{\partial
t}\right)\\
&=&\int_{X}(n-i-1)\partial\psi\wedge\overline{\partial}\psi\wedge\partial\omega\wedge\overline{\partial}\left(\frac{\partial\psi}{\partial
t}\right)\wedge\omega^{n-i-2}_{\psi}\wedge(\sqrt{-1}\partial\overline{\partial}\psi)^{i-2}\\
&&\wedge\sqrt{-1}\partial\overline{\partial}\left(\frac{\partial\psi}{\partial
s}\right) \\
&+&\int_{X}(n-i-1)\partial\psi\wedge\partial\omega\wedge\overline{\partial}\left(\frac{\partial\psi}{\partial
s}\right)\wedge\overline{\partial}\left(\frac{\partial\psi}{\partial
t}\right)\wedge\omega^{n-i-2}\wedge(\sqrt{-1}\partial\overline{\partial}\psi)^{i-1};
\end{eqnarray*}
thus
\begin{equation}
A_{i}+B_{i}=\int_{X}(n-i-1)\partial\psi\wedge\partial\omega\wedge\overline{\partial}\left(\frac{\partial\psi}{\partial
s}\right)\wedge\overline{\partial}\left(\frac{\partial\psi}{\partial
t}\right)\wedge\omega^{n-i-2}_{\psi}\wedge(\sqrt{-1}\partial\overline{\partial}\psi)^{i-1}.\label{2.87}
\end{equation}
Define, where $a_{2i+1}$ and $a_{2i+2}$ are nonzero constants and we require $\overline{a_{2i+1}}=a_{2i+2}$,
\begin{eqnarray}
\mathcal{L}^{2i+1}_{\omega}(\varphi',\varphi'')&:=&\frac{1}{V_{\omega}}\int^{1}_{0}\int_{X}a_{2i+1}\partial\varphi_{t}\wedge\partial\omega\wedge\overline{\partial}
\dot{\varphi}_{t}\wedge\overline{\partial}\varphi_{t}\wedge\omega^{n-i-2}_{\varphi_{t}}\nonumber\wedge(\sqrt{-1}\partial\overline{\partial}\varphi_{t})^{i-1},
\label{2.88}\\
\mathcal{L}^{2i+2}_{\omega}(\varphi',\varphi'')&:=&\frac{1}{V_{\omega}}\int^{1}_{0}\int_{X}a_{2i+2}\overline{\partial}\varphi_{t}\wedge\overline{\partial}\omega
\wedge\partial\dot{\varphi}_{t}\wedge\partial\varphi_{t}\wedge\omega^{n-i-2}_{\varphi_{t}}\nonumber \\
&&\wedge(\sqrt{-1}\partial\overline{\partial}\varphi_{t})^{i-1}.\label{2.89}
\end{eqnarray}

Consider
\begin{eqnarray}
\Psi^{2i+1}&=&\left[\int_{X}a_{2i+1}\partial\psi\wedge\partial\omega\wedge\overline{\partial}\left(\frac{\partial\psi}{\partial
s}\right)\wedge\overline{\partial}\psi\wedge\omega^{n-i-2}_{\psi}\wedge(\sqrt{-1}\partial\overline{\partial}\psi)^{i-1}\right]ds
\nonumber\\
&+&\left[\int_{X}a_{2i+1}\partial\psi\wedge\partial\omega\wedge\overline{\partial}\left(\frac{\partial\psi}{\partial
t}\right)\wedge\overline{\partial}\psi\wedge\omega^{n-i-2}_{\psi}\wedge(\sqrt{-1}\partial\overline{\partial}\psi)^{i-1}\right]dt,\label{2.90}
\end{eqnarray}
and
\begin{eqnarray}
\Psi^{2i+2}&=&\left[\int_{X}a_{2i+2}\overline{\partial}\psi\wedge\overline{\partial}\omega\wedge\partial\left(\frac{\partial\psi}{\partial
s}\right)\wedge\partial\psi\wedge\omega^{n-i-2}_{\psi}\wedge(\sqrt{-1}\partial\overline{\partial}\psi)^{i-1}\right]ds
\nonumber\\
&+&\left[\int_{X}a_{2i+2}\overline{\partial}\psi\wedge\overline{\partial}\omega\wedge\partial\left(\frac{\partial\psi}{\partial
t}\right)\wedge\partial\psi\wedge\omega^{n-i-2}_{\psi}\wedge(\sqrt{-1}\partial\overline{\partial}\psi)^{i-1}\right]dt.\label{2.91}
\end{eqnarray}
So
\begin{equation}
d\Psi^{2i+1}=I^{2i+1}\cdot dt\wedge ds, \ \ \
d\Psi^{2i+2}=I^{2i+2}\cdot dt\wedge ds,\label{2.92}
\end{equation}
where
\begin{eqnarray*}
\frac{I^{2i+1}}{a_{2i+1}}&:=&\int_{X}\frac{\partial}{\partial
t}\left[\partial\psi\wedge\partial\omega\wedge\overline{\partial}\left(\frac{\partial\psi}{\partial
s}\right)\wedge\overline{\partial}\psi\wedge\omega^{n-i-2}_{\psi}\wedge(\sqrt{-1}\partial\overline{\partial}\psi)^{i-1}\right]
\\
&-&\int_{X}\frac{\partial}{\partial
s}\left[\partial\psi\wedge\partial\omega\wedge\overline{\partial}\left(\frac{\partial\psi}{\partial
t}\right)\wedge\overline{\partial}\psi\wedge\omega^{n-i-2}_{\psi}\wedge(\sqrt{-1}\partial\overline{\partial}\psi)^{i-1}\right]
\\
&=&\int_{X}\left[\partial\left(\frac{\partial\psi}{\partial
t}\right)\wedge\partial\omega\wedge\overline{\partial}\left(\frac{\partial\psi}{\partial
s}\right)\wedge\overline{\partial}\psi\wedge\omega^{-i-2}_{\psi}\wedge(\sqrt{-1}\partial\overline{\partial}\psi)^{i-1}\right.
\\
&+&\partial\psi\wedge\partial\omega\wedge\overline{\partial}\left(\frac{\partial^{2}\psi}{\partial
t\partial
s}\right)\wedge\overline{\partial}\psi\wedge\omega^{n-i-2}_{\psi}\wedge(\sqrt{-1}\partial\overline{\partial}\psi)^{i-1}
\\
&+&\partial\psi\wedge\partial\omega\wedge\overline{\partial}\left(\frac{\partial\psi}{\partial
s}\right)\wedge\overline{\partial}\left(\frac{\partial\psi}{\partial
t}\right)\wedge\omega^{n-i-2}_{\psi}\wedge(\sqrt{-1}\partial\overline{\partial}\psi)^{i-1}
\\
&+&\partial\psi\wedge\partial\omega\wedge\overline{\partial}\left(\frac{\partial\psi}{\partial
s}\right)\wedge\overline{\partial}\psi\wedge(n-i-2)\omega^{n-i-3}_{\psi}\wedge\sqrt{-1}\partial\overline{\partial}\left(\frac{\partial\psi}{\partial
t}\right)\\
&&\wedge(\sqrt{-1}\partial\overline{\partial}\psi)^{i-1} \\
&+&\left.\partial\psi\wedge\partial\omega\wedge\overline{\partial}\left(\frac{\partial\psi}{\partial
s}\right)\wedge\overline{\partial}\psi\wedge\omega^{n-i-2}_{\psi}\wedge(i-1)(\sqrt{-1}\partial\overline{\partial}\psi)^{i-2}\right. \\
&&\left.\wedge\sqrt{-1}\partial\overline{\partial}
\left(\frac{\partial\psi}{\partial t}\right)\right]\\
&-&\int_{X}\left[\partial\left(\frac{\partial\psi}{\partial
s}\right)\wedge\partial\omega\wedge\overline{\partial}\left(\frac{\partial\psi}{\partial
t}\right)\wedge\overline{\partial}\psi\wedge\omega^{n-i-2}_{\psi}\wedge(\sqrt{-1}\partial\overline{\partial}\psi)^{i-1}\right.
\\
&+&\partial\psi\wedge\partial\omega\wedge\overline{\partial}\left(\frac{\partial^{2}\psi}{\partial
s\partial
t}\right)\wedge\overline{\partial}\psi\wedge\omega^{n-i-2}_{\psi}\wedge(\sqrt{-1}\partial\overline{\partial}\psi)^{i-1}
\\
&+&\partial\psi\wedge\partial\omega\wedge\overline{\partial}\left(\frac{\partial\psi}{\partial
t}\right)\wedge\overline{\partial}\left(\frac{\partial\psi}{\partial
s}\right)\wedge\omega^{n-i-2}_{\psi}\wedge(\sqrt{-1}\partial\overline{\partial}\psi)^{i-1}
\\
&+&\partial\psi\wedge\partial\omega\wedge\overline{\partial}\left(\frac{\partial\psi}{\partial
t}\right)\wedge\overline{\partial}\psi\wedge(n-i-2)\omega^{n-i-3}_{\psi}\wedge\sqrt{-1}\partial\overline{\partial}\left(\frac{\partial\psi}{\partial
s}\right)\\
&&\wedge(\sqrt{-1}\partial\overline{\partial}\psi)^{i-1}\\
&+&\left.\partial\psi\wedge\partial\omega\wedge\overline{\partial}\left(\frac{\partial\psi}{\partial
t}\right)\wedge\overline{\partial}\psi\wedge\omega^{n-i-2}_{\psi}\wedge(i-1)(\sqrt{-1}\partial\overline{\partial}\psi)^{i-2}\right. \\
&&\left.\wedge\sqrt{-1}\partial\overline{\partial}\left(
\frac{\partial\psi}{\partial s}\right)\right].
\end{eqnarray*}
Therefore
\begin{eqnarray}
\frac{I^{2i+1}}{a_{2i+1}}&=&\int_{X}-\partial\left(\frac{\partial\psi}{\partial
t}\right)\wedge\overline{\partial}\left(\frac{\partial\psi}{\partial
s}\right)\wedge\partial\omega\wedge\overline{\partial}\psi\wedge\omega^{n-i-2}_{\psi}\wedge(\sqrt{-1}\partial\overline{\partial}\psi)^{i-1}\nonumber\\
&+&\int_{X}\partial\left(\frac{\partial\psi}{\partial
s}\right)\wedge\overline{\partial}\left(\frac{\partial\psi}{\partial
t}\right)\wedge\partial\omega\wedge\overline{\partial}\psi\wedge\omega^{n-i-2}_{\psi}\wedge(\sqrt{-1}\partial\overline{\partial}\psi)^{i-1}
\nonumber\\
&+&2\int_{X}\partial\psi\wedge\partial\omega\wedge\overline{\partial}\left(\frac{\partial\psi}{\partial
s}\right)\wedge\overline{\partial}\left(\frac{\partial\psi}{\partial
t}\right)\wedge\omega^{n-i-2}_{\psi}\wedge(\sqrt{-1}\partial\overline{\partial}\psi)^{i-1}
\nonumber\\
&+&\int_{X}\partial\psi\wedge\partial\omega\wedge\overline{\partial}\left(\frac{\partial\psi}{\partial
s}\right)\wedge\overline{\partial}\psi\wedge(n-i-2)\omega^{n-i-3}_{\psi}\wedge\sqrt{-1}\partial\overline{\partial}\left(\frac{\partial\psi}{\partial
t}\right)\nonumber\\
&&\wedge(\sqrt{-1}\partial\overline{\partial}\psi)^{i-1} \nonumber\\
&-&\int_{X}\partial\psi\wedge\partial\omega\wedge\overline{\partial}\left(\frac{\partial\psi}{\partial
t}\right)\wedge\overline{\partial}\psi\wedge(n-i-2)\omega^{n-i-3}_{\psi}\wedge\sqrt{-1}\partial\overline{\partial}\left(\frac{\partial\psi}{\partial
s}\right)\nonumber \\
&&\wedge(\sqrt{-1}\partial\overline{\partial}\psi)^{i-1} +(i-1)(A_{i}+B_{i})/(n-i-1)\nonumber\\
&=&H_{i}+\frac{i+1}{n-(i+1)}(A_{i}+B_{i})+A_{i+1}+B_{i+1}.\label{2.93}
\end{eqnarray}
Here
\begin{eqnarray}
H_{i}&:=&\int_{X}-\partial\left(\frac{\partial\psi}{\partial
t}\right)\wedge\overline{\partial}\left(\frac{\partial\psi}{\partial
s}\right)\wedge\partial\omega\wedge\overline{\partial}\psi\wedge\omega^{n-i-2}_{\psi}\wedge(\sqrt{-1}\partial\overline{\partial}\psi)^{i-1}
\nonumber\\
&+&\int_{X}\partial\left(\frac{\partial\psi}{\partial
s}\right)\wedge\overline{\partial}\left(\frac{\partial\psi}{\partial
t}\right)\wedge\partial\omega\wedge
\overline{\partial}\psi\wedge\omega^{n-i-2}_{\psi}
\wedge(\sqrt{-1}\partial\overline{\partial}\psi)^{i-1}.\label{2.94}
\end{eqnarray}
Similarly
\begin{equation}
\frac{I^{2i+2}}{a_{2i+2}}=\overline{H_{i}}+\frac{i+1}{n-(i+1)}
(\overline{A_{i}}+\overline{B_{i}})+\overline{A_{i+1}}+\overline{B_{i+1}}.\label{2.95}
\end{equation}
Calculate
\begin{eqnarray*}
H_{i}&=&\int_{X}\partial\left[\overline{\partial}\left(\frac{\partial\psi}{\partial
s}\right)\wedge\partial\omega\wedge\overline{\partial}\psi\wedge\omega^{n-i-2}_{\psi}\wedge(\sqrt{-1}\partial\overline{\partial}\psi)^{i-1}\right]\frac{\partial\psi}{\partial
t} \\
&+&\int_{X}\overline{\partial}\left[\partial\left(\frac{\partial\psi}{\partial
s}\right)\wedge\partial\omega\wedge\overline{\partial}\psi\wedge\omega^{n-i-2}_{\psi}\wedge(\sqrt{-1}\partial\overline{\partial}\psi)^{i-1}\right]\frac{\partial\psi}{\partial
t} \\
&=&\int_{X}\frac{\partial\psi}{\partial
t}\left[\partial\overline{\partial}\left(\frac{\partial\psi}{\partial
s}\right)\wedge\partial\omega\wedge\overline{\partial}\psi\wedge\omega^{n-i-2}_{\psi}\wedge(\sqrt{-1}\partial\overline{\partial}\psi)^{i-1}\right.
\\
&-&\left.\overline{\partial}\left(\frac{\partial\psi}{\partial
s}\right)\wedge\partial\left(\partial\omega\wedge\overline{\partial}\psi\wedge\omega^{n-i-2}_{\psi}\wedge(\sqrt{-1}\partial\overline{\partial}\psi)^{i-1}\right)\right]
\\
&+&\int_{X}\frac{\partial\psi}{\partial
t}\left[\overline{\partial}\partial\left(\frac{\partial\psi}{\partial
s}\right)\wedge\partial\omega\wedge\overline{\partial}\psi\wedge\omega^{n-i-2}_{\psi}\wedge(\sqrt{-1}\partial\overline{\partial}\psi)^{i-1}\right.
\\
&-&\left.\partial\left(\frac{\partial\psi}{\partial
s}\right)\wedge\overline{\partial}\left(\partial\omega\wedge\overline{\partial}\psi\wedge\omega^{n-i-2}_{\psi}\wedge(\sqrt{-1}\partial\overline{\partial}\psi)^{i-1}\right)\right]
\\
&=&\int_{X}\frac{\partial\psi}{\partial
t}\cdot\overline{\partial}\left(\frac{\partial\psi}{\partial
s}\right)\wedge\partial\omega\wedge\partial\overline{\partial}\psi\wedge\omega^{n-i-2}_{\psi}\wedge(\sqrt{-1}\partial\overline{\partial}\psi)^{i-1}
\\
&-&\int_{X}\frac{\partial\psi}{\partial
t}\cdot\partial\left(\frac{\partial\psi}{\partial
s}\right)\wedge\overline{\partial}\partial\omega\wedge\overline{\partial}\psi\wedge\omega^{n-i-2}_{\psi}\wedge(\sqrt{-1}\partial\overline{\partial}\psi)^{i-1}
\\
&-&\int_{X}\frac{\partial\psi}{\partial
t}\cdot\partial\left(\frac{\partial\psi}{\partial
s}\right)\wedge\partial\omega\wedge\overline{\partial}\psi\wedge(n-i-2)\omega^{n-i-3}_{\psi}\wedge\overline{\partial}\omega\wedge(\sqrt{-1}\partial\overline{\partial}\psi)^{i-1}.
\end{eqnarray*}
Since
\begin{eqnarray*}
& & \int_{X}\frac{\partial\psi}{\partial
t}\cdot\overline{\partial}\left(\frac{\partial\psi}{\partial
s}\right)\wedge\partial\omega\wedge\partial\overline{\partial}\psi\wedge\omega^{n-i-2}_{\psi}\wedge(\sqrt{-1}\partial\overline{\partial}\psi)^{i-1}
\\
&=&\int_{X}\left[\frac{\partial\psi}{\partial
t}\cdot\overline{\partial}\left(\frac{\partial\psi}{\partial
s}\right)\wedge\partial\omega\wedge\partial\overline{\partial}\psi\wedge\omega^{n-i-2}_{\psi}\wedge(\sqrt{-1}\partial\overline{\partial}\psi)^{i-2}\right]
\wedge(-\sqrt{-1}\overline{\partial}\partial\psi) \\
&=&\int_{X}\sqrt{-1}\overline{\partial}\left[\frac{\partial\psi}{\partial
t}\cdot\overline{\partial}\left(\frac{\partial\psi}{\partial
s}\right)\wedge\partial\omega\wedge\partial\overline{\partial}\psi\wedge\omega^{n-i-2}_{\psi}\wedge(\sqrt{-1}\partial\overline{\partial}\psi)^{i-2}\right]\wedge\partial\psi
\\
&=&\int_{X}\sqrt{-1}\left[\overline{\partial}\left(\frac{\partial\psi}{\partial
t}\right)\wedge\overline{\partial}\left(\frac{\partial\psi}{\partial
s}\right)\wedge\partial\omega\wedge\partial\overline{\partial}\psi\wedge\omega^{n-i-2}_{\psi}\wedge(\sqrt{-1}\partial\overline{\partial}\psi)^{i-2}\right.
\\
&-&\left.\frac{\partial\psi}{\partial
t}\cdot\overline{\partial}\left(\frac{\partial\psi}{\partial
s}\right)\wedge\overline{\partial}\left(\partial\omega\wedge\partial\overline{\partial}\psi\wedge\omega^{n-i-2}_{\psi}\wedge(\sqrt{-1}\partial\overline{\partial}\psi)^{i-2}\right)\right]\wedge\partial\psi\\
&=&\int_{X}\sqrt{-1}\overline{\partial}\left(\frac{\partial\psi}{\partial
t}\right)\wedge\overline{\partial}\left(\frac{\partial\psi}{\partial
s}\right)\wedge\partial\omega\wedge\partial\overline{\partial}\psi\wedge\omega^{n-i-2}_{\psi}\wedge(\sqrt{-1}\partial\overline{\partial}\psi)^{i-2}\wedge\partial\psi
\\
&-&\int_{X}\sqrt{-1}\frac{\partial\psi}{\partial
t}\cdot\overline{\partial}\left(\frac{\partial\psi}{\partial
s}\right)\wedge\left(\overline{\partial}\partial\omega\wedge\partial\overline{\partial}\psi\wedge\omega^{n-i-2}_{\psi}\wedge(\sqrt{-1}\partial\overline{\partial}\psi)^{i-2}\right.
\\
&-&\left.\partial\omega\wedge\partial\overline{\partial}\psi\wedge(n-i-2)\omega^{n-i-3}_{\psi}\wedge\overline{\partial}\omega\wedge(\sqrt{-1}\partial\overline{\partial}\psi)^{i-2}\right)\wedge\partial\psi,
\end{eqnarray*}
it follows that
\begin{eqnarray}
H_{i}&=&\int_{X}\overline{\partial}\left(\frac{\partial\psi}{\partial
s}\right)\wedge\overline{\partial}\left(\frac{\partial\psi}{\partial
t}\right)\wedge\partial\psi\wedge\partial\omega\wedge\omega^{n-i-2}_{\psi}\wedge(\sqrt{-1}\partial\overline{\partial}\psi)^{i-1}
\nonumber\\
&-&\int_{X}\frac{\partial\psi}{\partial
t}\cdot\overline{\partial}\left(\frac{\partial\psi}{\partial
s}\right)\wedge\partial\psi\wedge\overline{\partial}\partial\omega\wedge\omega^{n-i-2}_{\psi}\wedge(\sqrt{-1}\partial\overline{\partial}\psi)^{i-1}
\label{2.96}\\
&+&\int_{X}\frac{\partial\psi}{\partial
t}\cdot\overline{\partial}\left(\frac{\partial\psi}{\partial
s}\right)\wedge\partial\omega\wedge\overline{\partial}\omega\wedge(n-i-2)\omega^{n-i-3}_{\psi}\wedge(\sqrt{-1}\partial\overline{\partial}\psi)^{i-1}\wedge\partial\psi
\nonumber\\
&-&\int_{X}\frac{\partial\psi}{\partial
t}\cdot\partial\left(\frac{\partial\psi}{\partial
s}\right)\wedge\overline{\partial}\partial\omega\wedge\overline{\partial}\psi\wedge\omega^{n-i-2}_{\psi}\wedge(\sqrt{-1}\partial\overline{\partial}\psi)^{i-1}
\nonumber\\
&+&\int_{X}\frac{\partial\psi}{\partial
t}\cdot\partial\left(\frac{\partial\psi}{\partial
s}\right)\wedge\partial\omega\wedge\overline{\partial}\omega\wedge(n-i-2)\omega^{n-i-3}_{\psi}\wedge(\sqrt{-1}\partial\overline{\partial}\psi)^{i-1}\wedge\overline{\partial}\psi,\nonumber
\end{eqnarray}
and, similarly,
\begin{eqnarray}
\overline{H_{i}}&=&\int_{X}\partial\left(\frac{\partial\psi}{\partial
s}\right)\wedge\partial\left(\frac{\partial\psi}{\partial
t}\right)\wedge\overline{\partial}\psi\wedge\overline{\partial}\omega\wedge\omega^{n-i-2}_{\psi}\wedge(\sqrt{-1}\partial\overline{\partial}\psi)^{i-1}
\nonumber\\
&-&\int_{X}\frac{\partial\psi}{\partial
t}\cdot\partial\left(\frac{\partial\psi}{\partial
s}\right)\wedge\overline{\partial}\psi\wedge\partial\overline{\partial}\omega\wedge\omega^{n-i-2}_{\psi}\wedge(\sqrt{-1}\partial\overline{\partial}\psi)^{i-1}
\label{2.97}\\
&+&\int_{X}\frac{\partial\psi}{\partial
t}\cdot\partial\left(\frac{\partial\psi}{\partial
s}\right)\wedge\overline{\partial}\omega\wedge\partial\omega\wedge(n-i-2)\omega^{n-i-3}_{\psi}\wedge(\sqrt{-1}\partial\overline{\partial}\psi)^{i-1}\wedge\overline{\partial}\psi\nonumber\\
&-&\int_{X}\frac{\partial\psi}{\partial
t}\cdot\overline{\partial}\left(\frac{\partial\psi}{\partial
s}\right)\wedge\partial\psi\wedge\partial\overline{\partial}\omega\wedge\omega^{n-i-2}_{\psi}\wedge(\sqrt{-1}\partial\overline{\partial}\psi)^{i-1}\nonumber\\
&+&\int_{X}\frac{\partial\psi}{\partial
t}\cdot\overline{\partial}\left(\frac{\partial\psi}{\partial
s}\right)\wedge\overline{\partial}\omega\wedge\partial\omega\wedge(n-i-2)\omega^{n-i-3}_{\psi}\wedge(\sqrt{-1}\partial\overline{\partial}\psi)^{i-1}\wedge\partial\psi.\nonumber
\end{eqnarray}
So, for $2\leq i\leq n-2$,
\begin{eqnarray}
H_{i}+\overline{H_{i}}&=&\frac{A_{i}+B_{i}}{n-i-1}+\frac{\overline{A_{i}}+\overline{B_{i}}}{n-i-1},
\label{2.98}\\
\frac{I^{2i+1}}{a_{2i+1}}+\frac{I^{2i+2}}{a_{2i+2}}&=&H_{i}+\overline{H_{i}}+\frac{i+1}{n-(i+1)}(A_{i}+B_{i}+\overline{A_{i}}+\overline{B_{i}})\nonumber\\
&+&(A_{i+1}+B_{i+1}+\overline{A_{i+1}}+\overline{B_{i+1}})\nonumber \\
&=&\frac{i+2}{n-(i+1)}(A_{i}+B_{i}+\overline{A_{i}}+\overline{B_{i}})\nonumber \\
&+&(A_{i+1}+B_{i+1}+\overline{A_{i+1}}+\overline{B_{i+1}}).\label{2.99}
\end{eqnarray}

Recall, see (\ref{2.49}),
\begin{eqnarray*}
\frac{I^{3}}{a_{3}}+\frac{I^{4}}{a_{4}}&=&\frac{3\sqrt{-1}}{n-2}\left((A_{1}+B_{1})-(\overline{A_{1}}+\overline{B_{1}})\right)+(A_{2}+B_{2}+\overline{A_{2}}+\overline{B_{2}})
\\
\frac{2I^{0}}{n(n-1)\sqrt{-1}}&=&\frac{I^{1}}{a_{1}}-\frac{I^{2}}{a_{2}}+(A_{1}+B_{1})-(\overline{A_{1}}+\overline{B_{1}}).
\end{eqnarray*}
Let
\begin{eqnarray}
c_{i}&:=&A_{i}+B_{i}+\overline{A_{i}}+\overline{B_{i}}, \ \ \ 2\leq
i\leq n-1, \label{2.100}\\
c_{1}&:=&A_{1}+B_{1}-(\overline{A_{1}}+\overline{B_{1}}).\label{2.101}
\end{eqnarray}
Notice that $c_{n-1}=0$. So
\begin{eqnarray}
\frac{I^{2i+1}}{a_{2i+1}}+\frac{I^{2i+2}}{a_{2i+2}}&=&\frac{i+2}{n-(i+1)}c_{i}+c_{i+1},
\ \ \ 2\leq i\leq n-2, \label{2.102}\\
\frac{I^{3}}{a_{3}}+\frac{I^{4}}{a_{4}}&=&\frac{3}{-(n-2)\sqrt{-1}}c_{1}+c_{2},\label{2.103}
\end{eqnarray}
and
\begin{eqnarray*}
\frac{2I^{0}}{n(n-1)\sqrt{-1}}&=&\frac{I^{1}}{a_{1}}-\frac{I^{2}}{a_{2}}+c_{1}
\\
&=&\frac{I^{1}}{a_{1}}-\frac{I^{2}}{a_{2}}+\left(\frac{I^{3}}{a_{3}}+\frac{I^{4}}{a_{4}}-c_{2}\right)\frac{n-2}{3\sqrt{-1}}
\\
&=&\left(\frac{I^{1}}{a_{1}}-\frac{I^{2}}{a_{2}}\right)+\frac{n-2}{3\sqrt{-1}}\left(\frac{I^{3}}{a_{3}}+\frac{I^{4}}{a_{4}}\right)-\frac{n-2}{3\sqrt{-1}}c_{2}.
\end{eqnarray*}
It is sufficient to determine $c_{2}$. Let
\begin{equation}
J_{i}:=\frac{I^{2i+1}}{a_{2i+1}}+\frac{I^{2i+2}}{a_{2i+2}}.\label{2.104}
\end{equation}
Then
\begin{equation}
c_{i+1}=-\frac{i+2}{n-(i+1)}c_{i}+J_{i}, \ \ \ 2\leq i\leq n-2.\label{2.105}
\end{equation}
To completely determine $c_{2}$, it is sufficient to solve (\ref{2.105}). A
direct calculation shows
\begin{eqnarray*}
c_{3}&=&-\frac{4}{n-3}c_{2}+J_{2}, \\
c_{4}&=&-\frac{5}{n-4}c_{3}+J_{3} \ \ = \ \ -\frac{5}{n-4}\left(-\frac{4}{n-3}c_{2}+J_{2}\right)+J_{3} \\
&=&(-1)^{2}\frac{5\times
4}{(n-4)(n-3)}c_{2}-\frac{5}{n-4}J_{2}+J_{3}, \\
c_{5}&=&-\frac{6}{n-5}c_{4}+J_{4}\\
&=&-\frac{6}{n-5}\left[(-1)^{2}\frac{5\times
4}{(n-4)(n-3)}c_{2}-\frac{5}{n-4}J_{2}+J_{3}\right]+J_{4}\\
&=&(-1)^{3}\frac{6\times 5\times
4}{(n-5)(n-4)(n-3)}c_{2}+(-1)^{2}\frac{6\times
5}{(n-5)(n-4)}J_{2}-\frac{6}{n-5}J_{3}+J_{4}
\end{eqnarray*}
Hence, we have
\begin{equation}
c_{i}=(-1)^{i-2}\frac{(i+1)!(n-i-1)!}{3!(n-3)!}c_{2}
+\sum^{i-1}_{k=2}(-1)^{i-1-k}\frac{(i+1)!(n-i-1)!}{(k+2)!(n-k-2)!}J_{k}.\label{2.106}
\end{equation}

By induction on $i$, we have
\begin{eqnarray*}
c_{i+1}&=&-\frac{i+2}{n-(i+1)}c_{i}+J_{i} \\
&=&-\frac{i+2}{n-(i+1)}\left[(-1)^{i-2}\frac{(i+1)!(n-i-1)!}{3!(n-3)!}c_{2}\right. \\
&+&\left.\sum^{i-1}_{k=2}(-1)^{i-1-k}\frac{(i+1)!(n-i-1)!}{(k+2)!(n-k-2)!}J_{k}\right]+J_{i}
\\
&=&(-1)^{i+1-2}\frac{(i+2)!(n-i-2)!}{3!(n-3)!}c_{2}\\
&+&\sum^{i-1}_{k=2}(-1)^{i-k}\frac{(i+2)!(n-i-2)!}{(k+2)!(n-k-2)!}J_{k}+J_{i}.
\end{eqnarray*}
So (\ref{2.106}) holds for $2\leq i\leq n-2$ and we have
\begin{eqnarray}
c_{2}&=&(-1)^{i-2}\frac{3!(n-3)!}{(i+1)!(n-i-1)!}\nonumber \\
&\cdot&\left[c_{i}-\sum^{i-1}_{k=2}(-1)^{i-1-k}\frac{(i+1)!(n-i-1)!}{(k+2)!(n-k-2)!}J_{k}\right].\label{2.107}
\end{eqnarray}
Setting $i=n-1$ yields
\begin{eqnarray}
c_{2}&=&(-1)^{n-3}\frac{3!(n-3)!}{n!}\nonumber \\
&\cdot&\left[c_{n-1}-\sum^{n-2}_{k=2}(-1)^{n-2-k}\frac{n!}{(k+2)!(n-k-2)!}J_{k}\right],
\ \ \ n\geq3. \label{2.108}
\end{eqnarray}

We deduce from (\ref{2.49}) and (\ref{2.108}) that
\begin{eqnarray*}
\frac{2I^{0}}{n(n-1)\sqrt{-1}}&=&\left(\frac{I^{0}}{a_{1}}-\frac{I^{2}}{a_{2}}\right)+\frac{n-2}{3\sqrt{-1}}\left(\frac{I^{3}}{a_{3}}+\frac{I^{4}}{a_{4}}\right)
\\
&-&\frac{n-2}{3\sqrt{-1}}(-1)^{n-3}\frac{3!(n-3)!}{n!} \\
&\cdot&\left[c_{n-1}-\sum^{n-2}_{k=2}(-1)^{n-2-k}\frac{n!}{(k+2)!(n-k-2)!}J_{k}\right]
\\
&=&\left(\frac{I^{0}}{a_{1}}-\frac{I^{2}}{a_{2}}\right)+\frac{n-2}{3\sqrt{-1}}\left(\frac{I^{3}}{a_{3}}+\frac{I^{4}}{a_{4}}\right)
\\
&+&\frac{2(n-2)!}{\sqrt{-1}}\sum^{n-2}_{k=2}\frac{(-1)^{k+1}}{(k+2)!(n-k-2)!}\left(\frac{I^{2k+1}}{a_{2k+1}}+\frac{I^{2k+2}}{a_{2k+2}}\right).
\end{eqnarray*}
Equivalently,
\begin{equation}
\frac{2I^{0}}{n(n-1)\sqrt{-1}}=\left(\frac{I^{1}}{a_{1}}
-\frac{I^{2}}{a_{2}}\right)+\sqrt{-1}\sum^{n-2}_{k=1}(-1)^{k}
\frac{\binom{n}{k+2}}{\binom{n}{2}}\left(\frac{I^{2k+1}}{a_{2k+1}}
+\frac{I^{2k+2}}{a_{2k+2}}\right).\label{2.109}
\end{equation}
Set
\begin{eqnarray}
\frac{1}{a_{1}}&=&-\frac{2}{n(n-1)\sqrt{-1}}, \ \ \ \frac{1}{a_{2}}
\ \ = \ \ \frac{2}{n(n-1)\sqrt{-1}}, \label{2.110}\\
\frac{\sqrt{-1}(-1)^{k}\binom{n}{k+2}}{\binom{n}{2}a_{2k+1}}&=&-\frac{2}{n(n-1)\sqrt{-1}},
\ \ \ a_{2i+1} \ \ = \ \ a_{2i+2},\label{2.111}
\end{eqnarray}
we obtain
\begin{eqnarray}
a_{1}&=&-\frac{n(n-1)\sqrt{-1}}{2}, \ \ \ a_{2} \ \ = \ \
\frac{n(n-1)\sqrt{-1}}{2}, \label{2.112}\\
a_{2k+1}&=&a_{2k+2} \nonumber\\
&=&\frac{\sqrt{-1}(-1)^{k}\binom{n}{k+2}n(n-1)\sqrt{-1}}{-2\binom{n}{2}}
\nonumber\\
&=&\frac{(-1)^{k}n(n-1)\binom{n}{k+2}}{2\binom{n}{2}} \ \ = \ \
(-1)^{k}\binom{n}{k+2}.\label{2.113}
\end{eqnarray}
Consequently,
\begin{equation}
I^{0}+\sum^{n-2}_{k=0}\left(I^{2k+1}+I^{2k+2}\right)=0.\label{2.114}
\end{equation}

\begin{theorem} \label{thm2.1} The functional, for $n\geq2$,
\begin{eqnarray}
&&\mathcal{L}^{{\rm
M}}_{\omega}(\varphi',\varphi'') \ \ := \ \ \frac{1}{V_{\omega}}\int^{1}_{0}\int_{X}\dot{\varphi}_{t}\omega^{n}_{\varphi_{t}}dt\label{2.115}\\
&-&\frac{1}{V_{\omega}}\int^{1}_{0}\int_{X}\frac{n(n-1)\sqrt{-1}}{2}\partial\omega\wedge\omega^{n-2}_{\varphi_{t}}\wedge(\overline{\partial}\dot{\varphi}_{t}\cdot\varphi_{t})dt
\nonumber\\
&+&\frac{1}{V_{\omega}}\int^{1}_{0}\int_{X}\frac{n(n-1)\sqrt{-1}}{2}\overline{\partial}\omega\wedge\omega^{n-2}_{\varphi_{t}}\wedge(\partial\dot{\varphi}_{t}\cdot\varphi_{t})dt
\nonumber\\
&+&\sum^{n-2}_{i=1}\frac{1}{V_{\omega}}\int^{1}_{0}\int_{X}(-1)^{i}\binom{n}{i+2}\partial\varphi_{t}\wedge\partial\omega\wedge\overline{\partial}\dot{\varphi}_{t}\wedge\overline{\partial}
\varphi_{t}\wedge\omega^{n-i-2}_{\varphi_{t}}\nonumber\\
&&\wedge(\sqrt{-1}\partial\overline{\partial}\varphi_{t})^{i-1}\nonumber\\
&+&\sum^{n-2}_{i=1}\frac{1}{V_{\omega}}\int^{1}_{0}\int_{X}(-1)^{i}\binom{n}{i+2}\overline{\partial}\varphi_{t}\wedge\overline{\partial}
\omega\wedge\partial\dot{\varphi}_{t}\wedge
\partial\varphi_{t}\wedge\omega^{n-i-2}_{\varphi_{t}}\nonumber\\
&&\wedge(\sqrt{-1}\partial\overline{\partial}\varphi_{t})^{i-1}.\nonumber
\end{eqnarray}
is independent of the choice of the smooth path
$\{\varphi_{t}\}_{0\leq t\leq1}$ in $\mathcal{P}_{\omega}$ from
$\varphi'$ to $\varphi''$.

\end{theorem}

\begin{proof} Using (\ref{2.114}), we can prove Theorem \ref{thm2.1} in the similar way as \cite{L1}.
\end{proof}

\begin{corollary} \label{cor2.2} Suppose $n\geq2$. For any $\varphi\in\mathcal{P}_{\omega}$ one has
\begin{eqnarray}
\mathcal{L}^{{\rm M}}_{\omega}(\varphi)&:=&\mathcal{L}^{{\rm
M}}_{\omega}(0,\varphi) \ \ = \ \
\frac{1}{V_{\omega}}\sum^{n}_{i=0}\int_{X}\frac{1}{n+1}\varphi\omega^{i}_{\varphi}\wedge\omega^{n-i}
\nonumber\\
&-&\sum^{n-2}_{i=0}\frac{i+1}{2V_{\omega}}\int_{X}\varphi\omega^{i}_{\varphi}\wedge\omega^{n-2-i}\wedge\sqrt{-1}\partial\omega\wedge\overline{\partial}\varphi
\label{2.116}\\
&+&\sum^{n-2}_{i=0}\frac{i+1}{2V_{\omega}}\int_{X}\varphi\omega^{i}_{\varphi}\wedge\omega^{n-2-i}\wedge\sqrt{-1}\overline{\partial}\omega\wedge\partial\varphi.\nonumber
\end{eqnarray}
\end{corollary}

\begin{proof} Since $\mathcal{L}^{{\rm M}}_{\omega}(\varphi)$ is
independent of the choice of smooth path, we pick
$\varphi_{t}=t\varphi$, $0\leq t\leq1$. Then $\overline{\partial}\dot{\varphi}_{t}\wedge\overline{\partial}\varphi_{t}=\partial\dot{\varphi}_{t}\wedge\partial\varphi_{t}=0$, and
\begin{eqnarray*}
\mathcal{L}^{{\rm
M}}_{\omega}(\varphi)&=&\frac{1}{V_{\omega}}\int^{1}_{0}\int_{X}\varphi\omega^{n}_{t\varphi}dt
\\
&-&\frac{n(n-1)}{2V_{\omega}}\int^{1}_{0}\int_{X}\sqrt{-1}\partial\omega\wedge\omega^{n-2}_{t\varphi}\wedge(\overline{\partial}\varphi\cdot
t\varphi)dt \\
&+&\frac{n(n-1)}{2V_{\omega}}\int^{1}_{0}\int_{X}\sqrt{-1}\overline{\partial}\omega\wedge\omega^{n-2}_{t\varphi}\wedge(\partial\varphi\cdot
t\varphi)dt \  \ =: \ \ J_{0}+J_{1}+J_{2}.
\end{eqnarray*}
Now we compute $J_{0},J_{1},J_{2}$, respectively. Using
\begin{equation*}
\omega_{t\varphi}=\omega+t\sqrt{-1}\partial\overline{\partial}\varphi=\omega+t(\omega_{\varphi}-\omega)=t\omega_{\varphi}+(1-t)\omega,
\end{equation*}
it follows that
\begin{eqnarray*}
J_{0}&=&\frac{1}{V_{\omega}}\int
_{X}\int^{1}_{0}\varphi\sum^{n}_{i=0}\binom{n}{i}\omega^{i}_{\varphi}\wedge\omega^{n-i}t^{i}(1-t)^{n-i}dt
\\
&=&\sum^{n}_{i=0}\frac{1}{V_{\omega}}\int_{X}\varphi\binom{n}{i}\omega^{i}_{\varphi}\wedge\omega^{n-i}\cdot\int^{1}_{0}t^{i}(1-t)^{n-i}dt
\\
&=&\sum^{n}_{i=0}\frac{1}{V_{\omega}}\int_{X}\binom{n}{i}\frac{\Gamma(i+1)\Gamma(n-i+1)}{\Gamma(n+2)}\varphi\omega^{i}_{\varphi}\wedge\omega^{n-i}
\\
&=&\sum^{n}_{i=0}\frac{1}{V_{\omega}}\int_{X}\frac{n!}{i!(n-i)!}\frac{i!(n-i)!}{(n+1)!}\varphi\omega^{i}_{\varphi}\wedge\omega^{n-i}
=\sum^{n}_{i=0}\frac{1}{V_{\omega}}\int_{X}\frac{1}{n+1}\varphi\omega^{i}_{\varphi}\wedge\omega^{n-i},
\end{eqnarray*}
where $\Gamma(x)$ is the Gamma function. Similarly, we have
\begin{eqnarray*}
J_{1}&=&-\frac{n(n-1)}{2V_{\omega}}\int_{X}\int^{1}_{0}\sqrt{-1}
\partial\omega\wedge[t\omega_{\varphi}+(1-t)\omega]^{n-2}\wedge(\overline{\partial}\varphi\cdot
t\varphi)dt \\
&=&-\frac{n(n-1)}{2V_{\omega}}\int_{X}\int^{1}_{0}\sqrt{-1}\partial\omega\\
&&\wedge\sum^{n-2}_{i=0}\binom{n-2}{i}t^{i+1}\omega^{i}_{\varphi}(1-t)^{n-2-i}\omega^{n-2-i}\wedge(\overline{\partial}\varphi\cdot
\varphi)dt\\
&=&-\frac{n(n-1)}{2V_{\omega}}\int_{X}\sqrt{-1}\partial\omega\wedge\sum^{n-2}_{i=0}\omega^{i}_{\varphi}\wedge\omega^{n-2-i}\wedge(\overline{\partial}\varphi\cdot\varphi)
\\
&\cdot&\int^{1}_{0}\binom{n-2}{i}t^{i+1}(1-t)^{n-2-i}dt \\
&=&-\frac{n(n-1)}{2V_{\omega}}\int_{X}\sqrt{-1}\partial\omega\wedge\sum^{n-2}_{i=0}\omega^{i}_{\varphi}\wedge\omega^{n-2-i}\wedge(\overline{\partial}\varphi\cdot\varphi)\cdot\frac{i+1}{n(n-1)}
\\
&=&\sum^{n-2}_{i=0}\frac{-1}{2V_{\omega}}\int_{X}(i+1)\varphi\omega^{i}_{\varphi}\wedge\omega^{n-2-i}\wedge\sqrt{-1}\partial\omega\wedge\overline{\partial}\varphi.
\end{eqnarray*}
Taking the complex conjugate gives
\begin{equation*}
J_{2}=\sum^{n-2}_{i=0}\frac{1}{2V_{\omega}}\int_{X}(i+1)\varphi\omega^{i}_{\varphi}\wedge\omega^{n-2-i}\wedge\sqrt{-1}\overline{\partial}\omega\wedge\partial\varphi.
\end{equation*}
Together with the expressions of $J_{0}, J_{1}$ and $J_{2}$, we
complete the proof.
\end{proof}

\begin{remark} \label{rmk2.3} When $(X,\omega)$ is a compact K\"ahler manifold, the
functional (\ref{2.115}) or (\ref{2.116}) coincides with the original
one.
\end{remark}

Let $S$ be a non-empty set and $A$ an additive group. A mapping
$\mathcal{N}: S\times S\to A$ is said to satisfy the 1-cocycle
condition if
\begin{itemize}

\item[(i)]
$\mathcal{N}(\sigma_{1},\sigma_{2})+\mathcal{N}(\sigma_{2},\sigma_{1})=0$;

\item[(ii)]
$\mathcal{N}(\sigma_{1},\sigma_{2})+\mathcal{N}(\sigma_{2},\sigma_{3})+\mathcal{N}(\sigma_{3},\sigma_{1})=0$.

\end{itemize}

\begin{corollary} \label{cor2.4} (1) The functional $\mathcal{L}^{{\rm M}}_{\omega}$
satisfies the $1$-cocycle condition.\\
(2) For any $\varphi\in\mathcal{P}_{\omega}$ and any constant
$C\in\mathbb{R}$, we have
\begin{equation}
\mathcal{L}^{{\rm
M}}_{\omega}(\varphi,\varphi+C)=C\cdot\left(1+\frac{{\rm
Err}_{\omega}(\varphi)}{V_{\omega}}\right), \ \ \ {\rm Err}_{\omega}(\varphi):=\int_{X}\omega^{n}-\int_{X}\omega^{n}_{\varphi}.\label{2.117}
\end{equation}
In particular, if
$\partial\overline{\partial}\omega=\partial\omega\wedge\overline{\partial}\omega=0$,
then $\mathcal{L}^{{\rm M}}_{\omega}(\varphi,\varphi+C)=C$.\\
(3) For any $\varphi_{1},\varphi_{2}\in\mathcal{P}_{\omega}$ and any
constant $C\in\mathbb{R}$, we have
\begin{equation}
\mathcal{L}^{{\rm
M}}_{\omega}(\varphi_{1},\varphi_{2}+C)=\mathcal{L}^{{\rm
M}}_{\omega}(\varphi_{1},\varphi_{2})+C\cdot\left(1-\frac{{\rm
Err}_{\omega}(\varphi_{2})}{V_{\omega}}\right).\label{2.118}
\end{equation}
In particular, if
$\partial\overline{\partial}\omega=\partial\omega\wedge\overline{\partial}\omega=0$,
then $\mathcal{L}^{{\rm
M}}_{\omega}(\varphi_{1},\varphi_{2}+C)=\mathcal{L}^{{\rm
M}}_{\omega}(\varphi_{1},\varphi_{2})+C$.
\end{corollary}

\begin{proof} The proof is similar to that given in \cite{L1,L2}.
\end{proof}
\section{Aubin-Yau functionals on compact complex manifolds}

\subsection{The main idea}
The strategy to construct Aubin-Yau functionals is to use the inequalities (\ref{1.5})
and (\ref{1.6}) to determine the extra terms. Firstly we can show that
\begin{eqnarray}
\frac{n}{n+1}\mathcal{I}^{{\rm AY}}_{\omega|\bullet}(\varphi)-\mathcal{J}^{{\rm AY}}_{\omega|\bullet}(\varphi)
&=&\frac{1}{V_{\omega}}\sum^{n-1}_{i=1}\frac{i}{n+1}\int_{X}\sqrt{-1}\partial\varphi
\wedge\overline{\partial}\varphi\wedge\omega^{i}_{\varphi}\wedge\omega^{n-1-i} \nonumber\\
&-&\frac{\mathcal{A}_{\omega}(\varphi)+\mathcal{B}_{\omega}(\varphi)}{n+1}+\frac{\mathcal{C}_{\omega}(\varphi)
+\mathcal{D}_{\omega}(\varphi)}{2}, \label{3.1}\\
(n+1)\mathcal{J}^{{\rm AY}}_{\omega|\bullet}(\varphi)-\mathcal{I}^{{\rm AY}}_{\omega|\bullet}(\varphi)
&=&\frac{1}{V_{\omega}}\sum^{n-1}_{i=0}(n-1-i)\int_{X}\sqrt{-1}\partial\varphi
\wedge\overline{\partial}\varphi\wedge\omega^{i}_{\varphi}\wedge\omega^{n-1-i}\nonumber \\
&+&\frac{\mathcal{E}_{\omega}(\varphi)+\mathcal{F}_{\omega}(\varphi)}{2}+(n+1)\left[\mathcal{A}^{1}_{\omega}(\varphi)
+\mathcal{B}^{1}_{\omega}(\varphi)\right]\label{3.2}\\
&-&\frac{\mathcal{A}^{2}_{\omega}(\varphi)+\mathcal{B}^{2}_{\omega}(\varphi)}{n-1},\nonumber
\end{eqnarray}
where $\mathcal{I}^{{\rm AY}}_{\omega|\bullet}(\varphi),\mathcal{J}^{{\rm AY}}_{\omega|\bullet}(\varphi),
\mathcal{A}_{\omega}(\varphi),\mathcal{B}_{\omega}(\varphi),\mathcal{C}_{\omega}(\varphi),
\mathcal{D}_{\omega}(\varphi),\mathcal{E}_{\omega}(\varphi),\mathcal{F}_{\omega}(\varphi),\mathcal{A}^{1}_{\omega}(\varphi),
\mathcal{B}^{1}_{\omega}(\varphi)$,\\
$\mathcal{A}^{2}_{\omega}(\varphi),\mathcal{B}^{2}_{\omega}(\varphi)$
are functionals determined in next subsection. Inspired by (\ref{3.1}) and (\ref{3.2}),
we define Aubin-Yau functionals as follows:
\begin{eqnarray}
\mathcal{I}^{{\rm AY}}_{\omega}(\varphi)&:=&\mathcal{I}^{{\rm AY}}_{\omega|\bullet}(\varphi) \nonumber\\
&+&a^{1}_{1}\mathcal{A}^{1}_{\omega}(\varphi)+a^{2}_{1}\mathcal{A}^{2}_{\omega}(\varphi)
+b^{1}_{1}\mathcal{B}^{1}_{\omega}(\varphi)+b^{2}_{1}\mathcal{B}^{2}_{\omega}(\varphi)\nonumber\\
&+&c_{1}\mathcal{C}_{\omega}(\varphi)+d_{1}\mathcal{D}_{\omega}(\varphi)+e_{1}
\mathcal{E}_{\omega}(\varphi)+f_{1}\mathcal{F}_{\omega}(\varphi), \label{3.3} \\
\mathcal{J}^{{\rm AY}}_{\omega}(\varphi)&=&\mathcal{J}^{{\rm AY}}_{\omega|\bullet}(\varphi)\nonumber \\
&+&(a^{1}_{2}-1)\mathcal{A}^{1}_{\omega}(\varphi)+(a^{2}_{2}-1)\mathcal{A}^{2}_{\omega}(\varphi)
+(b^{1}_{2}-1)\mathcal{B}^{1}_{\omega}(\varphi)+(b^{2}_{2}-1)\mathcal{B}^{2}_{\omega}(\varphi)\nonumber \\
&+&c_{2}\mathcal{C}_{\omega}(\varphi)+d_{2}\mathcal{D}_{\omega}(\varphi)+e_{2}\mathcal{E}_{\omega}(\varphi)
+f_{2}\mathcal{F}_{\omega}(\varphi).\label{3.4}
\end{eqnarray}
Here $a^{i}_{j},b^{i}_{j},c_{k},d_{k},e_{k}$, and $f_{k}$ are constants determined by the following
two inequalities:
\begin{eqnarray*}
\frac{n}{n+1}\mathcal{I}^{{\rm AY}}_{\omega}(\varphi)
-\mathcal{J}^{{\rm AY}}_{\omega}(\varphi)
&=&\frac{1}{V_{\omega}}\sum^{n-1}_{i=1}\frac{i}{n+1}\int_{X}\sqrt{-1}\partial
\varphi\wedge\overline{\partial}\varphi\wedge\omega^{i}_{\varphi}\wedge\omega^{n-1-i}\geq0, \\
(n+1)\mathcal{J}^{{\rm AY}}_{\omega}(\varphi)-\mathcal{I}^{{\rm AY}}_{\omega}(\varphi)
&=&\sum^{n-1}_{i=1}\frac{n-1-i}{V_{\omega}}\int_{X}\sqrt{-1}\partial\varphi\wedge
\overline{\partial}\varphi\wedge\omega^{i}_{\varphi}\wedge\omega^{n-1-i}\geq0.
\end{eqnarray*}
It gives us a system of linear equations of 16 parameters and eventually we evaluate
these parameters:
\begin{eqnarray*}
a^{1}_{1}&=&b^{1}_{1} \ \ = \ \ -\frac{n}{n-1}, \ \ \ a^{1}_{2} \ \ = \ \ b^{1}_{2} \ \ = \ \
-\frac{n}{n^{2}-1}, \\
a^{2}_{1}&=&b^{2}_{1} \ \ = \ \ \frac{n}{(n-1)^{2}}, \ \ \ a^{2}_{2} \ \ = \ \
b^{2}_{2} \ \ = \ \ \frac{n}{n+1}\left(1+\frac{n}{(n-1)^{2}}\right), \\
c_{1}&=&d_{1} \ \ = \ \ -\frac{n+1}{2(n-1)}, \ \ \ e_{2} \ \ = \ \ f_{2} \ \ = \ \ -\frac{n}{2(n^{2}-1)}, \\
c_{2}&=&d_{2} \ \ = \ \ e_{1} \ \ = \ \ f_{1} \ \ = \ \ -\frac{1}{2(n-1)}.
\end{eqnarray*}
By a long computation we find an explicit and shorted formulae for $\mathcal{I}^{{\rm AY}}_{\omega}(\varphi)$
and $\mathcal{J}^{{\rm AY}}_{\omega}(\varphi)$:
\begin{eqnarray}
& & \mathcal{I}^{{\rm AY}}_{\omega}(\varphi) \ \ = \ \ \frac{1}{V_{\omega}}\int_{X}(\omega^{n}-\omega^{n}_{\varphi}) \label{3.5}\\
&-&\frac{n}{2V_{\omega}}\sum^{n-2}_{i=1}\int_{X}\varphi\omega^{i}_{\varphi}\wedge
\omega^{n-2-i}\wedge\sqrt{-1}\partial\omega\wedge\overline{\partial}\varphi-\frac{n}{2V_{\omega}}\int_{X}\varphi\omega^{n-2}
\wedge\sqrt{-1}\partial\omega\wedge\overline{\partial}\varphi \nonumber\\
&+&\frac{n}{2V_{\omega}}\sum^{n-2}_{i=1}\int_{X}\varphi\omega^{i}_{\varphi}\wedge\omega^{n-2-i}
\wedge\sqrt{-1}\overline{\partial}\omega\wedge\partial\varphi+\frac{n}{2V_{\omega}}
\int_{X}\varphi\omega^{n-2}\wedge\sqrt{-1}\overline{\partial}\omega\wedge\partial\varphi,\nonumber \\
& & \mathcal{J}^{{\rm AY}}_{\omega}(\varphi) \ \ = \ \ -\mathcal{L}^{{\rm M}}_{\omega}(\varphi)+
\frac{1}{V_{\omega}}\int_{X}\varphi\omega^{n} \label{3.6}\\
&-&\frac{n}{2V_{\omega}}\sum^{n-2}_{i=1}\int_{X}\varphi\omega^{i}_{\varphi}\wedge\omega^{n-2-i}
\wedge\sqrt{-1}\partial\omega\wedge\overline{\partial}\varphi-\frac{n}{2V_{\omega}}
\int_{X}\varphi\omega^{n-2}\wedge\sqrt{-1}\partial\varphi\wedge\overline{\partial}\varphi\nonumber\\
&+&\frac{n}{2V_{\omega}}\sum^{n-2}_{i=1}\int_{X}\varphi\omega^{i}_{\varphi}
\wedge\omega^{n-2-i}\wedge\sqrt{-1}\overline{\partial}\omega\wedge\partial\varphi
+\frac{n}{2V_{\omega}}\int_{X}\varphi\omega^{n-2}\wedge\sqrt{-1}\overline{\partial}\omega\wedge\partial\varphi.\nonumber
\end{eqnarray}

\subsection{The construction of Aubin-Yau functionals}
Let $(X,g)$ be a compact complex manifold of the complex dimension
$n\geq3$ and $\omega$ be its associated real $(1,1)$-form. We recall
some notation in \cite{L1}. For any $\varphi\in\mathcal{P}_{\omega}$
we set
\begin{eqnarray}
\mathcal{I}^{{\rm
AY}}_{\omega|\bullet}(\varphi)&:=&\frac{1}{V_{\omega}}\int_{X}\varphi(\omega^{n}-\omega^{n}_{\varphi}),
\label{3.7}\\
\mathcal{J}^{{\rm
AY}}_{\omega|\bullet}(\varphi)&:=&\int^{1}_{0}\frac{\mathcal{I}^{{\rm
AY}}_{\omega|\bullet}(s\cdot\varphi)}{s}ds \ \ = \ \
\frac{1}{V_{\omega}}\int^{1}_{0}\int_{X}\varphi(\omega^{n}-\omega^{n}_{s\cdot\varphi})ds.\label{3.8}
\end{eqnarray}
Two relations showed in \cite{L1} are
\begin{eqnarray}
& & \frac{n}{n+1}\mathcal{I}^{{\rm
AY}}_{\omega|\bullet}(\varphi)-\mathcal{J}^{{\rm
AY}}_{\omega|\bullet}(\varphi) \label{3.9}\\
&=&\frac{1}{V_{\omega}}\int_{X}\varphi\cdot(-\sqrt{-1}\partial\overline{\partial}\varphi)\wedge\sum^{n-1}_{j=1}\frac{j}{n+1}\omega^{n-1-j}\wedge\omega^{j}_{\varphi},
\nonumber\\
& & (n+1)\mathcal{J}^{{\rm
AY}}_{\omega|\bullet}(\varphi)-\mathcal{I}^{{\rm
AY}}_{\omega|\bullet}(\varphi) \label{3.10}\\
&=&\frac{1}{V_{\omega}}\int_{X}\varphi\cdot(-\sqrt{-1}\partial\overline{\partial}\varphi)\wedge\sum^{n-1}_{j=0}(n-1-j)\omega^{n-1-j}\wedge\omega^{j}_{\varphi}.\nonumber
\end{eqnarray}
According to the expression of $\mathcal{L}^{{\rm
M}}_{\omega}(\varphi)$, we set
\begin{eqnarray}
\mathcal{A}_{\omega}(\varphi)&:=&\sum^{n-2}_{i=0}\frac{i+1}{2V_{\omega}}\int_{X}\varphi
\omega^{i}_{\varphi}\wedge\omega^{n-2-i}\wedge-\sqrt{-1}\partial\omega\wedge\overline{\partial}\varphi,
\label{3.11}\\
\mathcal{B}_{\omega}(\varphi)&:=&\sum^{n-2}_{i=0}\frac{i+1}{2V_{\omega}}\int_{X}\varphi
\omega^{i}_{\varphi}\wedge\omega^{n-2-i}\wedge\sqrt{-1}\overline{\partial}\omega\wedge\partial\varphi.\label{3.12}
\end{eqnarray}

Using (\ref{3.9}) we obtain
\begin{eqnarray*}
& & \frac{n}{n+1}\mathcal{I}^{{\rm
AY}}_{\omega|\bullet}(\varphi)-\mathcal{J}^{{\rm
AY}}_{\omega|\bullet}(\varphi) \\
&=&\frac{1}{V_{\omega}}\int_{X}\sqrt{-1}\partial\left(\varphi\sum^{n-1}_{j=1}\frac{j}{n+1}\omega^{n-1-j}\wedge\omega^{j}_{\varphi}\right)\wedge\overline{\partial}\varphi
\\
&=&\frac{1}{V_{\omega}}\int_{X}\sqrt{-1}\left(\partial\varphi\wedge\sum^{n-1}_{j=1}\frac{j}{n+1}\omega^{n-1-j}\wedge\omega^{j}_{\varphi}\right)\wedge\overline{\partial}\varphi
\\
&+&\frac{1}{V_{\omega}}\int_{X}\sqrt{-1}\varphi\sum^{n-1}_{j=1}\frac{j}{n+1}\left[
(n-1-j)\omega^{n-2-j}\wedge\partial\omega\wedge\omega^{j}_{\varphi}\right.
\\
&+&\left.\omega^{n-1-j}\wedge
j\omega^{j-1}_{\varphi}\wedge\partial\omega\right]\wedge\overline{\partial}\varphi;
\end{eqnarray*}
from the identity $i(n-1-i)+(i+1)^{2}=(i+1)+in$, it follows that
\begin{eqnarray*}
& & \frac{n}{n+1}\mathcal{I}^{{\rm
AY}}_{\omega|\bullet}(\varphi)-\mathcal{J}^{{\rm
AY}}_{\omega|\bullet}(\varphi) \\
&=&\frac{1}{V_{\omega}}\sum^{n-1}_{i=1}\frac{i}{n+1}\int_{X}\sqrt{-1}\partial\varphi\wedge\overline{\partial}\varphi
\wedge\omega^{n-1-i}\wedge\omega^{i}_{\varphi} \\
&+&\frac{1}{V_{\omega}}\sum^{n-2}_{i=1}\frac{i(n-1-i)}{n+1}\int_{X}\varphi\omega^{i}_{\varphi}\wedge\omega^{n-2-i}\wedge
\sqrt{-1}\partial\omega\wedge\overline{\partial}\varphi \\
&+&\frac{1}{V_{\omega}}\sum^{n-2}_{i=0}\frac{(i+1)^{2}}{n+1}\int_{X}\varphi\omega^{i}_{\varphi}\wedge\omega^{n-2-i}\wedge
\sqrt{-1}\partial\omega\wedge\overline{\partial}\varphi \\
&=&\frac{1}{V_{\omega}}\sum^{n-1}_{i=1}\frac{i}{n+1}\int_{X}\sqrt{-1}\partial\varphi
\wedge\overline{\partial}\varphi\wedge\omega^{i}_{\varphi}\wedge\omega^{n-1-i}
\\
&+&\frac{1}{V_{\omega}}\sum^{n-2}_{i=1}\frac{i+1+in}{n+1}\int_{X}\varphi\omega^{i}_{\varphi}
\wedge\omega^{n-2-i}\wedge\sqrt{-1}\partial\omega\wedge\overline{\partial}\varphi
\\
&+&\frac{1}{(n+1)V_{\omega}}\int_{X}\varphi\omega^{n-2}\wedge\sqrt{-1}\partial\omega\wedge\overline{\partial}\varphi.
\end{eqnarray*}
To simplify the notation, we set
\begin{equation}
\mathcal{C}_{\omega}(\varphi):=\frac{1}{V_{\omega}}\sum^{n-2}_{i=1}\frac{in}{n+1}\int_{X}\varphi
\omega^{i}_{\varphi}\wedge\omega^{n-2-i}\wedge\sqrt{-1}\partial\omega\wedge\overline{\partial}\varphi.\label{3.13}
\end{equation}
Since $n\geq3$, the above expression is well defined. Therefore
\begin{eqnarray}
& & \frac{n}{n+1}\mathcal{I}^{{\rm
AY}}_{\omega|\bullet}(\varphi)-\mathcal{J}^{{\rm
AY}}_{\omega|\bullet}(\varphi) \label{3.14}\\
&=&\frac{1}{V_{\omega}}\sum^{n-1}_{i=1}\frac{i}{n+1}\int_{X}\sqrt{-1}\partial\varphi\wedge
\overline{\partial}\varphi\wedge\omega^{i}_{\varphi}\wedge\omega^{n-1-i}
-\frac{2}{n+1}\mathcal{A}_{\omega}(\varphi)+\mathcal{C}_{\omega}(\varphi).\nonumber
\end{eqnarray}
On the other hand, using the slightly different method, we obtain
(see \ref{A.1})
\begin{eqnarray}
& & \frac{n}{n+1}\mathcal{I}^{{\rm
AY}}_{\omega|\bullet}(\varphi)-\mathcal{J}^{{\rm
AY}}_{\omega|\bullet}(\varphi) \label{3.15}\\
&=&\frac{1}{V_{\omega}}\sum^{n-1}_{i=1}\frac{i}{n+1}\int_{X}\sqrt{-1}\partial\varphi
\wedge\overline{\partial}\varphi\wedge\omega^{i}_{\varphi}\wedge\omega^{n-1-i}
-\frac{2}{n+1}\mathcal{B}_{\omega}(\varphi)+\mathcal{D}_{\omega}(\varphi)\nonumber
\end{eqnarray}
where
\begin{equation}
\mathcal{D}_{\omega}(\varphi):=\frac{1}{V_{\omega}}\sum^{n-2}_{i=1}\frac{in}{n+1}
\int_{X}\varphi\omega^{i}_{\varphi}\wedge\omega^{n-2-i}\wedge-\sqrt{-1}\overline{\partial}\omega
\wedge\partial\varphi.\label{3.16}
\end{equation}
Equations (\ref{3.14}) and (\ref{3.15}) implies
\begin{eqnarray}
\frac{n}{n+1}\mathcal{I}^{{\rm
AY}}_{\omega|\bullet}(\varphi)-\mathcal{J}^{{\rm
AY}}_{\omega|\bullet}(\varphi)&=&\frac{1}{V_{\omega}}\sum^{n-1}_{i=1}\frac{i}{n+1}\int_{X}\sqrt{-1}\partial\varphi\wedge
\overline{\partial}\varphi\wedge\omega^{i}_{\varphi}\wedge\omega^{n-1-i}\nonumber
\\
&-&\frac{\mathcal{A}_{\omega}(\varphi)+\mathcal{B}_{\omega}(\varphi)}{n+1}+\frac{\mathcal{C}_{\omega}(\varphi)+\mathcal{D}_{\omega}(\varphi)}{2}.\label{3.17}
\end{eqnarray}

By the definition we have
\begin{eqnarray*}
\mathcal{J}^{{\rm
AY}}_{\omega|\bullet}(\varphi)&=&\frac{1}{V_{\omega}}\int^{1}_{0}\int_{X}(\varphi\omega^{n}-
\varphi\omega^{n}_{s\varphi})ds \ \ = \ \
\frac{1}{V_{\omega}}\int_{X}\varphi\omega^{n}-\frac{1}{V_{\omega}}\int^{1}_{0}\int_{X}
\varphi\omega^{n}_{t\varphi}dt \\
&=&\frac{1}{V_{\omega}}\int_{X}\varphi\omega^{n}-(\mathcal{L}^{{\rm
M}}_{\omega}(\varphi)-\mathcal{A}_{\omega}(\varphi)-\mathcal{B}_{\omega}(\varphi))
\\
&=&\frac{1}{V_{\omega}}\int_{X}\varphi\omega^{n}-\mathcal{L}^{{\rm
M}}_{\omega}(\varphi)+\mathcal{A}_{\omega}(\varphi)+\mathcal{B}_{\omega}(\varphi).
\end{eqnarray*}

If we define
\begin{eqnarray}
\mathcal{E}_{\omega}(\varphi)&:=&\sum^{n-3}_{i=0}\frac{n^{2}}{V_{\omega}}\int_{X}\varphi\omega^{i}_{\varphi}\wedge\omega^{n-2-i}
\wedge\sqrt{-1}\partial\omega\wedge\overline{\partial}\varphi, \label{3.18}\\
\mathcal{A}^{1}_{\omega}(\varphi)&:=&\sum^{n-3}_{i=0}\frac{i+1}{2V_{\omega}}
\int_{X}\varphi\omega^{i}_{\varphi}\wedge\omega^{n-2-i}\wedge-\sqrt{-1}\partial
\omega\wedge\overline{\partial}\varphi \label{3.19}\\
\mathcal{A}^{2}_{\omega}(\varphi)&:=&\frac{n-1}{2V_{\omega}}\int_{X}\varphi\omega^{n-2}_{\varphi}\wedge-\sqrt{-1}
\partial\omega\wedge\overline{\partial}\varphi,\label{3.20}
\end{eqnarray}
then
$\mathcal{A}^{1}_{\omega}(\varphi)+\mathcal{A}^{2}_{\omega}(\varphi)=
\mathcal{A}_{\omega}(\varphi)$ and it follows that (see \ref{A.1})
\begin{eqnarray}
(n+1)\mathcal{J}^{{\rm AY}}_{\omega|\bullet}-\mathcal{I}^{{\rm
AY}}_{\omega|\bullet}(\varphi)&=&\frac{1}{V_{\omega}}\sum^{n-1}_{i=0}(n-1-i)\int_{X}\sqrt{-1}\partial\varphi
\wedge\overline{\partial}\varphi\wedge\omega^{i}_{\varphi}\wedge\omega^{n-1-i}
\nonumber\\
&+&\mathcal{E}_{\omega}(\varphi)+2(n+1)\mathcal{A}^{1}_{\omega}
(\varphi)-\frac{2}{n-1}\mathcal{A}^{2}_{\omega}(\varphi).\label{3.21}
\end{eqnarray}

Introduce the corresponding functionals 
\begin{eqnarray}
\mathcal{F}_{\omega}(\varphi)&:=&\sum^{n-3}_{i=0}\frac{n^{2}}{V_{\omega}}\int_{X}\varphi\omega^{i}_{\varphi}\wedge
\omega^{n-2-i}\wedge-\sqrt{-1}\overline{\partial}\omega
\wedge\partial\varphi, \label{3.22}\\
\mathcal{B}^{1}_{\omega}(\varphi)&:=&\sum^{n-3}_{i=0}\frac{i+1}{2V_{\omega}}\int_{X}
\varphi\omega^{i}_{\varphi}\wedge\omega^{n-2-i}\wedge\sqrt{-1}\overline{\partial}\omega\wedge
\partial\varphi \label{3.23}\\
\mathcal{B}^{2}_{\omega}(\varphi)&:=&\frac{n-1}{2V_{\omega}}\int_{X}\varphi
\omega^{n-2}_{\varphi}\wedge\sqrt{-1}\overline
{\partial}\omega\wedge\partial\varphi.\label{3.24}
\end{eqnarray}
Then
$\mathcal{B}^{1}_{\omega}(\varphi)+\mathcal{B}^{2}_{\omega}(\varphi)=\mathcal{B}_{\omega}(\varphi)$
and hence (see \ref{A.3})
\begin{eqnarray}
(n+1)\mathcal{J}^{{\rm
AY}}_{\omega|\bullet}(\varphi)-\mathcal{I}^{{\rm
AY}}_{\omega|\bullet}(\varphi)&=&\frac{1}{V_{\omega}}\sum^{n-1}_{i=0}(n-1-i)
\int_{X}\sqrt{-1}\partial\varphi\wedge\overline{\partial}\varphi\wedge\omega^{i}_{\varphi}
\wedge\omega^{n-1-i} \nonumber \\
&+&\mathcal{F}_{\omega}(\varphi)+2(n+1)
\mathcal{B}^{1}_{\omega}(\varphi)-\frac{2}{n-1}\mathcal{B}^{2}_{\omega}(\varphi)
.\label{3.25}
\end{eqnarray}
The equations (\ref{3.21}) and (\ref{3.25}) together gives
\begin{eqnarray}
(n+1)\mathcal{J}^{{\rm
AY}}_{\omega|\bullet}(\varphi)-\mathcal{I}^{{\rm
AY}}_{\omega|\bullet}(\varphi)&=&\frac{1}{V_{\omega}}\sum^{n-1}_{i=0}(n-1-i)
\int_{X}\sqrt{-1}\partial\varphi\wedge\overline{\partial}\varphi
\wedge\omega^{i}_{\varphi}\wedge\omega^{n-1-i}\nonumber \\
&+&\frac{\mathcal{E}_{\omega}(\varphi)+\mathcal{F}_{\omega}(\varphi)}{2}
+(n+1)(\mathcal{A}^{1}_{\omega}(\varphi)+\mathcal{B}^{1}_{\omega}(\varphi)
\label{3.26}\\
&-&\frac{\mathcal{A}^{2}_{\omega}(\varphi)
+\mathcal{B}^{2}_{\omega}(\varphi)}{n-1}.\nonumber
\end{eqnarray}

Now, we define Aubin-Yau functionals over any compact complex
manifolds as follows:
\begin{eqnarray}
\mathcal{I}^{{\rm AY}}_{\omega}(\varphi)&:=&\mathcal{I}^{{\rm
AY}}_{\omega|\bullet}(\varphi)\nonumber\\
&+&a^{1}_{1}\mathcal{A}^{1}_{\omega}(\varphi)+a^{2}_{1}\mathcal{A}^{2}_{\omega}(\varphi)
+b^{1}_{1}\mathcal{B}^{1}_{\omega}(\varphi)+b^{2}_{1}\mathcal{B}^{2}_{\omega}(\varphi)\nonumber\\
&+&c_{1}\mathcal{C}_{\omega}(\varphi)+d_{1}\mathcal{D}_{\omega}(\varphi)
+e_{1}\mathcal{E}_{\omega}(\varphi)+f_{1}\mathcal{F}_{\omega}(\varphi),\label{3.27}
\\
\mathcal{J}^{{\rm AY}}_{\omega}(\varphi)&:=&-\mathcal{L}^{{\rm
M}}_{\omega}(\varphi)+\frac{1}{V_{\omega}}\int_{X}\varphi\omega^{n}\nonumber\\
&+&a^{1}_{2}\mathcal{A}^{1}_{\omega}(\varphi)+a^{2}_{2}\mathcal{A}^{2}_{\omega}(\varphi)
+b^{1}_{2}\mathcal{B}^{1}_{\omega}(\varphi)+b^{2}_{2}\mathcal{B}^{2}_{\omega}(\varphi)\nonumber\\
&+&c_{2}\mathcal{C}_{\omega}(\varphi)+d_{2}\mathcal{D}_{\omega}(\varphi)
+e_{2}\mathcal{E}_{\omega}(\varphi)+f_{2}\mathcal{F}_{\omega}(\varphi),
\nonumber\\
&=&\mathcal{J}^{{\rm AY}}_{\omega|\bullet}(\varphi) \nonumber \\
&+&(a^{1}_{2}-1)\mathcal{A}^{1}_{\omega}(\varphi)+(a^{2}_{2}-1)\mathcal{A}^{2}_{\omega}(\varphi)
+(b^{1}_{2}-1)\mathcal{B}^{1}_{\omega}(\varphi)+(b^{2}_{2}-1)\mathcal{B}_{\omega}^{2}(\varphi)
\nonumber\\
&+&c_{2}\mathcal{C}_{\omega}(\varphi)+d_{2}\mathcal{D}_{\omega}(\varphi)
+e_{2}\mathcal{E}_{\omega}(\varphi)+f_{2}\mathcal{F}_{\omega}(\varphi).\label{3.28}
\end{eqnarray}

Plugging (\ref{3.27}) and (\ref{3.28}) into (\ref{3.26}) and
(\ref{3.17}), we obtain
\begin{equation}
\frac{n}{n+1}\mathcal{I}^{{\rm
AY}}_{\omega}(\varphi)-\mathcal{J}^{{\rm
AY}}_{\omega}(\varphi)=\frac{1}{V_{\omega}}\sum^{n-1}_{i=1}\frac{i}{n+1}\int_{X}\sqrt{-1}\partial\varphi
\wedge\overline{\partial}\varphi\wedge\omega^{i}_{\varphi}
\wedge\omega^{n-1-i}\geq0,\label{3.29}
\end{equation}
and
\begin{equation}
(n+1)\mathcal{J}^{{\rm AY}}_{\omega}(\varphi)-\mathcal{I}^{{\rm
AY}}_{\omega}(\varphi)=\sum^{n-1}_{i=0}\frac{n-1-i}{V_{\omega}}\int_{X}\sqrt{-1}\partial
\varphi\wedge\overline{\partial}\varphi\wedge\omega^{i}_{\varphi}
\wedge\omega^{n-1-i}\geq0,\label{3.30}
\end{equation}
where we require that constants satisfy the following linear
equations system
\begin{eqnarray}
\frac{n}{n+1}a^{1}_{1}-(a^{1}_{2}-1)&=&\frac{1}{n+1},\label{3.31} \ \ \
\frac{n}{n+1}a^{2}_{1}-(a^{2}_{2}-1) \ \ = \ \ \frac{1}{n+1},\\
\frac{n}{n+1}b^{1}_{1}-(b^{1}_{2}-1)&=&\frac{1}{n+1}, \label{3.32}
\ \ \
\frac{n}{n+1}b^{2}_{1}-(b^{2}_{2}-1) \ \ = \ \ \frac{1}{n+1},\\
\frac{n}{n+1}c_{1}-c_{2}&=&-\frac{1}{2}, \label{3.33} \ \ \
\frac{n}{n+1}d_{1}-d_{2} \ \ = \ \ -\frac{1}{2}, \\
\frac{n}{n+1}e_{1}-e_{2}&=&0,\label{3.34}\ \ \
\frac{n}{n+1}f_{1}-f_{2} \ \ = \ \ 0,\\
(n+1)(a^{1}_{2}-1)-a^{1}_{1}&=&-(n+1),\label{3.35} \ \ \
(n+1)(a^{2}_{2}-1)-a^{2}_{1} \ \ = \ \ \frac{1}{n-1}, \\
(n+1)(b^{1}_{2}-1)-b^{1}_{1}&=&-(n+1), \label{3.36} \ \ \
(n+1)(b^{2}_{2}-1)-b^{2}_{1} \ \ = \ \ \frac{1}{n-1},\\
(n+1)c_{2}-c_{1}&=&0, \label{3.37} \ \ \
(n+1)d_{2}-d_{1} \ \ = \ \ 0, \\
(n+1)e_{2}-e_{1}&=&-\frac{1}{2},\label{3.38} \ \ \
(n+1)f_{2}-f_{1} \ \ = \ \ -\frac{1}{2}.
\end{eqnarray}
The constants $a^{j}_{i},b^{j}_{i},c_{i},d_{i},e_{i}$ and $f_{i}$ ,
calculated in Appendix B, are
\begin{eqnarray}
a^{1}_{1}&=&b^{1}_{1} \ \ = \ \ -\frac{n}{n-1}, \ \ \ a^{1}_{2} \ \
= \ \
b^{1}_{2} \ \ = \ \ -\frac{n}{n^{2}-1},\label{3.39}\\
a^{2}_{1}&=&b^{2}_{1} \ \ = \ \ \frac{n}{(n-1)^{2}}, \ \ \ a^{2}_{2}
\ \ = \ \ b^{2}_{2} \ \ = \ \
\frac{n}{n+1}\left(1+\frac{n}{(n-1)^{2}}\right) \label{3.40}\\
c_{1}&=&d_{1} \ \ = \ \ -\frac{n+1}{2(n-1)}, \ \ \ e_{2} \ \ = \ \
f_{2} \ \ = \ \ -\frac{n}{2(n^{2}-1)}, \label{3.41}\\
c_{2}&=&d_{2} \ \ = \ \ e_{1} \ \ = \ \ f_{1} \ \ = \ \
-\frac{1}{2(n-1)}.\label{3.42}
\end{eqnarray}

The explicit formulas for $\mathcal{I}^{{\rm AY}}_{\omega}(\varphi)$
and $\mathcal{J}^{{\rm AY}}_{\omega}(\varphi)$ are given in
Proposition \ref{C.1} and \ref{C.2} respectively. Namely,
\begin{eqnarray}
\mathcal{I}^{{\rm AY}}_{\omega}(\varphi)&=&\frac{1}{V_{\omega}}\int_{X}\varphi(\omega^{n}-\omega^{n}_{\varphi})-\frac{n}{2V_{\omega}}\sum^{n-2}_{i=0}\int_{X}\varphi\omega^{i}_{\varphi}\wedge\omega^{n-2-i}
\wedge\sqrt{-1}\partial\omega\wedge\overline{\partial}\varphi\nonumber\\
&+&\frac{n}{2V_{\omega}}\sum^{n-2}_{i=0}\int_{X}\varphi\omega^{i}_{\varphi}\wedge\omega^{n-2-i}
\wedge\sqrt{-1}\overline{\partial}\omega\wedge\partial\varphi\nonumber,\\
&=&\frac{1}{V_{\omega}}\sum^{n-1}_{i=0}\int_{X}\sqrt{-1}\partial\varphi\wedge\overline{\partial}\varphi\wedge\omega^{i}_{\varphi}\wedge\omega^{n-1-i},\label{3.43}\\
\mathcal{J}^{{\rm AY}}_{\omega}(\varphi) &=&
-\mathcal{L}^{{\rm
M}}_{\omega}(\varphi)+\frac{1}{V_{\omega}}\int_{X}\varphi\omega^{n}
-\frac{n}{2V_{\omega}}\sum^{n-2}_{i=0}\int_{X}\varphi\omega^{i}_{\varphi}\wedge\omega^{n-2-i}
\wedge\sqrt{-1}\partial\omega\wedge\overline{\partial}\varphi\nonumber\\
&+&\frac{n}{2V_{\omega}}\sum^{n-2}_{i=0}\int_{X}\varphi\omega^{i}_{\varphi}\wedge\omega^{n-2-i}
\wedge\sqrt{-1}\overline{\partial}\omega\wedge\partial\varphi\nonumber \\
&=&\frac{1}{V_{\omega}}\sum^{n-1}_{i=0}\int_{X}\frac{n-i}{n+1}\sqrt{-1}\partial\varphi\wedge\overline{\partial}\varphi\wedge\omega^{i}_{\varphi}\wedge\omega^{n-i-1}.\label{3.44}
\end{eqnarray}
Here the formulas (\ref{3.43}) and (\ref{3.44}) come from the solution of the system of linear equations (\ref{3.29}) and (\ref{3.30}).

From (\ref{3.29}), (\ref{3.30}), (\ref{3.37}) and (\ref{3.44}), we
deduce the following

\begin{theorem} \label{thm3.1} For any $\varphi\in\mathcal{P}_{\omega}$, one has
\begin{eqnarray}
\frac{n}{n+1}\mathcal{I}^{{\rm
AY}}_{\omega}(\varphi)-\mathcal{J}^{{\rm
AY}}_{\omega}(\varphi)&\geq&0, \label{3.45}\\
(n+1)\mathcal{J}^{{\rm AY}}_{\omega}(\varphi)-\mathcal{I}^{{\rm
AY}}_{\omega}(\varphi)&\geq&0.\label{3.46}
\end{eqnarray}
In particular
\begin{eqnarray}
\frac{1}{n+1}\mathcal{I}^{{\rm
AY}}_{\omega}(\varphi)&\leq&\mathcal{J}^{{\rm AY}}_{\omega}(\varphi)
\ \ \leq \ \ \frac{n}{n+1}\mathcal{I}^{{\rm AY}}_{\omega}(\varphi),
\label{3.47}\\
\frac{n+1}{n}\mathcal{J}^{{\rm
AY}}_{\omega}(\varphi)&\leq&\mathcal{I}^{{\rm AY}}_{\omega}(\varphi)
\ \ \leq \ \ (n+1)\mathcal{J}^{{\rm AY}}_{\omega}(\varphi), \label{3.48}\\
\frac{1}{n}\mathcal{J}^{{\rm AY}}_{\omega}(\varphi) \ \ \leq \ \
\frac{1}{n+1}\mathcal{J}^{{\rm AY}}_{\omega}(\varphi)&\leq&
\mathcal{I}^{{\rm
AY}}_{\omega}(\varphi)-\mathcal{J}^{{\rm AY}}_{\omega}(\varphi)\label{3.49}\\
&\leq&\frac{n}{n+1}\mathcal{I}^{{\rm AY}}_{\omega}(\varphi) \ \ \leq
\ \ n\mathcal{J}^{{\rm AY}}_{\omega}(\varphi).
\end{eqnarray}
\end{theorem}

\appendix
\section{Proof the identities (\ref{3.14}), (\ref{3.21}) and (\ref{3.25})}
In Appendix A we verify the identities (\ref{3.14}), (\ref{3.21}) and
(\ref{3.25}).

\begin{eqnarray}
& & \frac{n}{n+1}\mathcal{I}^{{\rm
AY}}_{\omega|\bullet}(\varphi)-\mathcal{J}^{{\rm
AY}}_{\omega|\bullet}(\varphi) \label{A.1}\\
&=&\frac{1}{V_{\omega}}\int_{X}\left(\varphi\cdot\sum^{n-1}_{j=1}\frac{j}{n+1}\omega^{n-1-j}
\wedge\omega^{j}_{\varphi}\right)\wedge\sqrt{-1}\overline{\partial}\partial\varphi
\nonumber\\
&=&\frac{1}{V_{\omega}}\int_{X}-\sqrt{-1}\overline{\partial}\left(\varphi\sum^{n-1}_{j=1}
\frac{j}{n+1}\omega^{n-1-j}\wedge\omega^{j}_{\varphi}\right)\wedge\partial\varphi
\nonumber\\
&=&\frac{1}{V_{\omega}}\int_{X}-\sqrt{-1}\left(\overline{\partial}\varphi\wedge
\sum^{n-1}_{j=1}\frac{j}{n+1}\omega^{n-1-j}\wedge\omega^{j}_{\varphi}\right)\wedge\partial\varphi
\nonumber\\
&+&\frac{1}{V_{\omega}}\int_{X}-\sqrt{-1}\varphi\sum^{n-1}_{j=1}\frac{j}{n+1}\left[(n-1-j)\omega^{n-2-j}
\wedge\overline{\partial}\omega\wedge\omega^{j}_{\varphi}\right.\nonumber\\
&+&\left.\omega^{n-1-j}\wedge
j\omega^{j-1}_{\varphi}\wedge\overline{\partial}\omega\right]\wedge\partial\varphi
\nonumber\\
&=&\frac{1}{V_{\omega}}\sum^{n-1}_{i=1}\frac{i}{n+1}\int_{X}\sqrt{-1}\partial\varphi\wedge\overline{\partial}
\varphi\wedge\omega^{n-1-i}\wedge\omega^{i}_{\varphi} \nonumber\\
&+&\frac{1}{V_{\omega}}\sum^{n-2}_{i=1}\frac{i(n-1-i)}{n+1}\int_{X}-\sqrt{-1}\varphi
\omega^{n-2-i}\wedge\omega^{i}_{\varphi}\wedge\overline{\partial}\omega\wedge\partial\varphi
\nonumber\\
&+&\frac{1}{V_{\omega}}\sum^{n-2}_{i=0}\frac{(i+1)^{2}}{n+1}\int_{X}-\sqrt{-1}\varphi
\omega^{n-2-i}\wedge\omega^{i}_{\varphi}\wedge\overline{\partial}\omega\wedge\partial\varphi
\nonumber\\
&=&\frac{1}{V_{\omega}}\sum^{n-1}_{i=1}\frac{i}{n+1}\int_{X}\sqrt{-1}\partial\varphi
\wedge\overline{\partial}\varphi\wedge\omega^{n-1-i}\wedge\omega^{i}_{\varphi}
-\frac{2}{n+1}\mathcal{B}_{\omega}(\varphi)+\mathcal{D}_{\omega}(\varphi)\nonumber
\end{eqnarray}
which gives (\ref{3.15}). Calculate
\begin{eqnarray*}
& & (n+1)\mathcal{J}^{{\rm AY}}_{\omega|\bullet}-\mathcal{I}^{{\rm
AY}}_{\omega|\bullet}(\varphi) \label{A.2}\\
&=&\frac{1}{V_{\omega}}\int_{X}\left(\varphi\sum^{n-1}_{j=0}(n-1-j)\omega^{n-1-j}\wedge\omega^{j}_{\varphi}\right)\wedge(-\sqrt{-1}\partial\overline{\partial}\varphi)
\\
&=&\frac{1}{V_{\omega}}\int_{X}\sqrt{-1}\partial\left(\varphi\sum^{n-1}_{j=0}(n-1-j)\omega^{n-1-j}\wedge\omega^{j}_{\varphi}\right)
\wedge\overline{\partial}\varphi \\
&=&\frac{1}{V_{\omega}}\int_{X}\sqrt{-1}\partial\varphi\wedge\sum^{n-1}_{j=0}
(n-1-j)\omega^{n-1-j}\wedge\omega^{j}_{\varphi}\wedge\overline{\partial}\varphi
\\
&+&\frac{1}{V_{\omega}}\int_{X}\sqrt{-1}\varphi\sum^{n-1}_{j=0}\left[(n-1-j)^{2}\omega^{n-2-j}
\wedge\partial\omega\wedge\omega^{j}_{\varphi}\right.\\
&+&\left.(n-1-j)j\omega^{n-1-j}\wedge\omega^{j-1}_{\varphi}\wedge\partial\omega\right]\wedge\overline{\partial}\varphi
\\
&=&\frac{1}{V_{\omega}}\sum^{n-1}_{i=0}(n-1-j)\int_{X}\sqrt{-1}\partial\varphi
\wedge\overline{\partial}\varphi\wedge\omega^{i}_{\varphi}\wedge\omega^{n-1-i}
\\
&+&\frac{1}{V_{\omega}}\sum^{n-1}_{j=0}(n-1-j)^{2}\int_{X}\varphi\omega^{n-2-j}\wedge
\omega^{j}_{\varphi}\wedge\sqrt{-1}\partial\omega\wedge\overline{\partial}\varphi
\\
&+&\frac{1}{V_{\omega}}\sum^{n-1}_{j=0}(n-1-j)j\int_{X}\varphi\omega^{n-1-j}\wedge
\omega^{j-1}_{\varphi}\wedge\sqrt{-1}\partial\omega\wedge\overline{\partial}\varphi
\\
&=&\frac{1}{V_{\omega}}\sum^{n-1}_{i=0}(n-1-j)\int_{X}\sqrt{-1}\partial\varphi
\wedge\overline{\partial}\varphi\wedge\omega^{i}_{\varphi}\wedge\omega^{n-1-i}
\\
&+&\frac{1}{V_{\omega}}\sum^{n-2}_{j=0}(n-1-j)^{2}\int_{X}\varphi\omega^{n-2-j}\wedge\omega^{j}_{\varphi}\wedge\sqrt{-1}\partial\omega
\wedge\overline{\partial}\varphi \\
&+&\frac{1}{V_{\omega}}\sum^{n-3}_{i=0}(i+1)(n-i-2)\int_{X}\varphi\omega^{n-2-i}\wedge\omega^{i}_{\varphi}\wedge\sqrt{-1}
\partial\omega\wedge\overline{\partial}\varphi \\
&=&\frac{1}{V_{\omega}}\sum^{n-1}_{i=0}(n-1-i)\int_{X}\sqrt{-1}\partial\varphi
\wedge\overline{\partial}\varphi\wedge\omega^{i}_{\varphi}\wedge\omega^{n-1-i}
\\
&+&\frac{1}{V_{\omega}}\sum^{n-3}_{i=0}[n^{2}-(n+1)(i+1)]\int_{X}\varphi\omega^{n-2-i}\wedge\omega^{i}_{\varphi}
\wedge(\sqrt{-1}\partial\omega\wedge\overline{\partial}\varphi) \\
&+&\frac{1}{V_{\omega}}\int_{X}\varphi\omega^{n-2}_{\varphi}\wedge\sqrt{-1}\partial\omega
\wedge\overline{\partial}\varphi
\end{eqnarray*}
where we use the elementary identity
\begin{eqnarray*}
& & (n-1-i)^{2}+(i+1)(n-i-2)\\
&=&(n-1)^{2}+i^{2}-2(n-1)i+(n-2)(i+1)-i(i+1) \\
&=&n^{2}-2n+1+i^{2}-2ni+2i+ni+n-2i-2-i^{2}-i \\
&=&n^{2}-n-1-ni-i \ \ = \ \ -(n+1)(i+1)+n^{2}.
\end{eqnarray*}
Using the definitions of $\mathcal{E}_{\omega}(\varphi),
\mathcal{A}^{1}_{\omega}(\varphi)$,
$\mathcal{A}^{2}_{\omega}(\varphi)$, we have
$\mathcal{A}^{1}_{\omega}(\varphi)+\mathcal{A}^{2}_{\omega}(\varphi)=
\mathcal{A}_{\omega}(\varphi)$ and hence (\ref{3.21}) holds.
Similarly, we have
$\mathcal{B}^{1}_{\omega}(\varphi)+\mathcal{B}^{2}_{\omega}(\varphi)=\mathcal{B}_{\omega}(\varphi)$
and
\begin{eqnarray*}
& & (n+1)\mathcal{J}^{{\rm
AY}}_{\omega|\bullet}(\varphi)-\mathcal{I}^{{\rm
AY}}_{\omega|\bullet}(\varphi) \\
&=&\frac{1}{V_{\omega}}\int_{X}-\sqrt{-1}\overline{\partial}\left(\varphi
\sum^{n-1}_{j=0}(n-1-j)\omega^{n-1-j}\wedge\omega^{j}_{\varphi}\right)\wedge\partial\varphi
\\
&=&\frac{1}{V_{\omega}}\int_{X}-\sqrt{-1}\overline{\partial}\varphi\wedge
\sum^{n-1}_{j=0}(n-1-j)\omega^{n-1-j}\wedge\omega^{j}_{\varphi}\wedge\partial\varphi
\\
&+&\frac{1}{V_{\omega}}\int_{X}-\sqrt{-1}\varphi\sum^{n-1}_{j=0}(n-1-j)
((n-1-j)\omega^{n-2-j}\wedge\overline{\partial}\omega\wedge\omega^{j}_{\varphi}\\
&+&j\omega^{n-1-j}\wedge\omega^{j-1}_{\varphi}\wedge\overline{\partial}\omega)
\wedge\partial\varphi \\
&=&\frac{1}{V_{\omega}}\sum^{n-1}_{i=0}(n-1-i)\int_{X}\sqrt{-1}\partial\varphi
\wedge\overline{\partial}\varphi\wedge\omega^{n-1-i}\wedge\omega^{i}_{\varphi}
\\
&+&\frac{1}{V_{\omega}}\sum^{n-1}_{j=0}(n-1-j)^{2}\int_{X}\varphi\omega^{n-2-j}\wedge\omega^{j}_{\varphi}
\wedge(-\sqrt{-1}\overline{\partial}\omega\wedge\partial\varphi) \\
&+&\frac{1}{V_{\omega}}\sum^{n-1}_{j=0}(n-1-j)j\int_{X}\varphi\omega^{n-1-j}\wedge
\omega^{j-1}_{\varphi}\wedge(-\sqrt{-1}\overline{\partial}\omega\wedge\partial\varphi)
\\
&=&\frac{1}{V_{\omega}}\sum^{n-1}_{i=0}(n-1-i)\int_{X}\sqrt{-1}
\partial\varphi\wedge\overline{\partial}\varphi\wedge\omega^{n-1-i}\wedge
\omega^{i}_{\varphi} \\
&+&\frac{1}{V_{\omega}}\sum^{n-2}_{i=0}(n-1-i)^{2}\int_{X}\varphi\omega^{n-2-i}
\wedge\omega^{i}_{\varphi}\wedge(-\sqrt{-1}\overline{\partial}\omega\wedge\partial\varphi)
\\
&+&\frac{1}{V_{\omega}}\sum^{n-3}_{i=0}(n-i-2)(i+1)\int_{X}\varphi\omega^{n-2-i}\wedge
\omega^{i}_{\varphi}\wedge(-\sqrt{-1}\overline{\partial}\omega\wedge\partial\varphi)
\\
&=&\frac{1}{V_{\omega}}\sum^{n-1}_{i=0}(n-1-i)\int_{X}\sqrt{-1}\partial\varphi\wedge\overline{\partial}\varphi
\wedge\omega^{n-1-i}\wedge\omega^{i}_{\varphi} \\
&+&\frac{1}{V_{\omega}}\sum^{n-3}_{i=0}[n^{2}-(n+1)(i+1)]\int_{X}\varphi
\omega^{n-2-i}\wedge\omega^{i}_{\varphi}\wedge(-\sqrt{-1}\overline{\partial}\omega\wedge\partial\varphi)
\\
&+&\frac{1}{V_{\omega}}\int_{X}\varphi\omega^{n-2}_{\varphi}\wedge(-\sqrt{-1}
\overline{\partial}\omega\wedge\partial\varphi).
\end{eqnarray*}
and hence
\begin{eqnarray}
(n+1)\mathcal{J}^{{\rm
AY}}_{\omega|\bullet}(\varphi)-\mathcal{I}^{{\rm
AY}}_{\omega|\bullet}(\varphi)&=&\frac{1}{V_{\omega}}\sum^{n-1}_{i=0}(n-1-i)
\int_{X}\sqrt{-1}\partial\varphi\wedge\overline{\partial}\varphi\wedge\omega^{i}_{\varphi}
\wedge\omega^{n-1-i} \nonumber \\
&+&\mathcal{F}_{\omega}(\varphi)+2(n+1)\mathcal{B}^{1}_{\omega}(\varphi)-\frac{2}{n-1}\mathcal{B}^{2}_{\omega}(\varphi).\label{A.3}
\end{eqnarray}
Therefore (\ref{3.21}) and (\ref{3.25}) together gives
\begin{eqnarray}
(n+1)\mathcal{J}^{{\rm
AY}}_{\omega|\bullet}(\varphi)-\mathcal{I}^{{\rm
AY}}_{\omega|\bullet}(\varphi)&=&\frac{1}{V_{\omega}}\sum^{n-1}_{i=0}(n-1-i)
\int_{X}\sqrt{-1}\partial\varphi\wedge\overline{\partial}\varphi
\wedge\omega^{i}_{\varphi}\wedge\omega^{n-1-i}\nonumber \\
&+&\frac{\mathcal{E}_{\omega}(\varphi)+\mathcal{F}_{\omega}(\varphi)}{2}
+(n+1)(\mathcal{A}^{1}_{\omega}(\varphi)+\mathcal{B}^{1}_{\omega}(\varphi)\label{A.4}\\
&-&\frac{\mathcal{A}^{2}_{\omega}(\varphi)
+\mathcal{B}^{2}_{\omega}(\varphi)}{n-1}.\nonumber
\end{eqnarray}

\section{Solve the system of the linear equations} \label{oldnew}
In this section we try to solve the system of the linear equations
(\ref{3.31})-(\ref{3.38}). Firstly we solve (\ref{3.31}) and
(\ref{3.35}) as follows: (\ref{3.31}) and (\ref{3.35}) gives us the
following equations
\begin{eqnarray}
\frac{n}{n+1}a^{1}_{1}-\frac{1}{n+1}&=&a^{1}_{2}-1, \label{B.1}\\
(n+1)(a^{1}_{2}-1)+(n+1) &=&a^{1}_{1}, \nonumber\\
\frac{n}{n+1}a^{2}_{1}-\frac{1}{n+1}&=&a^{2}_{2}-1,\label{B.2}\\
(n+1)(a^{2}_{2}-1)-\frac{1}{n-1}&=& a^{2}_{1}.\nonumber
\end{eqnarray}
Plugging the first equation into second equation in (\ref{B.1}), we
have
\begin{equation*}
(n+1)\left(\frac{n}{n+1}a^{1}_{1}-\frac{1}{n+1}\right)+(n+1)=a^{1}_{1}
\end{equation*}
which implies
\begin{equation}
a^{1}_{1}=-\frac{n}{n-1}, \ \ \
a^{1}_{2}=-\frac{n}{n^{2}-1}.\label{B.3}
\end{equation}
Similarly,
\begin{equation*}
(n+1)\left(\frac{n}{n+1}a^{2}_{1}-\frac{1}{n+1}\right)-\frac{1}{n+1}=a^{2}_{1},
\end{equation*}
therefore
\begin{equation}
a^{2}_{1}=\frac{n}{(n-1)^{2}}, \ \ \
a^{2}_{2}=\frac{n}{n+1}\left(1+\frac{n}{(n-1)^{2}}\right)=\frac{n^{3}-n^{2}+n}{n^{3}-n^{2}-n+1}.\label{B.4}
\end{equation}
Secondly, (\ref{3.32}) and (\ref{3.36}) implies
\begin{eqnarray}
\frac{n}{n+1}b^{1}_{1}-\frac{1}{n+1}&=&b^{1}_{2}-1,\label{B.5}\\
(n+1)(b^{1}_{2}-1)&=&^{1}_{1}-(n+1), \nonumber\\
\frac{n}{n+1}b^{2}_{1}-\frac{1}{n+1}&=&b^{2}_{2}-1, \label{B.6}\\
(n+1)(b^{2}_{2}-1) &=&b^{2}_{1}+\frac{1}{n-1}.\nonumber
\end{eqnarray}
The above linear equations system gives
\begin{equation*}
(n+1)\left(\frac{n}{n+1}b^{1}_{1}-\frac{1}{n+1}\right)=b^{1}_{1}-(n+1)
\end{equation*}
and
\begin{equation*}
(n+1)\left(\frac{n}{n+1}b^{2}_{1}-\frac{1}{n+1}\right)=b^{2}_{1}+\frac{1}{n-1},
\end{equation*}
respectively. Hence
\begin{eqnarray}
b^{1}_{1}&=&-\frac{n}{n-1}, \label{B.7}\\
b^{1}_{2}&=&
-\frac{n}{n^{2}-1}, \nonumber\\
b^{2}_{1}&=&\frac{n}{(n-1)^{2}}, \label{B.8}\\
 b^{2}_{2}&=&
\frac{n}{n+1}\left(1+\frac{n}{(n-1)^{2}}\right)\nonumber.
\end{eqnarray}
Continuously, equations (\ref{3.33}) and (\ref{3.37}) shows that
\begin{eqnarray*}
\frac{n}{n+1}c_{1}-c_{2}&=&-\frac{1}{2}, \ \ \ (n+1)c_{2}-c_{1} \ \
= \ \ 0, \\
\frac{n}{n+1}d_{1}-d_{2}&=&-\frac{1}{2}, \ \ \ (n+1)d_{2}-d_{1} \ \
= \ \ 0.
\end{eqnarray*}
Eliminating $c_{2}$ and $d_{2}$ respectively, we have
\begin{eqnarray*}
(n+1)\left(\frac{n}{n+1}c_{1}+\frac{1}{2}\right)-c_{1}&=&0, \\
(n+1)\left(\frac{n}{n+1}d_{1}+\frac{1}{2}\right)-d_{1}&=&0.
\end{eqnarray*}
Thus
\begin{eqnarray}
c_{1}&=&-\frac{n+1}{2(n-1)}, \label{B.9}\\
c_{2}&=&
-\frac{1}{2(n-1)}, \nonumber\\
d_{1}&=&-\frac{n+1}{2(n-1)}, \label{B.10}\\
d_{2}&=& -\frac{1}{2(n-1)}.\nonumber
\end{eqnarray}
Similarly, from (\ref{3.34}) and (\ref{3.38}) we obtain
\begin{eqnarray*}
\frac{n}{n+1}e_{1}-e_{2}=0, \ \ \ (n+1)e_{2}-e_{1} \ \ = \ \
-\frac{1}{2}, \\
\frac{n}{n+1}f_{1}-f_{2}=0, \ \ \ (n+1)f_{2}-f_{1} \ \ = \ \
-\frac{1}{2},
\end{eqnarray*}
and hence
\begin{eqnarray}
e_{1}&=&f_{1} \ \ = \ \ -\frac{1}{2(n-1)},\label{B.11}\\
e_{2}&=&f_{2} \ \ = \ \ -\frac{n}{2(n^{2}-1)}.\label{B.12}
\end{eqnarray}
\section{Explicit formulas of $\mathcal{I}^{{\rm
AY}}_{\omega}(\varphi)$ and $\mathcal{J}^{{\rm
AY}}_{\omega}(\varphi)$} \label{oldnew}
In this section we give the explicit formulas of $\mathcal{I}^{{\rm
AY}}_{\omega}(\varphi)$ and $\mathcal{J}^{{\rm
AY}}_{\omega}(\varphi)$. In what follows, we assume that $n\geq3$.
Using the constants determined in Appendix B, we have
\begin{eqnarray*}
\mathcal{I}^{{\rm AY}}_{\omega}(\varphi)&=&
\frac{1}{V_{\omega}}\int_{X}\varphi(
\omega^{n}-\omega^{n}_{\varphi})\\
&=&\frac{n}{n-1}\sum^{n-3}_{i=0}\frac{i+1}{2V_{\omega}}\int_{X}\varphi\omega^{i}_{\varphi}\wedge\omega^{n-2-i}\wedge\sqrt{-1}\partial\omega\wedge\overline{\partial}\varphi
\\
&-&\frac{n}{n-1}\frac{1}{2V_{\omega}}\int_{X}\varphi\omega^{n-2}_{\varphi}\wedge\sqrt{-1}\partial
\omega\wedge\overline{\partial}\varphi\\
&-&\frac{n}{n-1}\sum^{n-3}_{i=0}\frac{i+1}{2V_{\omega}}\int_{X}\varphi\omega^{i}_{\varphi}\wedge\omega^{n-2-i}\wedge\sqrt{-1}\overline{\partial}
\omega\wedge\partial\varphi \\
&+&\frac{n}{n-1}\frac{1}{2V_{\omega}}\int_{X}\varphi\omega^{n-2}_{\varphi}\wedge\sqrt{-1}
\overline{\partial}\omega\wedge\partial\varphi \\
&-&\frac{n}{n-1}\sum^{n-2}_{i=1}\frac{i}{2V_{\omega}}\int_{X}\varphi\omega^{i}_{\varphi}
\wedge\omega^{n-3-i}\wedge\sqrt{-1}\partial\omega\wedge\overline{\partial}\varphi
\\
&+&\frac{n}{n-1}\sum^{n-2}_{i=1}\frac{i}{2V_{\omega}}\int_{X}\varphi\omega^{i}_{\varphi}\wedge
\omega^{n-2-i}\wedge\sqrt{-1}\overline{\partial}\omega\wedge\partial\varphi
\\
&-&\frac{n^{2}}{n-1}\sum^{n-3}_{i=0}\frac{1}{2V_{\omega}}\int_{X}\varphi\omega^{i}_{\varphi}\wedge
\omega^{n-2-i}\wedge\sqrt{-1}\partial\omega\wedge\overline{\partial}\varphi
\\
&+&\frac{n^{2}}{n-1}\sum^{n-3}_{i=0}\frac{1}{2V_{\omega}}\int_{X}\varphi\omega^{i}_{\varphi}
\wedge\omega^{n-2-i}\wedge\sqrt{-1}\overline{\partial}\omega\wedge\partial\varphi.
\end{eqnarray*}
When $n=3$, it is easy to see that
\begin{eqnarray*}
\mathcal{I}^{{\rm
AY}}_{\omega}(\varphi)&=&\frac{1}{V_{\omega}}\int_{X}\varphi(\omega^{3}-\omega^{3}_{\varphi})
\\
&+&\frac{3}{4V_{\omega}}\int_{X}\varphi\omega\wedge\sqrt{-1}\partial\omega\wedge\overline{\partial}\varphi
-\frac{3}{4V_{\omega}}\int_{X}\varphi\omega_{\varphi}\wedge\sqrt{-1}\partial\omega\wedge\overline{\partial}\varphi
\\
&-&\frac{3}{4V_{\omega}}\int_{X}\varphi\omega\wedge\sqrt{-1}\overline{\partial}\omega\wedge\partial\varphi
+\frac{3}{4V_{\omega}}\int_{X}\varphi\omega_{\varphi}\wedge\sqrt{-1}\overline{\partial}\omega\wedge\partial\varphi
\\
&-&\frac{3}{4V_{\omega}}\int_{X}\varphi\omega_{\varphi}\wedge\sqrt{-1}\partial\omega
\wedge\overline{\partial}\varphi+\frac{3}{4V_{\omega}}\int_{X}\varphi\omega_{\varphi}\wedge\sqrt{-1}\overline{\partial}\omega\wedge\partial\varphi
\\
&-&\frac{9}{4V_{\omega}}\int_{X}\varphi\omega\wedge\sqrt{-1}\partial\omega\wedge
\overline{\partial}\varphi+
\frac{9}{4V_{\omega}}\int_{X}\varphi\omega\wedge\sqrt{-1}\overline{\partial}\omega\wedge\partial\varphi
\\
&=&\frac{1}{V_{\omega}}\int_{X}\varphi(\omega^{3}-\omega^{3}_{\varphi})
\\
&-&\frac{3}{2V_{\omega}}\int_{X}\varphi\omega\wedge\sqrt{-1}\partial\omega\wedge\overline{\partial}\varphi
+\frac{3}{2V_{\omega}}\int_{X}\varphi\omega\wedge\sqrt{-1}\overline{\partial}\omega\wedge\partial\varphi
\\
&-&\frac{3}{2V_{\omega}}\int_{X}\varphi\omega_{\varphi}\wedge\sqrt{-1}\partial\omega\wedge\overline{\partial}\varphi
+\frac{3}{2V_{\omega}}\int_{X}\varphi\omega_{\varphi}\wedge\sqrt{-1}\overline{\partial}\omega\wedge\partial\varphi.
\end{eqnarray*}
For general $n\geq4$, a simple computation shows
\begin{eqnarray*}
& & \mathcal{I}^{{\rm AY}}_{\omega}(\varphi) \ \ = \ \
\frac{1}{V_{\omega}}\int_{X}\varphi(\omega^{n}-\omega^{n}_{\varphi})
\\
&+&\sum^{n-3}_{i=1}\frac{1}{2V_{\omega}}\left[\frac{n(i+1)}{n-1}-\frac{in}{n-1}-\frac{n^{2}}{n-1}\right]
\int_{X}\varphi\omega^{i}_{\varphi}\wedge\omega^{n-2-i}\wedge\sqrt{-1}\partial\omega
\wedge\overline{\partial}\varphi \\
&+&\frac{n}{n-1}\frac{1}{2V_{\omega}}\int_{X}\varphi\omega^{n-2}\wedge\sqrt{-1}\partial
\omega\wedge\overline{\partial}\varphi \\
&-&\frac{n(n-2)}{n-1}\frac{1}{2V_{\omega}}\int_{X}\varphi\omega^{n-2}_{\varphi}\wedge\sqrt{-1}\partial\omega
\wedge\overline{\partial}\varphi \\
&-&\frac{n}{n-1}\frac{1}{2V_{\omega}}\int_{X}\varphi\omega^{n-2}_{\varphi}\wedge\sqrt{-1}\partial\omega
\wedge\overline{\partial}\varphi \\
&+&\sum^{n-3}_{i=1}\frac{1}{2V{_\omega}}\left[-\frac{n(i+1)}{n-1}+\frac{in}{n-1}+\frac{n^{2}}{n-1}\right]
\int_{X}\varphi\omega^{i}_{\varphi}\wedge\omega^{n-2-i}\wedge\sqrt{-1}\overline{\partial}\omega\wedge\partial\varphi
\\
&-&\frac{n}{n-1}\frac{1}{2V_{\omega}}\int_{X}\varphi\omega^{n-2}\wedge\sqrt{-1}\overline{\partial}\omega\wedge\partial\varphi
\\
&+&\frac{n(n-2)}{n-1}\frac{2}{2V_{\omega}}\int_{X}\varphi\omega^{n-2}_{\varphi}\wedge\sqrt{-1}
\overline{\partial}\omega\wedge\partial\varphi \\
&+&\frac{n}{n-1}\frac{1}{2V_{\omega}}\int_{X}\varphi\omega^{n-2}_{\varphi}\wedge\sqrt{-1}\overline{\partial}\omega
\wedge\partial\varphi \\
&-&\frac{n^{2}}{n-1}\frac{1}{2V_{\omega}}\int_{X}\varphi\omega^{n-2}\wedge\sqrt{-1}\partial\omega\wedge
\overline{\partial}\varphi+\frac{n^{2}}{n-1}\frac{1}{2V_{\omega}}\int_{X}\varphi\omega^{n-2}\wedge\sqrt{-1}
\overline{\partial}\omega\wedge\partial\varphi \\
&=&\frac{1}{V_{\omega}}\int_{X}\varphi(\omega^{n}-\omega^{n}_{\varphi})
\\
&-&\sum^{n-3}_{i=1}\frac{n}{2V_{\omega}}\int_{X}\varphi\omega^{i}_{\varphi}\wedge\omega^{n-2-i}
\wedge\sqrt{-1}\partial\omega\wedge\overline{\partial}\varphi \\
&+&\sum^{n-3}_{i=1}\frac{n}{2V_{\omega}}\int_{X}\varphi\omega^{i}_{\varphi}\wedge\omega^{n-2-i}
\wedge\sqrt{-1}\overline{\partial}\omega\wedge\partial\varphi \\
&-&\frac{n}{2V_{\omega}}\int_{X}\varphi\omega^{n-2}\wedge\sqrt{-1}\partial\omega\wedge
\overline{\partial}\varphi+\frac{n}{2V_{\omega}}\int_{X}\varphi\omega^{n-2}\wedge\sqrt{-1}\overline{\partial}\omega\wedge\partial\varphi
\\
&-&\frac{n}{2V_{\omega}}\int_{X}\varphi\omega^{n-2}_{\varphi}\wedge\sqrt{-1}\partial\omega\wedge\overline{\partial}\varphi
+\frac{n}{2V_{\omega}}\int_{X}\varphi\omega^{n-2}_{\varphi}\wedge\sqrt{-1}\overline{\partial}\omega\wedge\partial\varphi.
\end{eqnarray*}
Thus

\begin{proposition} \label{C.1} If $n\geq3$, one has
\begin{eqnarray*}
& & \mathcal{I}^{{\rm AY}}_{\omega}(\varphi) \ \ = \ \
\frac{1}{V_{\omega}}\int_{X}\varphi(\omega^{n}-\omega^{n}_{\varphi})\\
&-&\frac{n}{2V_{\omega}}\sum^{n-2}_{i=1}\int_{X}\varphi\omega^{i}_{\varphi}\wedge\omega^{n-2-i}
\wedge\sqrt{-1}\partial\omega\wedge\overline{\partial}\varphi-
\frac{n}{2V_{\omega}}\int_{X}\varphi\omega^{n-2}\wedge\sqrt{-1}\partial\omega\wedge
\overline{\partial}\varphi \nonumber\\
&+&\frac{n}{2V_{\omega}}\sum^{n-2}_{i=1}\int_{X}\varphi\omega^{i}_{\varphi}\wedge\omega^{n-2-i}
\wedge\sqrt{-1}\overline{\partial}\omega\wedge\partial\varphi
+\frac{n}{2V_{\omega}}\int_{X}\varphi\omega^{n-2}\wedge\sqrt{-1}\overline{\partial}\omega\wedge\partial\varphi
\nonumber.
\end{eqnarray*}
\end{proposition}

Similarly, we have
\begin{eqnarray*}
& & \mathcal{J}^{{\rm AY}}_{\omega}(\varphi)\\
&=&-\mathcal{L}^{{\rm
M}}_{\omega}(\varphi)+\frac{1}{V_{\omega}}\int_{X}\varphi\omega^{n}
\\
&+&\frac{n}{n^{2}-1}\sum^{n-3}_{i=0}\frac{i+1}{2V_{\omega}}\int_{X}\varphi
\omega^{i}_{\varphi}\wedge\omega^{n-2-i}\wedge\sqrt{-1}\partial\omega\wedge\overline{\partial}\varphi
\\
&-&\frac{n}{n+1}\left(n-1+\frac{n}{n-1}\right)\frac{1}{2V_{\omega}}\int_{X}\varphi\omega^{n-2}_{\varphi}
\wedge\sqrt{-1}\partial\omega\wedge\overline{\partial}\varphi \\
&-&\frac{n}{n^{2}-1}\sum^{n-3}_{i=0}\frac{i+1}{2V_{\omega}}\int_{X}\varphi\omega^{i}_{\varphi}
\wedge\omega^{n-2-i}\wedge\sqrt{-1}\overline{\partial}\omega\wedge\partial\varphi
\\
&+&\frac{n}{n+1}\left(n-1+\frac{n}{n-1}\right)\frac{2}{2V_{\omega}}\int_{X}\varphi
\omega^{n-2}_{\varphi}\wedge\sqrt{-1}\overline{\partial}\omega\wedge\partial\varphi
\\
&-&\frac{n}{n^{2}-1}\sum^{n-2}_{i=1}\frac{i}{2V_{\omega}}\int_{X}\varphi\omega^{i}_{\varphi}
\wedge\omega^{n-2-i}\wedge\sqrt{-1}\partial\omega\wedge\overline{\partial}\varphi
\\
&+&\frac{n}{n^{2}-1}\sum^{n-2}_{i=1}\frac{i}{2V_{\omega}}\int_{X}\varphi\omega^{i}_{\varphi}
\wedge\omega^{n-2-i}\wedge\sqrt{-1}\overline{\partial}\omega\wedge\partial\varphi
\\
&-&\frac{n^{3}}{n^{2}-1}\sum^{n-3}_{i=0}\frac{1}{2V_{\omega}}\int_{X}\varphi\omega^{i}_{\varphi}
\wedge\omega^{n-2-i}\wedge\sqrt{-1}\partial\omega\wedge\overline{\partial}\varphi
\\
&+&\frac{n^{3}}{n^{2}-1}\sum^{n-3}_{i=0}\frac{1}{2V_{\omega}}\int_{X}\varphi\omega^{i}_{\varphi}\wedge\omega^{n-2-i}
\wedge\sqrt{-1}\overline{\partial}\omega\wedge\partial\varphi
\end{eqnarray*}
When $n=3$, we have
\begin{eqnarray*}
& &\mathcal{J}^{{\rm AY}}_{\omega}(\varphi)\\
&=&-\mathcal{L}^{{\rm
M}}_{\omega}(\varphi)+\frac{1}{V_{\omega}}\int_{X}\varphi\omega^{3}
\\
&+&\frac{3}{8}\frac{1}{2V_{\omega}}\int_{X}\varphi\omega\wedge\sqrt{-1}\partial\omega\wedge\overline{\partial}\varphi
-\frac{3}{4}\left(2+\frac{3}{2}\right)\frac{1}{2V_{\omega}}\int_{X}\varphi\omega_{\varphi}
\wedge\sqrt{-1}\partial\omega\wedge\overline{\partial}\varphi \\
&-&\frac{3}{8}\frac{1}{2V_{\omega}}\int_{X}\varphi\omega\wedge\sqrt{-1}\overline{\partial}\omega
\wedge\partial\varphi+\frac{3}{4}\left(2+\frac{3}{2}\right)\frac{1}{2V_{\omega}}\int_{X}\varphi
\omega_{\varphi}\wedge\sqrt{-1}\overline{\partial}\omega\wedge\partial\varphi
\\
&-&\frac{3}{8}\frac{1}{2V_{\omega}}\int_{X}\varphi\omega_{\varphi}\wedge\sqrt{-1}
\partial\omega\wedge\overline{\partial}\varphi+\frac{3}{8}\frac{1}{2V_{\omega}}\int_{X}
\varphi\omega_{\varphi}\wedge\sqrt{-1}\overline{\partial}\omega\wedge\partial\varphi
\\
&-&\frac{27}{8}\frac{1}{2V_{\omega}}\int_{X}\varphi\omega\wedge\sqrt{-1}\partial\omega\wedge
\overline{\partial}\varphi+\frac{27}{8}\frac{1}{2V_{\omega}}\int_{X}\varphi\omega\wedge
\sqrt{-1}\overline{\partial}\omega\wedge\partial\varphi \\
&=&-\mathcal{L}^{{\rm
M}}_{\omega}(\varphi)+\frac{1}{V_{\omega}}\int_{X}\varphi\omega^{3}
\\
&-&\frac{3}{2V_{\omega}}\int_{X}\varphi\omega\wedge\sqrt{-1}\partial\omega\wedge
\overline{\partial}\varphi+\frac{3}{2V_{\omega}}\int_{X}\varphi\omega\wedge\sqrt{-1}
\overline{\partial}\omega\wedge\partial\varphi \\
&-&\frac{3}{2V_{\omega}}\int_{X}\varphi\omega_{\varphi}\wedge\sqrt{-1}\partial\omega\wedge\overline{\partial}\varphi
+\frac{3}{2V_{\omega}}\int_{X}\varphi\omega_{\varphi}\wedge\sqrt{-1}\overline{\partial}\omega\wedge\partial\varphi
\end{eqnarray*}
When $n\geq4$, we have
\begin{eqnarray*}
& &\mathcal{J}^{{\rm AY}}_{\omega}(\varphi) \ \ = \ \
-\mathcal{L}^{{\rm
M}}_{\omega}(\varphi)+\frac{1}{V_{\omega}}\int_{X}\varphi\omega^{n}
\\
&+&\frac{1}{2V_{\omega}}\sum^{n-3}_{i=1}\left[\frac{n(i+1)}{n^{2}-1}-\frac{in}{n^{2}-1}-\frac{n^{3}}{n^{2}-1}\right]
\int_{X}\varphi\omega^{i}_{\varphi}\wedge\omega^{n-2-i}\wedge\sqrt{-1}\partial\omega\wedge\overline{\partial}\varphi
\\
&+&\frac{n}{n^{2}-1}\frac{1}{2V_{\omega}}\int_{X}\varphi\omega^{n-2}\wedge\sqrt{-1}\partial\omega\wedge\overline{\partial}\varphi
\\
&-&\frac{n}{n+1}\left(n-1+\frac{n}{n-1}\right)\frac{1}{2V_{\omega}}\int_{Z}\varphi
\omega^{n-2}_{\varphi}\wedge\sqrt{-1}\partial\omega\wedge\overline{\partial}\varphi\\
&+&\frac{1}{2V_{\omega}}\sum^{n-3}_{i=1}\left[-\frac{n(i+1)}{n^{2}-1}+\frac{in}{n^{2}-1}+\frac{n^{3}}{n^{2}-1}\right]
\int_{X}\varphi\omega^{i}_{\varphi}\wedge\omega^{n-2-i}\wedge\sqrt{-1}\overline{\partial}\omega\wedge\partial\varphi
\\
&-&\frac{n}{n^{2}-1}\frac{1}{2V_{\omega}}\int_{X}\varphi\omega^{n-2}\wedge\sqrt{-1}\overline{\partial}\omega\wedge\partial\varphi
\\
&+&\frac{n}{n+1}\left(n-1+\frac{n}{n-1}\right)\frac{1}{2V_{\omega}}\int_{X}\varphi\omega^{n-2}_{\varphi}\wedge\sqrt{-1}\overline{\partial}\omega\wedge\partial\varphi
\\
&-&\frac{n(n-2)}{(n^{2}-1)2V_{\omega}}\int_{Z}\varphi\omega^{n-2}_{\varphi}\wedge\sqrt{-1}\partial\omega\wedge\overline{\partial}\varphi
+\frac{n(n-2)}{(n^{2}-1)2V_{\omega}}\int_{X}\varphi\omega^{n-2}_{\varphi}\wedge\sqrt{-1}\overline{\partial}\omega\wedge\partial\varphi
\\
&-&\frac{n^{3}}{n^{2}-1}\frac{1}{2V_{\omega}}\int_{X}\varphi\omega^{n-2}\wedge\sqrt{-1}\partial\omega\wedge\overline{\partial}\varphi
+\frac{n^{3}}{n^{2}-1}\frac{1}{2V_{\omega}}\int_{X}\varphi\omega^{n-2}\wedge\sqrt{-1}\overline{\partial}\omega\wedge\partial\varphi
\end{eqnarray*}
Using the identities
\begin{eqnarray*}
\frac{n(i+1)}{n^{2}-1}-\frac{in}{n^{2}-1}-\frac{n^{3}}{n^{2}-1}&=&\frac{n-n^{3}}{n^{2}-1}
\ \ =-n, \\
-\frac{n}{n+1}\left(n-1+\frac{n}{n-1}\right)-\frac{n(n-2)}{n^{2}-1}&=&
\frac{-n[(n-1)^{2}+n]-n(n-2)}{n^{2}-1} \ \ = \ \ -n,
\end{eqnarray*}
the above expression can be simplified as
\begin{eqnarray*}
& & \mathcal{J}^{{\rm AY}}_{\omega}(\varphi) \ \ = \ \
-\mathcal{L}^{{\rm
M}}_{\omega}(\varphi)+\frac{1}{V_{\omega}}\int_{X}\varphi\omega^{n}
\\
&-&\frac{n}{2V_{\omega}}\sum^{n-2}_{i=1}\int_{X}\varphi\omega^{i}_{\varphi}\wedge\omega^{n-2-i}
\wedge\sqrt{-1}\partial\omega\wedge\overline{\partial}\varphi-\frac{n}{2V_{\omega}}\int_{X}\varphi\omega^{n-2}\wedge\sqrt{-1}\partial\omega\wedge\overline{\partial}\varphi\nonumber\\
&+&\frac{n}{2V_{\omega}}\sum^{n-2}_{i=1}\int_{X}\varphi\omega^{i}_{\varphi}\wedge\omega^{n-2-i}
\wedge\sqrt{-1}\overline{\partial}\omega\wedge\partial\varphi+\frac{n}{2V_{\omega}}\int_{X}\varphi\omega^{n-2}\wedge\sqrt{-1}\overline{\partial}\omega\wedge\partial\varphi.\nonumber
\end{eqnarray*}

In summary,

\begin{proposition} \label{C.2} If $n\geq3$, one has
\begin{eqnarray*}
& & \mathcal{J}^{{\rm AY}}_{\omega}(\varphi) \ \ = \ \
-\mathcal{L}^{{\rm
M}}_{\omega}(\varphi)+\frac{1}{V_{\omega}}\int_{X}\varphi\omega^{n}
\\
&-&\frac{n}{2V_{\omega}}\sum^{n-2}_{i=1}\int_{X}\varphi\omega^{i}_{\varphi}\wedge\omega^{n-2-i}
\wedge\sqrt{-1}\partial\omega\wedge\overline{\partial}\varphi-\frac{n}{2V_{\omega}}\int_{X}\varphi\omega^{n-2}\wedge\sqrt{-1}\partial\omega\wedge\overline{\partial}\varphi\nonumber\\
&+&\frac{n}{2V_{\omega}}\sum^{n-2}_{i=1}\int_{X}\varphi\omega^{i}_{\varphi}\wedge\omega^{n-2-i}
\wedge\sqrt{-1}\overline{\partial}\omega\wedge\partial\varphi+\frac{n}{2V_{\omega}}\int_{X}\varphi\omega^{n-2}\wedge\sqrt{-1}\overline{\partial}\omega\wedge\partial\varphi.\nonumber
\end{eqnarray*}
\end{proposition}

\bibliographystyle{amsplain}

\begin{thebibliography}{10}

\bibitem{CT} Chen, X. X., Tian, G., \textit{Ricci flow on
K\"ahler-Einstein surfaces}, Invent. Math. {\bf 147}(2002), no.3,
487--544.

\bibitem{L1} Li, Y., \textit{Mabuchi and Aubin-Yau functionals over
complex surfaces}, preprint, arXiv:1002.3411, 2010.

\bibitem{L2} Li, Y., \textit{Mabuchi and Aubin-Yau functionals over
complex three-folds}, preprint, arXiv:1003.5307, 2010.

\bibitem{M} Mabuchi, T., \textit{K-energy maps integrating Futaki
invatiants}, Tohoku Math. Journ., {\bf 38}(1986), 575--593.

\bibitem{PS} Phong, D.H., Sturm, J., \textit{Lectures on stability and
constant scalar curvature}, preprint, arXiv: 0801.4179v2, 2008.


\bibitem{Y2} Yau, S.-T., \textit{Review of geometry and analysis},
Kodaira's issue. Asian J. Math. {\bf 4}(2000), no.1, 235--278

\end{thebibliography}

\end{document}